\documentclass[11pt]{article}

\usepackage[compress]{natbib}
\usepackage[margin=1in]{geometry}
\usepackage{setspace}
\usepackage{amsmath,amssymb,amsthm,amsfonts,mathtools}
\usepackage{booktabs}
\usepackage{nicefrac}
\usepackage{graphicx}
\usepackage{multirow}
\usepackage{algorithm}
\usepackage{algpseudocode}
\usepackage[shortlabels]{enumitem}
\usepackage[table,dvipsnames]{xcolor}
\usepackage{microtype}

\usepackage[T1]{fontenc}
\usepackage{lmodern}

\usepackage[scaled]{helvet}
\usepackage{courier}
\normalfont
\usepackage[T1]{fontenc}
\usepackage{url}
\usepackage{thmtools}
\usepackage{hyperref}
\hypersetup{
  colorlinks=true,
  citecolor=MidnightBlue,
  linkcolor=red!80!black,
  urlcolor=MidnightBlue
}

\allowdisplaybreaks[4]
\newtheoremstyle{jmlr}
  {9pt plus 3pt minus 5pt} 
  {9pt plus 3pt minus 5pt} 
  {\itshape}               
  {0pt}                    
  {\bfseries}              
  {.}                      
  {.5em}                   
  {}                       

\theoremstyle{jmlr}

\newtheorem{theorem}{Theorem}
\newtheorem{lemma}[theorem]{Lemma}
\newtheorem{proposition}[theorem]{Proposition}

\newtheorem{definition}{Definition}

\newtheorem{assumption}{Assumption}

\newenvironment{keywords}
  {\par\vspace{0.5em}\noindent\textbf{Keywords:}\ }
  {\par}

\newcommand{\equalcontrib}{\textsuperscript{*}}
\newcommand{\corresponding}{\textsuperscript{\dag}}

\begin{document}

\title{SGHA: A Single--Loop Fully First-Order Algorithm for\\
Nonconvex--Strongly-Convex Bilevel Optimization}

\date{\today}
\author{%
Zhihao Gu\textsuperscript{1}\equalcontrib \qquad
Qilong Wu\textsuperscript{2}\equalcontrib \qquad
Junchi Yang\textsuperscript{2}\corresponding\\[0.75em]
\textsuperscript{1}Department of Industrial and Manufacturing Engineering\\
The Pennsylvania State University\\[0.25em]
\textsuperscript{2}School of Data Science\\
The Chinese University of Hong Kong, Shenzhen\\[0.6em]
\small\texttt{zbg5155@psu.edu}\quad
\small\texttt{qilongwu@link.cuhk.edu.cn}\\
\small\texttt{yangjunchi@cuhk.edu.cn}
}

\maketitle

\renewcommand{\thefootnote}{\fnsymbol{footnote}}
\footnotetext[1]{Equal contributions.}
\footnotetext[2]{The corresponding author. Also affiliated with Shenzhen Research Institute of Big Data and 
Shenzhen International Center for Industrial and Applied Mathematics.}
\renewcommand{\thefootnote}{\arabic{footnote}}
\begin{abstract}

In this work, we study the oracle complexity of finding an $\epsilon$-stationary point for nonconvex--strongly-convex (NC--SC) bilevel optimization using only first-order oracles. Existing methods achieving the best-known complexity guarantees typically rely on double-loop, penalty-based procedures. We propose a novel \emph{single-loop} algorithm based on a constrained reformulation in which lower-level stationarity is imposed as a constraint. Specifically, we construct a regularized Lagrangian by introducing a quadratic regularizer and restricting the dual variable to a bounded domain, and then apply Smoothed Gradient Descent Ascent \citep{zhang2020single}, with Hessian-vector products approximated via finite differences of gradients. We refer to the resulting deterministic and stochastic algorithms as SGHA and Stoc--SGHA, respectively. In the deterministic setting, SGHA achieves an oracle complexity of $\mathcal{O}(\bar{\kappa}_y^{5}\epsilon^{-2})$, where $\bar{\kappa}_y$ denotes the relevant condition number. In the stochastic setting, Stoc--SGHA achieves an oracle complexity of $\mathcal{O}\left(\bar{\kappa}_y^{17}\epsilon^{-6}\rho^{-3}\right)$ with probability at least $1-\rho$ for any $\rho\in(0,1)$, and an oracle complexity of $\mathcal{O}\left(\bar{\kappa}_y^{17}\epsilon^{-6}\right)$ in expectation under an additional bounded-iterate assumption. Moreover, under an additional stochastic smoothness assumption imposed only on the lower-level objective, the stochastic oracle complexity of Stoc--SGHA improves to $\mathcal{O}\left(\bar{\kappa}_y^{11}\epsilon^{-4}\rho^{-2}\right)$ with high probability and $\mathcal{O}\left(\bar{\kappa}_y^{11}\epsilon^{-4}\right)$ in expectation, matching the $\epsilon$-dependence of the lower bounds. 

\end{abstract}


\section{Introduction}
\label{sec:intro}

Bilevel optimization \citep{colson2007overview} is a fundamental framework for modeling hierarchical decision-making problems with two nested levels. It originates from
the two-player general-sum Stackelberg game \citep{von2010market}, and reflects the sequential nature of the decision process
of two players. Recently, it has found broad applications in machine learning and related areas, including meta-learning \citep{rajeswaran2019meta}, hyperparameter optimization \citep{franceschi2018bilevel,bao2021stability}, model selection \citep{giovannelli2021bilevel,kunapuli2008bilevel}, adversarial learning \citep{goodfellow2020generative}, and reinforcement learning \citep{hong2023two,konda1999actor,zeng2024two}. In its canonical form, bilevel optimization can be written as the following problem:
\begin{equation}\tag{P}\label{Bilevel Formulation}
\min_{x \in \mathbb{R}^{d_x}} F(x) = f(x, y^*(x)),
\quad \text{s.t.}\quad
y^*(x) \in \arg\min_{y \in \mathbb{R}^{d_y}} g(x,y).
\end{equation}
The hyper-objective \(F\) depends on the upper-level variable \(x\) both directly through \(f\) and indirectly through the lower-level solution \(y^*(x)\).

We consider the smooth nonconvex--strongly-convex (NC--SC) setting \citep{ghadimi2018approximation,hong2023two,ji2021bilevel}, in which the lower-level objective $g(x,y)$ is strongly convex in $y$, while the upper-level objective $f(x,y)$ may be nonconvex. Under this assumption, the lower-level solution $y^\star(x)$ is unique, and the resulting hyperobjective $F$ is well defined and smooth. By the implicit function theorem \citep{krantz2002implicit}, the gradient of the hyperobjective can be characterized as follows:
\begin{equation}
    \nabla F(x)
=
\nabla_x f(x,y^*(x))
-
\nabla^2_{xy} g(x,y^*(x))
\left[\nabla^2_{yy} g(x,y^*(x))\right]^{-1}
\nabla_y f(x,y^*(x)).
\label{eq: implicit-differentiation formula}
\end{equation}
This motivates methods that estimate $\nabla y^*(x)$ and substitute it into the gradient expression~\citep{gould2016differentiating, domke2012generic, ghadimi2018approximation, dagreou2022framework}. However, they typically rely on the Hessian of $g$.

To avoid costly Hessian--vector products (HVPs), fully first-order methods have recently received increasing attention in bilevel optimization \citep{kwon2023fully, liu2022bome}. A prominent line of work adopts the penalty formulation
\[
F_{\lambda}(x,y)
=
f(x,y)
+
\lambda\bigl[g(x,y)-g^\star(x)\bigr],
\quad \text{with }
g^\star(x)
=
\min_{y} g(x,y)
=
g\bigl(x,y^\star(x)\bigr),
\]
where $\lambda>0$ is the penalty parameter and $g(x,y)-g^\star(x)$ measures lower-level suboptimality. \citet{chen2025near} and \citet{chen2026conditionnumberdependencybilevel} analyze this penalty-based approach with the choice $\lambda=\Theta(\epsilon^{-1})$. In particular, F$^2$BA and F$^2$BA$^+$ achieve a nearly optimal fully first-order oracle complexity of $\widetilde{\mathcal{O}}(\epsilon^{-2})$ in the deterministic NC--SC setting, while F$^2$BSA and F$^3$BSA attain the best-known complexity of $\widetilde{\mathcal{O}}(\epsilon^{-6})$ in the stochastic NC--SC setting.

However, these complexity guarantees rely crucially on double-loop procedures. Specifically, the algorithms introduce two auxiliary variables, $z$ and $y$, and solve the corresponding inner subproblems to prescribed accuracies in order to approximate $g^\star(x)$ and $\arg\min_y F_{\lambda}(x,y)$. A key observation in \citet{chen2025near} is that, when the inner iterates remain sufficiently close to their exact minimizers, the outer update in $x$ can use a relatively large stepsize because the value function $F_{\lambda}^\star(x) := \min_y F_{\lambda}(x, y)$ has a smoothness constant of order $\Theta(1)$, independent of the large penalty parameter $\lambda$. In contrast, if the inner loops are removed directly, the tracking errors in $y$ and $z$ evolve simultaneously with the upper-level variable $x$, while the large penalty parameter deteriorates the conditioning of the penalized problem. As a result, algorithms and analyses that cannot exploit the $\lambda$-independent smoothness of $F_\lambda^\star$ typically yield worse complexity guarantees \citep{liu2022bome,shen2023penalty,kwon2023fully}. Therefore, a direct single-loop implementation of the penalty-based approach does not retain the best-known dependence on $\epsilon$.

\begin{quote}
\centering
\itshape
Can a single-loop, fully first-order algorithm match the best
\(\epsilon\)-dependence \\
achieved by double-loop penalty methods?
\end{quote}

We answer this question by proposing a novel single-loop algorithm, \textbf{SGHA} (\textbf{S}moothed \textbf{G}DA with \textbf{H}essian \textbf{A}pproximation), together with its stochastic counterpart, \textbf{Stoc-SGHA}. Our approach builds on the Smoothed Gradient Descent Ascent (GDA) framework \citep{zhang2020single} and adapts it to NC--SC bilevel optimization through a regularized Lagrangian reformulation, combining a first-order approximation for Hessian--vector products.

The starting point of our approach is to reformulate
Problem~\eqref{Bilevel Formulation} as an equality-constrained problem. Since
the lower-level objective \(g(x,\cdot)\) is strongly convex, the lower-level
solution is uniquely characterized by the first-order optimality condition
$\nabla_y g(x,y^*(x))=0$.
Therefore, Problem~\eqref{Bilevel Formulation} can be equivalently written as
\begin{equation}\tag{R}\label{Reformulated Bilevel Optimization}
    \min_{x \in \mathbb{R}^{d_x},\, y \in \mathbb{R}^{d_y}} f(x, y),
    \quad \text{s.t.}\quad
    \nabla_y g(x,y) = 0.
\end{equation}
We then introduce the associated Lagrangian function
\begin{equation}\tag{L}
    \mathcal{L}(x,y,\lambda)
    :=
    f(x,y) + \lambda^{\top}\nabla_y g(x,y),
    \label{eq: NC-L Lagraingian}
\end{equation}
where \(\lambda\in\mathbb R^{d_y}\) is the multiplier.

The Lagrangian \eqref{eq: NC-L Lagraingian} is generally nonconvex in the primal variables $(x,y)$ and affine in the dual multiplier $\lambda$. Although the equality-constrained problem can be formally expressed as
$\min_{x,y}\max_{\lambda} \mathcal{L}(x,y,\lambda)$,
the inner maximization is unbounded whenever $\nabla_y g(x,y)\neq 0$, since $\lambda$ is unconstrained. Consequently, the resulting formulation is not a finite-valued smooth minimax problem, and many standard minimax algorithm analyses cannot be applied directly. Moreover, our goal is to exploit the particular structure of this Lagrangian to obtain sharper complexity guarantees.

To this end, we apply Smoothed GDA with three key modifications. First, we restrict the multiplier $\lambda$ to a sufficiently large bounded domain. Second, we introduce a small concave quadratic regularizer in $\lambda$, which renders the inner maximization strongly concave with well-controlled conditioning. Third, because the primal gradients of $\mathcal{L}$ involve Hessian--vector products of $g$, we approximate these terms using finite differences of first-order gradients. These modifications yield a well-conditioned minimax surrogate amenable to Smoothed GDA analysis while preserving the fully first-order nature of SGHA.

Our main complexity results are summarized in Table~\ref{tab:oracle_complexity}. Here, $\kappa_y$ and $\bar{\kappa}_y$ denote the condition numbers defined in Definition~\ref{Definition: kappa}. In the deterministic setting, Theorem~\ref{Theorem: Deterministic Sample Complexity} shows that SGHA (Algorithm~\ref{SGHA}) finds an $\epsilon$-stationary point of the NC--SC bilevel problem with oracle complexity $\mathcal{O}(\bar{\kappa}_y^{5}\epsilon^{-2})$. In the stochastic setting, Theorems~\ref{Theorem: Stochastic Sample Complexity HP} and~\ref{Theorem: Stochastic Sample Complexity Expectation} show that Stoc-SGHA (Algorithm~\ref{Stoc-SGHA}) achieves oracle complexity $\mathcal{O}\left(\bar{\kappa}_y^{17}\epsilon^{-6}\rho^{-3}\right)$ with probability at least $1-\rho$ for any $\rho\in(0,1)$, and $\mathcal{O}\left(\bar{\kappa}_y^{17}\epsilon^{-6}\right)$ in expectation under an additional bounded-iterate assumption. These bounds remove the logarithmic factors appearing in the best-known complexities of \citet{chen2025near}. Moreover, when the lower-level objective $g$ satisfies the stochastic smoothness condition in Assumption~\ref{Assumption: Stochastic Smoothness of g}, Theorems~\ref{Theorem: Stochastic Sample Complexity HP with g Smoothness} and~\ref{Theorem: Stochastic Sample Complexity Expectation with g Smoothness} improve these bounds to $\mathcal{O}\left(\bar{\kappa}_y^{11}\epsilon^{-4}\rho^{-2}\right)$ with probability at least $1-\rho$, and $\mathcal{O}\left(\bar{\kappa}_y^{11}\epsilon^{-4}\right)$ in expectation. They match the lower bound in \citep{arjevani2023lower} in the dependence on $\epsilon$.

\begin{table}[t]
\centering
\caption{Oracle complexity for finding an \(\epsilon\)-stationary point in deterministic and stochastic NC--SC bilevel optimization. Shaded cells denote our methods.
}
\label{tab:oracle_complexity}

\setlength{\tabcolsep}{6pt}
\renewcommand{\arraystretch}{1.15}
\resizebox{\linewidth}{!}{
\begin{tabular}{llccl}
\toprule
Problem Class & Algorithm & Oracle & \# Loops & Oracle Complexity \\
\midrule

\multirow{7}{*}{Det. NC--SC BiO}
& F$^2$SA \citep{kwon2023fully}
& $\mathcal{FO}$
& single--loop
& $\widetilde{\mathcal{O}}(\bar{\kappa}_y^{7}\epsilon^{-3})$
\\

& MEHA \citep{liu2024moreau}
& $\mathcal{FO}$
& single--loop
& $\mathcal{O}(\mathrm{poly}(\bar{\kappa}_y)\epsilon^{-4})$
\\

& V-PBGD \citep{shen2023penalty}
& $\mathcal{FO}$
& double--loop
& $\mathcal{O}(\mathrm{poly}(\bar{\kappa}_y)\epsilon^{-3})$
\\

& F$^2$BA \citep{chen2025near}
& $\mathcal{FO}$
& double--loop
& $\widetilde{\mathcal{O}}(\bar{\kappa}_y^{4}\epsilon^{-2})$
\\

& F$^2$BA$^+$ \citep{chen2026conditionnumberdependencybilevel}
& $\mathcal{FO}$
& double--loop
& $\widetilde{\mathcal{O}}(\bar{\kappa}_y^{3.5}\epsilon^{-2})$
\\

& \cellcolor{gray!12}\textbf{SGHA (Theorem \ref{Theorem: Deterministic Sample Complexity})}
& \cellcolor{gray!12}$\mathcal{FO}$
& \cellcolor{gray!12}single--loop
& \cellcolor{gray!12}$\mathcal{O}(\bar{\kappa}_y^{5}\epsilon^{-2})$
\\


& Lower Bound \citep{chen2026conditionnumberdependencybilevel}
& $\mathcal{FO}$
& --
& $\Omega(\kappa_y^{2.5}\epsilon^{-2})$
\\

\midrule

\multirow{6}{*}{Stoc. NC--SC BiO}
& F$^2$SA \citep{kwon2023fully}
& $\mathcal{SFO}$
& double--loop
& $\widetilde{\mathcal{O}}(\mathrm{poly}(\bar{\kappa}_y)\epsilon^{-7})$
\\

& F$^2$BSA \citep{chen2025near}
& $\mathcal{SFO}$
& double--loop
& $\widetilde{\mathcal{O}}(\bar{\kappa}_y^{12}\epsilon^{-6})$
\\

& F$^3$BSA \citep{chen2025near}
& $\mathcal{SFO}$
& double--loop
& $\widetilde{\mathcal{O}}(\bar{\kappa}_y^{11}\epsilon^{-6})$
\\

& \cellcolor{gray!12}\textbf{Stoc-SGHA (Theorem \ref{Theorem: Stochastic Sample Complexity HP}/Theorem \ref{Theorem: Stochastic Sample Complexity Expectation})}
& \cellcolor{gray!12}$\mathcal{SFO}$
& \cellcolor{gray!12}single--loop
& \cellcolor{gray!12}$\mathcal{O}\left(
\bar{\kappa}_y^{17}\epsilon^{-6}\rho^{-3}
\right)$/$\mathcal{O}\left(
\bar{\kappa}_y^{17}\epsilon^{-6}
\right)$
\\

& \cellcolor{gray!12}\textbf{Stoc-SGHA (Theorem \ref{Theorem: Stochastic Sample Complexity HP with g Smoothness}/Theorem \ref{Theorem: Stochastic Sample Complexity Expectation with g Smoothness})}
& \cellcolor{gray!12}$\mathcal{SFO}$ + Assumption \ref{Assumption: Stochastic Smoothness of g}
& \cellcolor{gray!12}single--loop
& \cellcolor{gray!12}$\mathcal{O}\left(
\bar{\kappa}_y^{11}\epsilon^{-4}\rho^{-2}
\right)$/$\mathcal{O}\left(
\bar{\kappa}_y^{11}\epsilon^{-4}
\right)$
\\



& Lower Bound \citep{chen2026conditionnumberdependencybilevel}
& $\mathcal{FO}$ for \(f\) + $\mathcal{SFO}$ for \(g\)
& --
& $\Omega(\kappa_y^{4.5}\epsilon^{-4})$
\\

\bottomrule
\end{tabular}
}

\vspace{2mm}
\begin{minipage}{0.98\linewidth}
\footnotesize
\textbf{Table Notes.}
Det. and Stoc. stand for ``deterministic" and ``stochastic", respectively.
NC--SC BiO stands for nonconvex--strongly-convex bilevel optimization. \(\mathcal{FO}\) denotes an exact first-order oracle. \(\mathcal{SFO}\) denotes a stochastic first-order oracle. The analysis of F$^2$SA covers both single-loop and nested-loop variants with the same complexity rate. The \(\Omega(\kappa_y^{4.5}\epsilon^{-4})\) lower bound is proved for
zero-respecting \(\mathcal{SFO}\) algorithms on a bounded constrained nonconvex--quadratic problem.
The displayed condition-number rates treat the initial Lyapunov
gap $\Delta_V$ and the remaining problem-dependent variance and
bounded-gradient constants as fixed; the theorem statements retain their
explicit dependence on these quantities.
\end{minipage}
\end{table}

\section{Related Work}

We organize the related work around two themes: bilevel optimization and minimax optimization. For bilevel optimization, we review the best-known deterministic and stochastic first-order oracle complexities for NC--SC problems. We then discuss minimax optimization, motivated by the nonconvex--concave (linear) Lagrangian formulation \eqref{eq: NC-L Lagraingian}.

\paragraph{Bilevel Optimization with Hessian}
The implicit-differentiation formula in \eqref{eq: implicit-differentiation formula} shows that computing the
hypergradient requires estimating the sensitivity
\(\nabla y^*(x)\), or equivalently applying the inverse lower-level
Hessian to \(\nabla_y f\).  One line of work uses
\textbf{I}terative \textbf{D}ifferentiation (ITD)
\citep{gould2016differentiating,franceschi2017forward,
shaban2019truncated,bolte2021nonsmooth}, which approximates
\(y^*(x)\) by a finite lower-level trajectory and differentiates
through its iterations.  Another line, known as
\textbf{A}pproximate \textbf{I}mplicit \textbf{D}ifferentiation (AID)
\citep{domke2012generic,ghadimi2018approximation,
pedregosa2016hyperparameter,franceschi2018bilevel,
grazzi2020iteration,ji2021bilevel,dagreou2022framework,
arbel2021amortized}, approximates the inverse-Hessian action by
solving the associated linear system.

These ideas also underlie several single-loop stochastic NC--SC
bilevel algorithms.  TTSA \citep{hong2023two}, SUSTAIN
\citep{khanduri2021near}, and STABLE \citep{chen2022single} construct
implicit-differentiation-based hypergradient estimators using
Hessian-vector products, Jacobian-vector products, or related
second-order information.  SOBA/SABA \citep{dagreou2022framework}
and MA-SOBA \citep{chen2024optimal} avoid explicit Hessian inversion
by recursively tracking the solution of the linear system, but still
require second-order oracle information and an auxiliary
linear-system variable.  Consequently, these methods fall outside
the fully first-order oracle model considered in this paper.

\paragraph{Fully First-order Bilevel Optimization.}
Methods based on the hypergradient typically require Hessian information, which has motivated the development of fully first-order approaches \citep{liu2022bome,shen2023penalty,kwon2023fully}. Most of these methods are double-loop and penalty-based: they incorporate lower-level optimality into a penalized upper-level objective whose gradients can be evaluated using only first-order information. In particular, \citet{kwon2023fully} proposed F$^2$SA, which achieves a deterministic oracle complexity of $\widetilde{\mathcal{O}}(\bar{\kappa}_y^{7}\epsilon^{-3})$ and a stochastic oracle complexity of $\widetilde{\mathcal{O}}(\mathrm{poly}(\bar{\kappa}_y)\epsilon^{-7})$ for NC--SC bilevel optimization using an increasing penalty parameter. By instead fixing the penalty parameter at $\lambda=\Theta(\epsilon^{-1})$, \citet{chen2026conditionnumberdependencybilevel} improved the deterministic complexity to $\widetilde{\mathcal{O}}(\bar{\kappa}_y^{3.5}\epsilon^{-2})$, while the stochastic complexity was improved to $\widetilde{\mathcal{O}}(\mathrm{poly}(\bar{\kappa}_y)\epsilon^{-6})$ \citep{kwon2024complexity,chen2025near}. Without additional assumptions, no known fully first-order method improves these dependencies on $\epsilon$.

On the lower-bound side, \citet{chen2026conditionnumberdependencybilevel} establish $\Omega(\kappa_y^{2.5}\epsilon^{-2})$ in the deterministic setting and $\Omega(\kappa_y^{4.5}\epsilon^{-4})$ in the stochastic setting for fully first-order methods. These results show that the deterministic complexity of F$^2$BA/F$^2$BA$^+$ is nearly optimal in its dependence on $\epsilon$. In contrast, the best-known stochastic upper bound still exhibits an extra factor of $\epsilon^{-2}$ relative to the $\Omega(\epsilon^{-4})$ lower bounds. We note that \citet{kwon2024complexity} also establish an
\(\Omega(\epsilon^{-4})\) lower bound under the stochastic smoothness assumption on the lower-level objective, using a \(y^\star\)-aware oracle, which additionally returns an estimate \(\hat y\) satisfying a prescribed accuracy guarantee relative to the exact lower-level solution \(y^\star(x)\).


Most of the methods discussed above are double-loop. While the analysis of \citet{kwon2023fully} permits a single-loop variant by setting the inner-loop length to one, it has strictly worse complexity guarantees. The best-known bounds of \citet{chen2025near}, by contrast, rely essentially on inner-loop updates, as discussed in Section~\ref{sec:intro}. A closely related single-loop method is FdeHBO \citep{yang2023achieving}, which avoids iteratively solving linear-systems and uses finite-difference approximations for Hessian--vector products. However, its guarantees require variance reduction and stronger assumptions, including stochastic smoothness in the upper-level function, stochastic Hessian, and higher-order smoothness. 
Other methods, such as F$^3$SA \citep{kwon2023fully}, also operate under similar stronger assumptions. We therefore omit these methods from Table~\ref{tab:oracle_complexity}.

\paragraph{Minimax optimization.}
Minimax optimization problems take the form
\[
\min_{\omega \in \mathcal{W}}\max_{\lambda \in \Lambda} f(\omega,\lambda),
\]
where $\mathcal{W}\subseteq\mathbb{R}^{d_\omega}$ and $\Lambda\subseteq\mathbb{R}^{d_\lambda}$ are nonempty closed convex sets. Motivated by the structure of \eqref{eq: NC-L Lagraingian}, we focus on the nonconvex--concave (NC--C) setting, where $f:\mathbb{R}^{d_\omega}\times\mathbb{R}^{d_\lambda}\to\mathbb{R}$ is nonconvex in $\omega$ and concave in $\lambda$.

There is a vast literature on nonconvex--concave (NC--C) minimax optimization
\citep{rafique2022weakly,nouiehed2019solving,kong2021accelerated,bot2023alternating}.
In the deterministic setting, the best-known gradient complexity is
$\widetilde{\mathcal O}(\epsilon^{-2.5})$ under the $f$-stationarity measure
\citep{lin2020near}, and $\widetilde{\mathcal O}(\epsilon^{-3})$ under
stationarity of the Moreau envelope of the value function \citep{yang2020catalyst}. In the stochastic
setting, SAPD+ achieves the best-known stochastic oracle complexity of
$\mathcal O(\epsilon^{-6})$ \citep{zhang2022sapdplus}. However, most of
these results assume a compact dual domain $\Lambda$, and thus do not apply
directly to our original unbounded Lagrangian formulation. Moreover, by
exploiting the special structure of the bilevel problem, our method achieves
sharper complexity guarantees than those available for generic NC--C minimax
optimization.


In contrast, existing single-loop algorithms generally have worse complexity than the best-known rates discussed above~\citep{lin2025two,xu2023unified,bot2023alternating}. Our algorithm builds on Smoothed GDA~\citep{zhang2020single,yang2022faster}, which introduces a Moreau-envelope-type smoothing mechanism via an auxiliary primal sequence. In the deterministic setting, \citet{zhang2020single} establish a complexity of \(\mathcal O(\epsilon^{-4})\) for general NC-C minimax problems, which improves to \(\mathcal O(\epsilon^{-2})\) when the objective is linear in \(\lambda\) and additional regularity conditions hold, including strict complementarity and bounded level sets of \(\max_{\lambda\in\Lambda} f(\cdot,\lambda)\). Their analysis does not directly apply to our setting, as we seek an \(\mathcal O(\epsilon^{-2})\) complexity without these additional assumptions. In both deterministic and stochastic settings, \citet{yang2022faster} obtain complexities of \(\mathcal O(\kappa\epsilon^{-2})\) and \(\mathcal O(\kappa^2\epsilon^{-4})\), respectively, under a Polyak--\L{}ojasiewicz condition, where \(\kappa\) denotes the condition number. These results do not directly apply to our setting: imposing strong concavity via a quadratic regularizer of order \(\Theta(\epsilon\|\lambda\|^2)\) will lead to a worse \(\epsilon\)-dependence than ours. Moreover, the gradients of \eqref{eq: NC-L Lagraingian} requires HVPs, which must be approximated in our first-order framework.



\section{Preliminaries and Technical Background}

\paragraph{Notations.}
For a vector \(v\) and a matrix \(A\), we use \(\|v\|\) and \(\|A\|\) to denote the Euclidean norm of \(v\) and the spectral norm of \(A\), respectively. For a closed convex set \(\mathcal C\), we denote by \(\mathcal P_{\mathcal C}\) its Euclidean projection. For a differentiable function \(h\), we write \(\nabla h\) for its gradient, and use \(\nabla_x h\) and \(\nabla_y h\) to denote the partial gradients with respect to \(x\) and \(y\). Similarly, \(\nabla^2_{xy}h\) and \(\nabla^2_{yy}h\) denote the corresponding Hessian blocks. Throughout the paper, \(\mathbb E[\cdot]\) denotes expectation with respect to all randomness in the stochastic oracle and the algorithm, and \(\mathbb E_k[\cdot]\) denotes conditional expectation given the history up to iteration \(k\). \(a\asymp b\) means that \(a/b\) is bounded above and below by
positive constants. The notation \(\mathcal O(\cdot)\), \(\Omega(\cdot)\), and \(\Theta(\cdot)\) hides absolute constants independent of the target accuracy, condition number, and failure probability, while \(\tilde{\mathcal O}(\cdot)\), \(\tilde{\Omega}(\cdot)\), and \(\tilde{\Theta}(\cdot)\) additionally hide logarithmic factors.

\subsection{Basic Assumptions and Definitions}
We consider the nonconvex--strongly--convex (NC--SC) bilevel optimization problem
\begin{equation}\tag{P}
\min_{x \in \mathbb{R}^{d_x}} F(x) = f(x, y^*(x)),
\quad \text{s.t.}\quad
y^*(x) \in \arg\min_{y \in \mathbb{R}^{d_y}} g(x,y).
\end{equation}
We first state the basic assumptions used throughout the paper. These assumptions are standard in the literature and, to the best of our knowledge, are among the weakest commonly imposed for smooth NC--SC bilevel optimization \citep{chen2026conditionnumberdependencybilevel,chen2025near,kwon2023penalty,kwon2023fully}.
\begin{assumption}\label{Assumption: Bilevel Optimization}
    We have the following assumptions:
    \begin{enumerate}
    \item $\|\nabla_y f(x,y)\| \leq l_{f,0}$ for all $x$ and $y$.
    \item $g(x,y)$ is $\mu_g$-strongly convex w.r.t. $y$ for every $x$.
        \item $f$ and $g$ are $l_{f,1}$ and $l_{g,1}$-smooth, respectively.
        \item The Hessians of $f$ and $g$ are $l_{f,2}$-- and $l_{g,2}$--Lipschitz, respectively.
        \item For all $x$ and $y$, $f(x,y)$ is bounded from below by some finite constant $\underline{f} > - \infty$.
\end{enumerate}
\end{assumption}
Next, we introduce the condition-number notation and the stationarity measure used throughout the paper.
\begin{definition}\label{Definition: kappa}
    Under Assumption \ref{Assumption: Bilevel Optimization}, we define the lower-level condition number \(\kappa_y:=l_{g,1}/\mu_g\), and the global condition number $\bar{\kappa}_y := \bar L/\mu_g$ with $\bar L := \max\{l_{f,0},l_{f,1},l_{g,1},l_{f,2},l_{g,2}\}$.
\end{definition}
Since $l_{g,1}\ge\mu_g$, these definitions imply
$\bar\kappa_y\ge\kappa_y\ge1$.
\begin{definition}\label{Definition: Stationary Point}
    Given a differentiable function $\varphi(x) : \mathbb{R}^d \to \mathbb{R}$, we call $x$ an $\epsilon$--first--order stationary point of $\varphi(x)$ if $\|\nabla\varphi(x)\| \leq \epsilon$.
\end{definition}
\subsection{Equality--constrained Reformulation}

Since \(g(x,\cdot)\) is differentiable and strongly convex for every fixed
\(x\), the lower-level minimizer $y^*(x)$ is unique. Moreover, the first-order
optimality condition is both necessary and sufficient, and hence
\[
    \nabla_y g(x,y^\star(x))=0 .
\]
Therefore, the graph of the lower-level solution map \(x\mapsto y^\star(x)\)
is exactly the feasible set described by the equality constraint
\(\nabla_y g(x,y^\star(x))=0\). Equivalently, any feasible pair \((x,y)\) of this reformulated problem satisfies \(y=y^\star(x)\), and hence
\(f(x,y)=F(x)\). Thus, it preserves the
objective values and minimizers of the original bilevel formulation. Consequently, the NC--SC bilevel problem can be equivalently
written as
\[ \tag{R}
    \min_{x \in \mathbb{R}^{d_x},\, y \in \mathbb{R}^{d_y}} f(x,y),
    \qquad
    \text{s.t.}\quad \nabla_y g(x,y)=0.
\]
The Lagrangian associated with this equality-constrained bilevel problem is
\[ \tag{L}
    \mathcal{L}(x,y,\lambda)
    :=
    f(x,y)+\lambda^\top \nabla_y g(x,y),
    \qquad
    \lambda\in\mathbb R^{d_y}.
\]

The next lemma shows that the equality constraint in the equality-constrained
bilevel problem is regular. In other words, lower-level strong convexity implies
the LICQ condition for \(\nabla_y g(x,y)=0\). This regularity ensures that local solutions of the \eqref{Reformulated Bilevel Optimization} admit KKT multipliers, and hence justifies the Lagrangian stationarity system introduced below.

\begin{lemma}[LICQ for the lower-level optimality constraint]
\label{Lemma: LICQ}
Suppose \(g(x,\cdot)\) is \(\mu_g\)-strongly convex for every \(x\). Then, the Jacobian of the constraint mapping \(\nabla_y g(x,y)\) has full row rank
at every \((x,y)\). In particular, LICQ holds at every feasible point of the
equality-constrained reformulation.
\end{lemma}

We note that several works also impose the lower-level first-order optimality condition as a constraint, either designing algorithms based on the KKT system of the resulting constrained problem \citep{liu2023averaged} or applying penalty methods \citep{kim2020mpec,mehra2021penalty}. Our approach differs in that we construct a regularized Lagrangian reformulation and develop a specific single-loop, fully first-order method for finding an approximate stationary point of \(\mathcal L\).

\subsection{From Lagrangian Stationarity to Bilevel Stationarity}
The following theorem shows that the equality-constrained reformulation allows us to reduce the analysis of stationarity for the hyperobjective \(F\) to that of the Lagrangian \(\mathcal L\).

\begin{theorem}
\label{Theorem: Approximation of Hypergradient}
Suppose Assumption \ref{Assumption: Bilevel Optimization} holds. Then
\[
\|\nabla F(x)\|
\le
\|\nabla_x \mathcal{L}(x, y, \lambda)\|
+
\kappa_y\|\nabla_y \mathcal{L}(x, y, \lambda)\|
+
\frac{L_{\bar F,y}}{\mu_g}\|\nabla_\lambda \mathcal{L}(x, y, \lambda)\|.
\]
with
\[
L_{\bar F,y}
=
l_{f,1}(1+\kappa_y)
+
l_{f,0}l_{g,2}
\left(
\frac{1}{\mu_g}
+
\frac{l_{g,1}}{\mu_g^2}
\right).
\]
\end{theorem}
The following theorem shows that for any pairs of $(x^*,y^*,\lambda^*)$ satisfying the KKT conditions of the equality-constrained bilevel problem, $\lambda^*$ is always bounded by a constant determined by the problem itself.
This provides one motivation for studying stationarity of the Lagrangian associated with the lower-level first-order optimality constraint, rather than penalizing the lower-level value gap as in \citep{kwon2023fully}.

\begin{theorem}[Boundedness of $\lambda^*$]\label{Theorem: Boundedness of lambda star}
Under Assumption \ref{Assumption: Bilevel Optimization}, the optimal Lagrangian multiplier $\lambda^*$ satisfying the KKT conditions is bounded with
\[
\|\lambda^*\| \leq \frac{l_{f,0}}{\mu_g}.
\]
\end{theorem}

\section{Deterministic NC--SC Bilevel Optimization}

In this section, we aim to solve deterministic nonconvex-strongly-convex (NC--SC) bilevel optimization problems. We first introduce SGHA and establish its main complexity guarantees, and then present the key analytical tools together with a proof sketch. The algorithm-specific notation used in the analysis of SGHA is summarized in Table~\ref{Notations}.

\subsection{Algorithm}

We first note that $\mathcal{L}$ is generally nonconvex in the primal variables $x$ and $y$, and affine in the multiplier $\lambda$. Because the maximization over an unconstrained multiplier is unbounded whenever the constraint is violated, \eqref{eq: NC-L Lagraingian} does not define a standard finite-valued minimax problem. Instead, we use it as a Lagrangian stationarity model for the equality-constrained reformulation.

However, Theorem~\ref{Theorem: Boundedness of lambda star} shows that every KKT multiplier of Problem~\eqref{Reformulated Bilevel Optimization} is bounded by a problem-dependent constant. Therefore, we restrict the dual variable to $\Lambda=\{\lambda:\|\lambda\|\le C_\lambda\}$, where $C_\lambda$ is chosen sufficiently large to contain all KKT multipliers of Problem~\eqref{Reformulated Bilevel Optimization}. This bounded domain provides uniform smoothness control of $\mathcal{L}$ along the algorithmic trajectory. We further add a small quadratic regularization term to $\lambda$
with a sufficiently small regularization parameter \(p_\lambda>0\). This term turns the projected affine dual maximization into a strongly concave
one, avoiding boundary-saturating or nonunique multiplier responses, while
only introducing a controllable perturbation to the original Lagrangian stationarity. This leads to the regularized surrogate
\begin{equation}
    \min_{x \in \mathbb{R}^{d_x}, y \in \mathbb{R}^{d_y}}\max_{\lambda \in \Lambda}\ f(x,y) + \lambda^{\top}\nabla_y g(x,y) - \frac{p_{\lambda}}{2}\|\lambda\|^2.
    \label{eq:reformed lagrangian}
\end{equation}

We then apply Smoothed Gradient Descent Ascent (GDA) \citep{zhang2020single}. Specifically, we introduce auxiliary smoothing variables \(\hat{x}\) and \(\hat{y}\), together with quadratic proximal terms on the primal variables using a regularization parameter \(p_w > 0 \). With an appropriate choice of \(p_w\), the resulting surrogate is strongly convex in \(w = (x,y)\) and strongly concave in \(\lambda\). We then perform gradient descent on the primal variables $w$, and projected gradient ascent on the dual variable \(\lambda\), for the problem
\begin{equation}
    \min_{x \in \mathbb{R}^{d_x}, y \in \mathbb{R}^{d_y}}\max_{\lambda \in \Lambda}\ f(x,y) + \lambda^{\top}\nabla_y g(x,y) + \frac{p_w}{2}\|x- \hat{x}\|^2 + \frac{p_w}{2}\|y- \hat{y}\|^2 - \frac{p_{\lambda}}{2}\|\lambda\|^2.
    \label{eq: K minimax}
\end{equation}

Furthermore, we note that the gradients induced by \eqref{eq: K minimax} involve Hessian-vector product terms associated with \(g\), namely $\nabla^2_{xy} g(x,y) \lambda$ and $\nabla^2_{yy} g(x,y)\lambda$. To retain a fully first-order method, we approximate these Hessian-vector products by central finite differences. Specifically, we use
\[
\tilde{\nabla}^2_{xy} g_{\lambda}(x,y)
:=
\frac{
\nabla_x g(x,y+\tau\lambda)-\nabla_x g(x,y-\tau\lambda)
}{
2\tau
} \approx \nabla_{xy}^2 g(x,y)\lambda,
\]
and
\[
\tilde{\nabla}^2_{yy} g_{\lambda}(x,y)
:=
\frac{
\nabla_y g(x,y+\tau\lambda)-\nabla_y g(x,y-\tau\lambda)
}{
2\tau
} \approx \nabla_{yy}^2 g(x,y)\lambda,
\]
where \(\tau>0\) is the finite-difference radius, whose value will be specified later.

The complete procedure is summarized in Algorithm~\ref{SGHA}, which we call \textbf{SGHA} (\textbf{S}moothed \textbf{G}DA with \textbf{H}essian \textbf{A}pproximation). In the convergence analysis, we use a common stepsize for the primal updates, namely \(\gamma_{x,k}=\gamma_{y,k}=:\gamma_{w,k}\), while the dual stepsize \(\gamma_{\lambda,k}\) is chosen separately and may scale differently. After each primal--dual update, the auxiliary variables \(\hat{x}\) and \(\hat{y}\) are updated by an averaging step with parameter \(\beta\in(0,1)\).

\begin{algorithm}[htbp]
\caption{\textbf{SGHA} (\textbf{S}moothed \textbf{G}DA with \textbf{H}essian \textbf{A}pproximation)}
\label{SGHA}
\textbf{Input:} Total iterations $K$, step sizes
$\{\gamma_{x,k}\}_k, \{\gamma_{y,k}\}_k, \{\gamma_{\lambda,k}\}_k$, initialization points $x^0,\hat{x}^0 \in \mathbb{R}^{d_x}, y^0,\hat{y}^0 \in \mathbb{R}^{d_y}, \lambda^0 \in \Lambda$. $\tau > 0$, $p_w, p_{\lambda} > 0$, $0 < \beta < 1$, $\Lambda := \left\{\lambda: \|\lambda\| \leq C_{\lambda}\right\}$ with $C_{\lambda} > 0$.
\begin{algorithmic}[1]
    \For{$k = 0,\ldots,K-1$}
            \State $x^{k+1} \leftarrow x^k - \gamma_{x,k}\left(\nabla_x f(x^k,y^k) + \tilde{\nabla}_{xy}^2 g_{\lambda^k}(x^k,y^k) + p_w(x^k - \hat{x}^k)\right)$
            \State $y^{k+1} \leftarrow y^k - \gamma_{y,k}\left(\nabla_y f(x^k,y^k) + \tilde{\nabla}_{yy}^2 g_{\lambda^k}(x^k,y^k) + p_w(y^k - \hat{y}^k)\right)$
        \State $\hat{x}^{k+1} = \hat{x}^k + \beta(x^{k+1} - \hat{x}^k)$
        \State $\hat{y}^{k+1} = \hat{y}^k + \beta(y^{k+1} - \hat{y}^k)$
        \State $\lambda^{k+1} = \mathcal{P}_{\Lambda}\left(\lambda^k + \gamma_{\lambda,k}\left(\nabla_y g(x^k,y^k) - p_{\lambda}\lambda^k\right)\right)$
    \EndFor
\end{algorithmic}
\end{algorithm}

\subsection{Complexity Guarantees}
We are now ready to state the oracle complexity of Algorithm \ref{SGHA} for solving deterministic nonconvex-strongly-convex bilevel optimization problems.

\begin{theorem}
\label{Theorem: Deterministic Sample Complexity}
Suppose that Assumption~\ref{Assumption: Bilevel Optimization} holds.
Fix any target accuracy $0<\epsilon\le 1$.
Let
$\Lambda:=
\left\{
\lambda:\|\lambda\|\le C_\lambda
\right\}$ with $
C_\lambda:=2\bar\kappa_y$.
Choose $
p_w=16\bar L\bar\kappa_y$
and $p_\lambda
=
\frac{1}{8}
\min\left\{
\mu_g,\frac{\epsilon}{2^{11}(1+\bar L)\bar\kappa_y^4}
\right\}$.
For every $k\ge 0$, use the constant stepsizes
\[
\gamma_{x,k}
=
\gamma_{y,k}
=
\frac{1}{2^8\bar L\bar\kappa_y}, \quad
\gamma_{\lambda,k}
=
\frac{\bar\kappa_y}{2^8\bar L}.
\]
Choose the averaging parameter and finite-difference radius as
\[
\beta
=
\frac{1}{2^{30}\bar\kappa_y^2},
\qquad
\tau
=
\frac{\epsilon}
{2^{24}(1+\bar L)\bar\kappa_y^4}.
\]
Let $\Delta_V:=V_0-\underline f$, where $V_0$ is the initial
Lyapunov value defined below, and run Algorithm~\ref{SGHA} for
\[
K
:=
\max\left\{
1,
\left\lceil
\frac{2^{52}(1+\bar L)^2\Delta_V}{\bar L}
\bar\kappa_y^5\epsilon^{-2}
\right\rceil
\right\}.
\]
Then there exists an index
$k_\star\in\{0,\ldots,K-1\}$ such that $
\left\|\nabla F(x^{k_\star})\right\|\le\epsilon$.

\end{theorem}

The resulting dependence on \(\epsilon\) matches the lower bound \(\Omega(\kappa_y^{2.5}\epsilon^{-2})\) established in \citet{chen2026conditionnumberdependencybilevel}. To the best of our knowledge, SGHA is the first single-loop method for deterministic NC--SC bilevel optimization to achieve the optimal \(\epsilon^{-2}\) dependence. Compared with the best-known double-loop upper bound \(\widetilde{ \mathcal{O}}\left(
\bar\kappa_y^{3.5}\epsilon^{-2}
\right)\) \citep{chen2026conditionnumberdependencybilevel}, our result also removes the logarithmic factor. We note, however, that the dependence on \(\bar{\kappa}_y\) remains worse than that of existing double-loop methods; improving the condition-number dependence of single-loop methods is left for future work.

\subsection{Analysis and Proof Sketch}

\begin{table}[t]
\centering
\caption{Notations in algorithms and analysis}
\begin{tabular}{ll}
\toprule
\textbf{Notation} & \textbf{Definition} \\
\midrule
$w$ & $(x,y)$ \\
$\hat{w}$ & $(\hat{x},\hat{y})$ \\
$\mathcal{L}(w,\lambda)$ & $f(w) + \lambda^{\top}\nabla_y g(w)$ \\
$\gamma_{w}$ & Common stepsize for the update of $x$ and $y$ \\
$\gamma_{\lambda}$ & Stepsize for the update of $\lambda$ \\
$K(w,\hat{w},\lambda)$ & $\mathcal{L}(w,\lambda) + \frac{p_w}{2}\|w-\hat{w}\|^2 - \frac{p_{\lambda}}{2}\|\lambda\|^2$  \\
$\Psi(\hat{w},\lambda)$ & $\min_{w} K(w,\hat{w},\lambda)$ \\
$\Phi(w,\hat{w})$ & $\max_{\lambda \in \Lambda} K(w,\hat{w},\lambda)$  \\
$P(\hat{w})$ & $\min_{w}\max_{\lambda \in \Lambda} K(w,\hat{w},\lambda)$  \\
$V_t$ & $K(w^t,\hat{w}^t,\lambda^t) - 2\Psi(\hat{w}^t,\lambda^t) + 2P(\hat{w}^t)$ \\
$w^*(\hat{w},\lambda)$ & $\arg\min_{w} K(w,\hat{w},\lambda)$ \\
$w^*(\hat{w})$ & $\arg\min_{w} \Phi(w,\hat{w})$ \\
$\hat{\lambda}^*(\hat{w})$ & $\arg\max_{\lambda \in \Lambda} \Psi(\hat{w},\lambda)$  \\
$w^+(\hat{w},\lambda)$ & $w - \gamma_{w} \nabla K(w,\hat{w},\lambda)$\\
$\lambda^+(\hat{w})$ & $\mathcal{P}_{\Lambda}\left(\lambda + \gamma_{\lambda}\left(\nabla_{y}g(w^*(\hat{w},\lambda)) - p_{\lambda}\lambda\right)\right)$ \\
$\epsilon_{\mathcal{L}}$ & Target accuracy for $\mathcal{L}$\\
$\epsilon$ & Target stationarity accuracy for $F$ \\
\bottomrule
\label{Notations}
\end{tabular}
\end{table}

We now provide a proof sketch of Theorem~\ref{Theorem: Deterministic Sample Complexity}. The same analysis framework will later be extended to the stochastic settings. We use the notation summarized in Table~\ref{Notations}. Following \citet{zhang2020single}, we introduce the following Lyapunov function:
\[
V_t = K(w^t,\hat{w}^t,\lambda^t) - 2\Psi(\hat{w}^t,\lambda^t) + 2P(\hat{w}^t).
\]
The following lemma shows that the Lyapunov function is lower bounded.

\begin{lemma}
    \label{Lemma: Lower Bound of Lyapunov Function}
Under Assumption \ref{Assumption: Bilevel Optimization}, the Lyapunov function satisfies
\[
    V_t
    \geq \underline{f}
\]
\end{lemma}
By analyzing the one-step progress bounds for $K(w,\hat{w},\lambda)$, $\Psi(\hat{w},\lambda)$ and $P(\hat{w})$, we obtain the following one-step descent inequality for the Lyapunov function:

\begin{proposition}
\label{Proposition: Recursion of Lyapunov Fucntion}
Suppose that Assumption~\ref{Assumption: Bilevel Optimization} holds,
and let
$
\Lambda:=\{\lambda:\|\lambda\|\le C_\lambda\}$ with $
C_\lambda>\frac{l_{f,0}}{\mu_g}$.
Let $l_{\mathcal L,1}$ be any constant satisfying $l_{\mathcal L,1}
\ge
\sqrt{2}\bigl(l_{f,1}+C_\lambda l_{g,2}\bigr)$ and let the primal regularization parameter $
p_w = 2l_{\mathcal L,1}$. Introduce the auxiliary constants
\[
\gamma_1
:=
\frac{p_w}{p_w-l_{\mathcal L,1}} = 2,
\quad
\gamma_2
:=
\frac{l_{g,1}}{p_w-l_{\mathcal L,1}} = \frac{l_{g,1}}{l_{\mathcal L,1}},
\]
\[
l_{\Psi,1}
:=
p_\lambda+l_{g,1}\gamma_2, \quad
\gamma_{3}
:=
\frac{
l_{\Psi,1}+\gamma_{\lambda,k}^{-1}
}{
\sqrt{
2(p_w-l_{\mathcal L,1})p_\lambda
+
\frac{(p_w-l_{\mathcal L,1})^2\mu_g^2}
{(p_w+l_{\mathcal L,1})^2}
}
}.
\]
Choose positive stepsizes
$\gamma_{w,k}$ and $\gamma_{\lambda,k}$ such that
\[
     \begin{aligned}
     \gamma_{w,k}
&\le
\min\left\{
\frac{1}{8p_w\beta},
\frac{1}{8(p_w+l_{\mathcal L,1})}
\right\},\\
\gamma_{\lambda,k}
&\le
\min\left\{
\frac{1}{2\left(2l_{\Psi,1}+p_\lambda + \frac{l_{g,1}^2}{p_w+l_{\mathcal{L},1}}\right)},
\frac{l_{\mathcal L,1}^2}{16l_{g,1}^2}\gamma_{w,k},
\frac{1}{1536p_w\beta\gamma_2^2},
\right\}.
\end{aligned}
     \]
Suppose that parameter $\beta\in(0,1)$ satisfies
\[
\beta
\le
\min\left\{
\frac{\sqrt{5}-1}{48\gamma_1},
\frac{1}{3072p_w\gamma_{\lambda,k}\gamma_{3}^2}
\right\}.
\]
Then, we have
\begin{align}
V_k-V_{k+1}
\ge\;&
\frac{1}{32\gamma_{w,k}}
\left\|
w^k-w^+(\hat w^k,\lambda^k)
\right\|^2
+
\frac{1}{64\gamma_{\lambda,k}}
\left\|
\lambda^+(\hat w^k)-\lambda^k
\right\|^2
\nonumber\\
&+
\frac{p_w\beta}{8}
\left\|
w^k-\hat w^k
\right\|^2
-
C_{\tau,k}\tau^2,
\label{eq: one-step Lyapunov inequality}
\end{align}
where
$
C_{\tau,k}
:=
\left(
\frac{\gamma_{w,k}}{4}
+
(p_w+l_{\mathcal L,1})\gamma_{w,k}^2
+
\frac{p_w\beta\gamma_{w,k}^2}{4}
\right)
l_{g,2}^2C_\lambda^4$.
\end{proposition}

The proposition also remains valid if $l_{\mathcal L,1}$ is replaced
consistently by any deterministic upper bound on the smoothness of
$\mathcal L(\cdot,\lambda)$ over $\lambda\in\Lambda$. This is the form used
in Theorem~\ref{Theorem: Deterministic Sample Complexity}, with
$L_0=8\bar L\bar\kappa_y$.

It follows from \eqref{eq: one-step Lyapunov inequality} that the one-step Lyapunov decrease controls three nonnegative stationarity and tracking residuals, up to the finite-difference error $C_{\tau,k}\tau^2$. For a target Lagrangian stationarity accuracy $\epsilon_{\mathcal{L}}$, choosing constant stepsizes and a finite-difference radius
\(\tau=\Theta(\bar\kappa_y^{-1}\epsilon_{\mathcal L})\), and telescoping \eqref{eq: one-step Lyapunov inequality} over \(K=\Theta(\bar\kappa_y^{-1}\epsilon_{\mathcal L}^{-2})\) iterations yields residuals of order \(\epsilon_{\mathcal L}^2\). The full proof is deferred to the appendix. A key ingredient is the special structure of the regularized surrogate \eqref{eq:reformed lagrangian}, which allows us to retain the coupled primal--dual coercivity in Lemma~\ref{Lemma: Error Bound 1} rather than bounding the dual error first.

Then, the following lemma shows that three residual terms controlled by the one-step Lyapunov inequality are sufficient to guarantee $\epsilon_{\mathcal{L}}$-stationarity of the original Lagrangian.
\begin{lemma}\label{Lemma: Deterministic Stationary Point Bridge}
Under the same assumption and hyperparameter setting as Proposition~\ref{Proposition: Recursion of Lyapunov Fucntion},
suppose further that $\epsilon_{\mathcal{L}}$ is sufficiently small such that
\[
\left(\gamma_{w,k}^{-1}+\bar{\kappa}_y p_w\right)\epsilon_{\mathcal{L}}
\le
\frac{\mu_g}{4}
\left(
C_\lambda-\frac{l_{f,0}}{\mu_g}
\right),
\quad
\bar{\kappa}_y\epsilon_{\mathcal{L}}
\le
\frac14
\left(
C_\lambda-\frac{l_{f,0}}{\mu_g}
\right).
\]
Then, if
\[
\|w^k-w^+(\hat{w}^k,\lambda^k)\|
\le \epsilon_{\mathcal{L}},\quad
\|\lambda^+(\hat{w}^k)-\lambda^k\|
\le \bar{\kappa}_y\epsilon_{\mathcal{L}},\quad
\|w^k-\hat{w}^k\|
\le \bar{\kappa}_y\epsilon_{\mathcal{L}},
\]
we have
\[
\begin{aligned}
\|\nabla_w\mathcal L(w^k,\lambda^k)\|
&\le
\left(\gamma_{w,k}^{-1}+\bar{\kappa}_y p_w\right)\epsilon_{\mathcal L},\\
\|\nabla_\lambda \mathcal L(w^k,\lambda^k)\|
&\le
\left(
\frac{l_{g,1}}{(p_w-l_{\mathcal L,1})\gamma_{w,k}}
+
\bar{\kappa}_y\gamma_{\lambda,k}^{-1}
\right)\epsilon_{\mathcal L}
+
p_\lambda C_\lambda.
\end{aligned}
\]
\end{lemma}
Lemma \ref{Lemma: Deterministic Stationary Point Bridge} establishes the bridge between the residual terms in the one-step Lyapunov inequality \eqref{eq: one-step Lyapunov inequality} and the stationarity of the Lagrangian function. Finally, by Theorem~\ref{Theorem: Approximation of Hypergradient}, the convergence of \(\nabla \mathcal L\) can be further translated into the convergence of the hypergradient in the original bilevel problem \eqref{Bilevel Formulation}.

\section{Stochastic NC--SC Bilevel Optimization}

We now extend SGHA to the stochastic setting, where exact gradients of
\(f\) and \(g\) are not directly available. We first specify the stochastic
first-order oracle ($\mathcal{SFO}$) used throughout this paper.

\begin{assumption}
\label{Assumption: Stochastic First-order Oracle}
We access the gradients of \(f\) and \(g\) via unbiased estimators
\(\nabla f(x,y;\xi)\) and \(\nabla g(x,y;\zeta)\) such that
\[
    \mathbb E_{\xi}\!\left[\nabla f(x,y;\xi)\right]
    =
    \nabla f(x,y),
    \quad
    \mathbb E_{\zeta}\!\left[\nabla g(x,y;\zeta)\right]
    =
    \nabla g(x,y).
\]
Moreover, the variance of stochastic gradients is bounded:
\[
    \mathbb E_{\xi}\!\left[
    \|\nabla f(x,y;\xi)-\nabla f(x,y)\|^2
    \right]
    \le
    \sigma_f^2,
    \quad
    \mathbb E_{\zeta}\!\left[
    \|\nabla g(x,y;\zeta)-\nabla g(x,y)\|^2
    \right]
    \le
    \sigma_g^2.
\]
\end{assumption}

We propose
Stoc-SGHA, stated in Algorithm~\ref{Stoc-SGHA}, for solving NC-SC bilevel optimization. Algorithm~\ref{Stoc-SGHA} has the same single--loop structure as
Algorithm~\ref{SGHA}. The only change is that the exact gradients used in SGHA
are replaced by mini-batch stochastic estimators. In particular, the gradients
of \(f\) are estimated using a mini-batch \(\mathcal B_f\), while the
finite-difference approximations of the Hessian-vector products of \(g\) are
computed using a mini-batch \(\mathcal B_g\). For simplicity, we take
\(|\mathcal B_f|=|\mathcal B_g|=B\). With \(\mathcal B_f=\{\xi_1,\ldots,\xi_B\}\), we use
\[
\nabla_x f(x^k,y^k;\mathcal{B}_f^k) = \frac{1}{B}\sum_{i = 1}^B \nabla_x f(x^k,y^k;\xi_i),\quad \nabla_y f(x^k,y^k;\mathcal{B}_f^k) = \frac{1}{B}\sum_{i = 1}^B \nabla_y f(x^k,y^k;\xi_i).
\]
With
\(\mathcal B_g=\{\zeta_1,\ldots,\zeta_B\}\), we use
\[
\tilde{\nabla}_{xy}^2 g_{\lambda}(x,y;\mathcal B_g)
    :=
    \frac{1}{B}\sum_{i=1}^B
    \frac{
    \nabla_x g(x,y+\tau\lambda;\zeta_i)
    -
    \nabla_x g(x,y-\tau\lambda;\zeta_i)
    }{2\tau}
    \approx
    \nabla_{xy}^2g(x,y)\lambda,
\]
and
\[
\tilde{\nabla}_{yy}^2 g_{\lambda}(x,y;\mathcal B_g)
    :=
    \frac{1}{B}\sum_{i=1}^B
    \frac{
    \nabla_y g(x,y+\tau\lambda;\zeta_i)
    -
    \nabla_y g(x,y-\tau\lambda;\zeta_i)
    }{2\tau}
    \approx
    \nabla_{yy}^2g(x,y)\lambda.
\]

The same sample \(\zeta_i\) is used for the two gradient evaluations in each
central difference. This coupling is essential for the sharper variance bound
under Assumption~\ref{Assumption: Stochastic Smoothness of g}.
Under the bounded-variance oracle in
Assumption~\ref{Assumption: Stochastic First-order Oracle}, the stochastic
finite-difference approximation introduces an additional variance term of order
$\sigma_g^2/(B\tau^2)$ relative to the deterministic setting. With
$\tau\asymp \bar\kappa_y^{-1}\epsilon_{\mathcal L}$, controlling this term requires
$B\asymp \epsilon_{\mathcal L}^{-4}$. Combined with
$K\asymp \bar\kappa_y^{-1}\epsilon_{\mathcal L}^{-2}$, this yields an overall sample complexity
of order $\bar\kappa_y^{-1}\epsilon_{\mathcal L}^{-6}$. This additional variance also prevents
a direct combination of existing minimax analyses with finite-difference
approximations, as such an approach would generally lead to worse complexity
guarantees.


\begin{algorithm}[htbp]
\caption{\textbf{Stoc-SGHA} (\textbf{Stoc}hastic \textbf{S}moothed \textbf{G}DA with \textbf{H}essian \textbf{A}pproximation)}
\label{Stoc-SGHA}
\textbf{Input:} Total iterations $K$, mini-batch set $\mathcal{B}_f^k = \{\xi_1^k,\ldots,\xi_B^k\}$, $\mathcal{B}_g^k = \{\zeta_1^k,\ldots,\zeta_B^k\}$, stepsizes
$\{\gamma_{x,k}\}_k, \{\gamma_{y,k}\}_k, \{\gamma_{\lambda,k}\}_k$, initialization points $x^0, \hat{x}^0 \in \mathbb{R}^{d_x}, y^0, \hat{y}^0 \in \mathbb{R}^{d_y}, \lambda^0 \in \Lambda$. $\tau > 0$, $p_w, p_{\lambda} > 0$, $0 < \beta < 1$, $\Lambda := \left\{\lambda: \|\lambda\| \leq C_{\lambda}\right\}$ with $C_{\lambda} > 0$.
\begin{algorithmic}[1]
    \For{$k = 0,\ldots,K-1$}
            \State $x^{k+1} \leftarrow x^k - \gamma_{x,k}\left(\nabla_x f(x^k,y^k;\mathcal{B}_f^k) + \tilde{\nabla}_{xy}^2 g_{\lambda^k}(x^k,y^k;\mathcal{B}_g^k) + p_w(x^k - \hat{x}^k)\right)$
            \State $y^{k+1} \leftarrow y^k - \gamma_{y,k}\left(\nabla_y f(x^k,y^k;\mathcal{B}_f^k) + \tilde{\nabla}_{yy}^2 g_{\lambda^k}(x^k,y^k;\mathcal{B}_g^k) + p_w(y^k - \hat{y}^k)\right)$
        \State $\hat{x}^{k+1} = \hat{x}^k + \beta(x^{k+1} - \hat{x}^k)$
        \State $\hat{y}^{k+1} = \hat{y}^k + \beta(y^{k+1} - \hat{y}^k)$
        \State $\lambda^{k+1} = \mathcal{P}_{\Lambda}(\lambda^k + \gamma_{\lambda,k}\left(\nabla_y g(x^k,y^k;\mathcal{B}_g^k) - p_{\lambda}\lambda^k\right))$
    \EndFor
\end{algorithmic}
\end{algorithm}

The stochastic analysis follows the same general framework as the deterministic case, with Proposition~\ref{Proposition: Stochastic Lyapunov Function Descent After Error Bound} establishing a one-step Lyapunov descent bound in expectation. A subtle difference, however, arises in translating the residual bounds into Lagrangian stationarity. In the deterministic setting, Lemma~\ref{Lemma: Deterministic Stationary Point Bridge} applies pathwise: if the three residuals in \eqref{eq: one-step Lyapunov inequality} are sufficiently small at an iterate, then that iterate is approximately stationary for \(\mathcal L\). In the stochastic setting, the telescoping argument only controls these residuals in expectation, which does not directly imply expected stationarity of \(\mathcal L\). We therefore use Markov's inequality to convert the expected residual bound into a high-probability bound, and then apply Lemma~\ref{Lemma: Deterministic Stationary Point Bridge} pathwise on this event. This yields the following high-probability oracle complexity guarantee for Algorithm~\ref{Stoc-SGHA}.

\begin{theorem}
\label{Theorem: Stochastic Sample Complexity HP}
Suppose that Assumptions \ref{Assumption: Bilevel Optimization} and \ref{Assumption: Stochastic First-order Oracle} hold.
Fix any target accuracy $0<\epsilon\le 1$ and failure probability
$\rho\in(0,1)$, and define
$
\bar\sigma^2:=1+\sigma_f^2+\sigma_g^2$.
Let
$
\Lambda:=\{\lambda:\|\lambda\|\le C_\lambda\}$ with $
C_\lambda:=2\bar\kappa_y$.
Choose
$
p_w=16\bar L\bar\kappa_y$ and $
p_\lambda
=
\frac18
\min\left\{
\mu_g,\,
\frac{\epsilon\sqrt{\rho}}
{2^{12}(1+\bar L)\bar\kappa_y^4}
\right\}$.
For every $k\ge0$, use the constant stepsizes
\[
\gamma_{x,k}
=
\gamma_{y,k}
=
\frac{1}{2^8\bar L\bar\kappa_y},
\qquad
\gamma_{\lambda,k}
=
\frac{\bar\kappa_y}{2^8\bar L}.
\]
Choose the averaging parameter, finite-difference radius, and batch size as
\[
\beta
=
\frac{1}{2^{30}\bar\kappa_y^2},
\qquad
\tau
=
\frac{\epsilon\sqrt{\rho}}
{2^{25}(1+\bar L)\bar\kappa_y^4}, \qquad B
=
\left\lceil
\frac{
2^{98}\bar\sigma^2(1+\bar L)^4
}{
\bar L^2
}
\bar\kappa_y^{12}\epsilon^{-4}\rho^{-2}
\right\rceil.
\]
Let $\Delta_V:=\mathbb E[V_0]-\underline f$, and run
Algorithm~\ref{Stoc-SGHA} for
\[
K
:=
\max\left\{
1,
\left\lceil
\frac{
5\cdot2^{53}(1+\bar L)^2\Delta_V
}{
\bar L
}
\bar\kappa_y^5\epsilon^{-2}\rho^{-1}
\right\rceil
\right\}.
\]
Then there exists
$k_\star\in\{0,\ldots,K-1\}$ such that
$
\mathbb P\left(
\|\nabla F(x^{k_\star})\|\le\epsilon
\right)
\ge 1-\rho$.
Therefore, the stochastic first-order oracle complexity is
\[
\mathcal O\left(
\bar\kappa_y^{17}\epsilon^{-6}\rho^{-3}
\right).
\]
\end{theorem}

We can also establish an in-expectation guarantee under the additional assumption that the iterates are almost surely bounded, i.e., $\|w^k\|\le D_w$ almost surely for all \(k\). This condition is weaker than that typically imposed in analyses of Smoothed--GDA--type methods \citep{zhang2020single}, which require both the primal and dual iterates to remain bounded. We use a good-event decomposition directly at the level of the bilevel hypergradient. On the event where the residuals are small, Lemma~\ref{Lemma: Deterministic Stationary Point Bridge} and Theorem~\ref{Theorem: Approximation of Hypergradient} give a small hypergradient. On the complementary event, boundedness of $w^k$, lower-level strong convexity, and global smoothness provide a uniform bound on $\|\nabla F(x^k)\|$. This sharper treatment avoids an unnecessary condition-number loss from separately bounding the Lagrangian-gradient components on the bad event. It leads to the following stochastic first-order oracle complexity guarantee in expectation for Algorithm~\ref{Stoc-SGHA}.

\begin{theorem}
\label{Theorem: Stochastic Sample Complexity Expectation}
Suppose that Assumptions~\ref{Assumption: Bilevel Optimization}
and~\ref{Assumption: Stochastic First-order Oracle} hold. Further assume
that there exists $D_w>0$ such that the iterates generated by
Algorithm~\ref{Stoc-SGHA} satisfy
\(\|w^k\|\le D_w\) almost surely for all \(k\ge0\).
Fix any target accuracy
\(0<\epsilon\le1\), and let
\(\Lambda:=\{\lambda:\|\lambda\|\le C_\lambda\}\) with
\(C_\lambda:=2\bar\kappa_y\).
Let \(\mathcal G\) be a problem-dependent constant, and denote
\(\bar\sigma^2:=\max\{1,\sigma_f^2,\sigma_g^2\}\).
Choose $p_w=16\bar L\bar\kappa_y$ and $p_\lambda
=
\frac18
\min\left\{
\mu_g,\,
\frac{\epsilon}
{2^{13}(1+\bar L)\mathcal G\bar\kappa_y^4}
\right\}$.
For every $k\ge0$, use the constant stepsizes
\[
\gamma_{x,k}
=
\gamma_{y,k}
=
\frac{1}{2^8\bar L\bar\kappa_y},
\qquad
\gamma_{\lambda,k}
=
\frac{\bar\kappa_y}{2^8\bar L}.
\]
Choose the averaging parameter, finite-difference radius, and batch size as
\[
\beta
=
\frac{1}{2^{30}\bar\kappa_y^2},
\qquad
\tau
=
\frac{\epsilon}
{2^{26}(1+\bar L)\mathcal G\bar\kappa_y^4},
\qquad
B
:=
\left\lceil
\frac{2^{102}\bar\sigma^2(1+\bar L)^4\mathcal G^4}
{\bar L^2}
\bar\kappa_y^{12}\epsilon^{-4}
\right\rceil.
\]
Let $\Delta_V:=\mathbb E[V_0]-\underline f$ and run
Algorithm~\ref{Stoc-SGHA} for
\[
K
:=
\max\left\{
1,\,
\left\lceil
\frac{5\cdot2^{55}(1+\bar L)^2\mathcal G^2\Delta_V}{\bar L}
\bar\kappa_y^5\epsilon^{-2}
\right\rceil
\right\}.
\]
Then there exists $k_\star\in\{0,\ldots,K-1\}$ such that $\mathbb E\left[
\|\nabla F(x^{k_\star})\|
\right]
\le\epsilon$.
Consequently, the stochastic first-order oracle complexity is
\[
\mathcal O\left(
\bar\kappa_y^{17}\epsilon^{-6}
\right).
\]
\end{theorem}


The resulting \(\epsilon\)-dependence matches the best-known fully first-order stochastic upper bound, \(\widetilde{\mathcal O}(\bar{\kappa}_y^{11}\epsilon^{-6})\), achieved by F\(^3\)BSA \citep{chen2025near}, while removing the logarithmic factor. To the best of our knowledge, this is the first single-loop fully first-order method to attain the same \(\epsilon^{-6}\) dependence. However, compared with the stochastic lower bound \(\Omega(\kappa_y^{4.5}\epsilon^{-4})\) in \citep{chen2026conditionnumberdependencybilevel}, an additional \(\epsilon^{-2}\) gap remains, together with a substantial gap in the condition-number dependence. Closing the gap between the best-known upper and lower bounds for stochastic NC--SC bilevel optimization remains an open problem.


\section{Improved Complexities with Lower-Level Stochastic Smoothness}

In this section, we show that the stochastic complexity of Stoc-SGHA can be
improved under an additional stochastic smoothness condition on the lower-level
objective \(g\). This condition is stronger than the standard bounded-variance
oracle model by controlling the mean-square variation of stochastic gradients
across different points.

\begin{assumption}
\label{Assumption: Stochastic Smoothness of g}
There exists \(\tilde l_{g,1}>0\) such that, for all
\((x_1,y_1),(x_2,y_2)\),
\[
\mathbb E_\zeta
\left[
\|\nabla g(x_1,y_1;\zeta)-\nabla g(x_2,y_2;\zeta)\|^2
\right]
\le
\tilde l_{g,1}
\left(
\|x_1-x_2\|^2+\|y_1-y_2\|^2
\right).
\]
\end{assumption}

Assumption~\ref{Assumption: Stochastic Smoothness of g} yields sharper variance control for the stochastic finite-difference approximation of the Hessian--vector products of \(g\). The key is that the two gradient evaluations in each central difference use the same sample in the mini-batch \(\mathcal B_g\).
Under the standard bounded-variance oracle in Assumption~\ref{Assumption: Stochastic First-order Oracle}, the finite-difference error contains a variance term of order \(\sigma_g^2/(B\tau^2)\). Under Assumption~\ref{Assumption: Stochastic Smoothness of g}, this improves to \(\tilde l_{g,1}C_\lambda^2/B\), where \(C_\lambda\) is the radius of the multiplier domain \(\Lambda\). Thus, the variance is no longer amplified by the small finite-difference radius \(\tau\), and it suffices to choose \(B\asymp \epsilon_{\mathcal L}^{-2}\). Combined with \(K\asymp \bar\kappa_y^{-1}\epsilon_{\mathcal L}^{-2}\) iterations from the refined Lyapunov argument, this yields an internal stochastic first-order oracle complexity of order \(\bar\kappa_y^{-1}\epsilon_{\mathcal L}^{-4}\). We first state the corresponding high-probability guarantee.

\begin{theorem}
\label{Theorem: Stochastic Sample Complexity HP with g Smoothness}
Suppose that Assumptions~\ref{Assumption: Bilevel Optimization},
\ref{Assumption: Stochastic First-order Oracle}, and
\ref{Assumption: Stochastic Smoothness of g} hold.
Fix any target accuracy $0<\epsilon\le1$ and failure probability
$\rho\in(0,1)$, and define
$
\bar M_g
:=
\max\left\{
1,\widetilde l_{g,1},\sigma_f^2,\sigma_g^2
\right\}$.
Let
$
\Lambda:=\{\lambda:\|\lambda\|\le C_\lambda\}$ with $
C_\lambda:=2\bar\kappa_y$.
Choose
$
p_w=16\bar L\bar\kappa_y$ and $
p_\lambda
=
\frac18
\min\left\{
\mu_g,\,
\frac{\epsilon\sqrt{\rho}}
{2^{12}(1+\bar L)\bar\kappa_y^4}
\right\}$.
For every $k\ge0$, use the constant stepsizes
\[
\gamma_{x,k}
=
\gamma_{y,k}
=
\frac{1}{2^8\bar L\bar\kappa_y},
\qquad
\gamma_{\lambda,k}
=
\frac{\bar\kappa_y}{2^8\bar L}.
\]
Choose the averaging parameter,  finite-difference radius and batch size as
\[
\beta
=
\frac{1}{2^{30}\bar\kappa_y^2},
\qquad
\tau
=
\frac{\epsilon\sqrt{\rho}}
{2^{25}(1+\bar L)\bar\kappa_y^4}, \qquad
B =
\left\lceil
\frac{
2^{50}\bar M_g(1+\bar L)^2
}{
\bar L^2
}
\bar\kappa_y^6\epsilon^{-2}\rho^{-1}
\right\rceil.
\]
Let $\Delta_V:=\mathbb E[V_0]-\underline f$, and run
Algorithm~\ref{Stoc-SGHA} for
\[
K
:=
\max\left\{
1,\,
\left\lceil
\frac{
5\cdot2^{53}(1+\bar L)^2\Delta_V
}{
\bar L
}
\bar\kappa_y^5\epsilon^{-2}\rho^{-1}
\right\rceil
\right\}.
\]
Then there exists
$k_\star\in\{0,\ldots,K-1\}$ such that
$
\mathbb P\left(
\|\nabla F(x^{k_\star})\|\le\epsilon
\right)
\ge1-\rho$.
Therefore, the stochastic first-order oracle complexity is
\[
\mathcal O\left(
\bar\kappa_y^{11}\epsilon^{-4}\rho^{-2}
\right).
\]
\end{theorem}

Finally, with the additional bounded-iterate condition, we obtain the following stochastic first-order oracle complexity guarantee in expectation
for Algorithm \ref{Stoc-SGHA}.

\begin{theorem}
\label{Theorem: Stochastic Sample Complexity Expectation with g Smoothness}
Suppose that Assumptions~\ref{Assumption: Bilevel Optimization},
\ref{Assumption: Stochastic First-order Oracle}, and
\ref{Assumption: Stochastic Smoothness of g} hold.
Further assume that there exists $D_w>0$ such that the iterates
generated by Algorithm~\ref{Stoc-SGHA} satisfy
$\|w^k\|\le D_w$ almost surely for all $k\ge0$.
Let $\mathcal G$ be a problem-dependent constant, and let $\bar M_g
:=
\max\left\{
1,\widetilde l_{g,1},\sigma_f^2,\sigma_g^2
\right\}$.
Fix any target accuracy
$0<\epsilon\le1$.
Let
$\Lambda:=\{\lambda:\|\lambda\|\le C_\lambda\}$ with
$C_\lambda:=2\bar\kappa_y$.
Choose
$
p_w=16\bar L\bar\kappa_y$ and $
p_\lambda
=
\frac18
\min\left\{
\mu_g,\,
\frac{\epsilon}
{2^{13}(1+\bar L)\mathcal G\bar\kappa_y^4}
\right\}.
$
For every $k\ge0$, use the constant stepsizes
\[
\gamma_{x,k}
=
\gamma_{y,k}
=
\frac{1}{2^8\bar L\bar\kappa_y},
\qquad
\gamma_{\lambda,k}
=
\frac{\bar\kappa_y}{2^8\bar L}.
\]
Choose the averaging parameter, finite-difference radius, and batch
size as
\[
\beta
=
\frac{1}{2^{30}\bar\kappa_y^2},
\qquad
\tau
=
\frac{\epsilon}
{2^{26}(1+\bar L)\mathcal G\bar\kappa_y^4},
\qquad
B
=
\left\lceil
\frac{
2^{52}\bar M_g(1+\bar L)^2\mathcal G^2
}{
\bar L^2
}
\bar\kappa_y^6\epsilon^{-2}
\right\rceil.
\]
Let $\Delta_V:=\mathbb E[V_0]-\underline f$, and run
Algorithm~\ref{Stoc-SGHA} for
\[
K
:=
\max\left\{
1,\,
\left\lceil
\frac{
5\cdot2^{55}(1+\bar L)^2\mathcal G^2\Delta_V
}{
\bar L
}
\bar\kappa_y^5\epsilon^{-2}
\right\rceil
\right\}.
\]
Then there exists
$k_\star\in\{0,\ldots,K-1\}$ such that
$
\mathbb E\left[
\|\nabla F(x^{k_\star})\|
\right]
\le\epsilon$.
Therefore, the stochastic first-order oracle complexity is
\[
\mathcal O\left(
\bar\kappa_y^{11}\epsilon^{-4}
\right).
\]
\end{theorem}


The theorems show that imposing stochastic smoothness only on the lower-level objective is sufficient to achieve an \(\epsilon^{-4}\) oracle complexity. This \(\epsilon\)-dependence is optimal, matching the rate for smooth nonconvex single-level optimization~\citep{arjevani2023lower}. Previously, such a rate was achieved only when the lower-level gradient is noise-free \citep{chen2025near}, a special case in which Assumption~\ref{Assumption: Stochastic Smoothness of g} holds trivially. 
Nevertheless, the \(\bar{\kappa}_y\)-dependence of Algorithm~\ref{Stoc-SGHA} remains worse than the lower bound, and closing this condition-number gap for single-loop fully first-order methods remains an open problem.

\section{Numerical Experiments}
\label{sec:experiments}


We conduct two sets of experiments: learn-to-regularize logistic regression and data hyper-cleaning. For each experiment, we tune the stepsizes and method-specific parameters of all baselines under a common evaluation protocol and report the best stable performance achieved by each method.

\subsection{Learn-to-Regularize Logistic Regression}
\label{sec:exp-l2reg}

We consider the learnable regularization problem for multiclass logistic regression,
which has been widely used in bilevel hyperparameter optimization.  Let
$\mathcal D_{\rm tr}$ and $\mathcal D_{\rm val}$ denote the training and validation sets.
For a feature vector $a_i \in \mathbb R^d$ and label $b_i$, let
$\ell(W; a_i,b_i)$ be the cross-entropy loss of a linear classifier
$W \in \mathbb R^{d \times C}$.  We learn one regularization parameter per feature by
solving
\begin{equation}
\label{eq:l2reg-bilevel}
\begin{aligned}
    \min_{x \in \mathbb R^d} \quad
    & f(x,W^*(x))
      := \frac{1}{|\mathcal D_{\rm val}|}
         \sum_{(a_i,b_i)\in \mathcal D_{\rm val}}
         \ell(W^*(x);a_i,b_i), \\
    \text{s.t.}\quad
    & W^*(x)
      = \arg\min_{W \in \mathbb R^{d\times C}}
      \frac{1}{|\mathcal D_{\rm tr}|}
      \sum_{(a_i,b_i)\in \mathcal D_{\rm tr}}
      \ell(W;a_i,b_i)
      + \frac{1}{2dC}\sum_{j=1}^d \exp(x_j)\|W_{j,\cdot}\|_2^2 .
\end{aligned}
\end{equation}
Here $W_{j,\cdot}\in\mathbb R^C$ denotes the $j$-th row of $W$, i.e., the vector of
classifier weights associated with the $j$-th TF-IDF feature across all classes.
The lower-level objective is strongly convex in $W$ due to the positive feature-wise ridge
regularization, while the upper-level objective is nonconvex in $x$ through both the
exponential parameterization and the solution map $W^*(x)$.

We conduct the experiment on the 20 Newsgroups dataset, which has been commonly used in
bilevel hyperparameter optimization \citep{grazzi2020iteration, ji2021bilevel}.  The dataset contains approximately
$18{,}000$ documents from $20$ classes. Each document is represented by a
TF--IDF vector with $d=101{,}631$ features. We split the original training set
into $7{,}354$ training samples and $3{,}960$ validation samples, and use the
$7{,}532$ documents in the by-date test set for evaluation.

We compare SGHA (Algorithm \ref{SGHA}) with penalty-based methods F$^2$BSA \citep{chen2025near} and F$^2$SA \citep{kwon2023fully}, the HVP-based methods stocBiO
\citep{ji2021bilevel} and MRBO \citep{yang2021provably}, the variance-reduced bilevel
method VRBO \citep{yang2021provably}, and a baseline without hyperparameter optimization.
We use $5$ or $10$ inner loops for the stochastic bilevel baselines with inner updates, and tune the stepsizes of all methods from preliminary logarithmic grids. The hyperparameters of SGHA are selected as $\beta=0.2$, $p_w=0.01$, $p_\lambda=10^{-5}$, and $\tau=10^{-4}$. We run all methods for $1000$ outer iterations.

Figure~\ref{fig:l2reg} summarizes the results, where the dashed line labeled ``w/o Reg'' denotes performance without hyperparameter tuning. SGHA consistently outperforms all baselines in terms of both test loss and test accuracy. Among the compared methods, F$^2$BSA is the strongest baseline, improving over F$^2$SA and the other stochastic bilevel methods, while SGHA still achieves the best overall performance.

\begin{figure}[t]
    \centering
    \includegraphics[width=\linewidth]{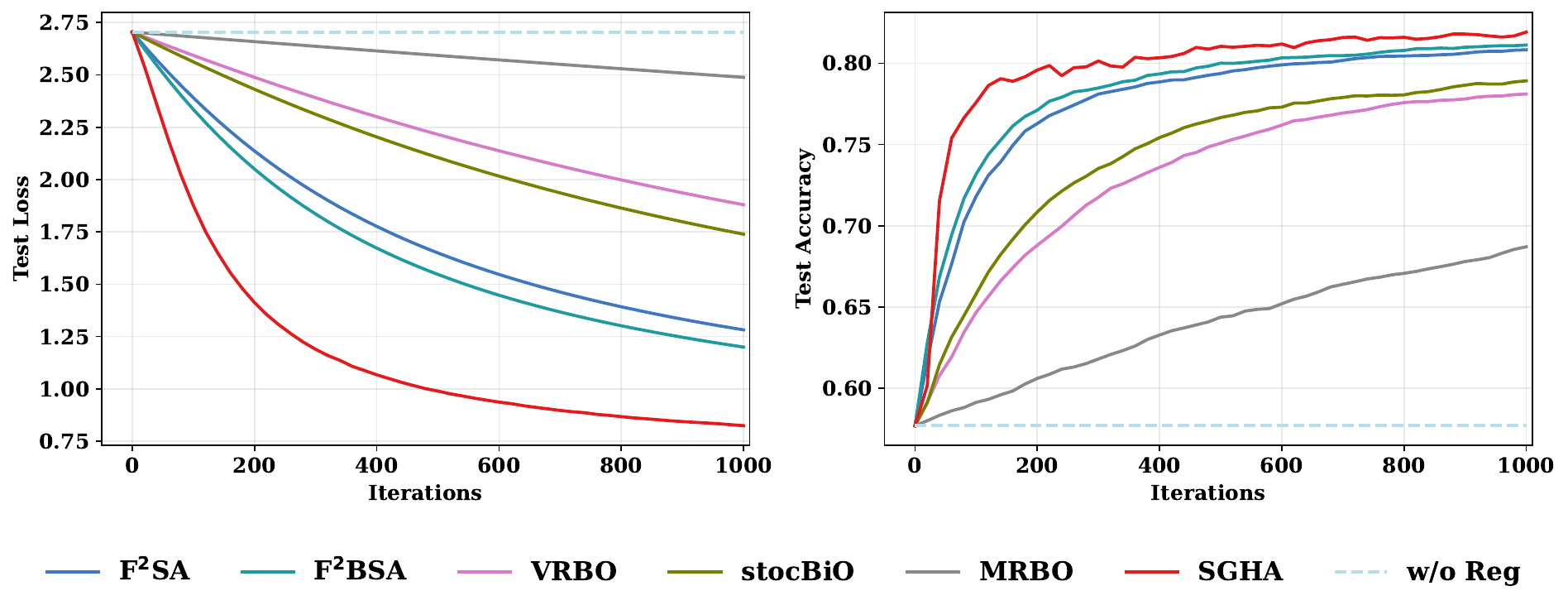}
    \caption{Learn-to-regularize logistic regression on the 20 Newsgroups dataset with
    all $20$ classes and $101,631$ TF-IDF features.  SGHA obtains the lowest test loss and
    the highest test accuracy among all compared methods.}
    \label{fig:l2reg}
\end{figure}

\subsection{Data Hyper-Cleaning on MNIST}
\label{subsec:mnist-hyperclean}

We further evaluate whether SGHA can identify corrupted training examples in a data hyper-cleaning problem. This experiment follows the ``weights on garbage'' diagnostic used in data-cleaning bilevel optimization: when a portion of the training labels is corrupted, an effective hyperparameter optimizer should assign small weights to corrupted examples while retaining useful information from the clean validation set. This experiment is inspired by the data hyper-cleaning formulation \citep{kwon2023fully, chen2025near}, where data-source weights are learned by minimizing validation loss and the performance is evaluated by the weights assigned to garbage data.  We adapt this idea to a stochastic nonconvex--strongly--convex MNIST hyper-cleaning problem with per-example weights and a strongly convex ridge-regularized lower-level logistic regression model.

Let $\widetilde{\mathcal D}_{\rm tr}=\{(a_i,\widetilde b_i)\}_{i=1}^n$ denote a
training set whose labels are independently corrupted with probability $p$, and
let $\mathcal D_{\rm val}$ be a clean validation set.  We associate each training
example with a learnable weight $\alpha_i(x_i)=2\sigma(x_i)$ with $\sigma(t)=\frac{1}{1+\exp(-t)}$. With a multiclass linear classifier $W\in\mathbb R^{d\times C}$ and cross-entropy
loss $\ell$, the bilevel problem is
\begin{equation}
\label{eq:mnist-hyperclean-bilevel}
\begin{aligned}
    \min_{x\in\mathbb R^n}\quad
    & f(x,W^*(x))
      := \frac{1}{|\mathcal D_{\rm val}|}
         \sum_{(a_i,b_i)\in\mathcal D_{\rm val}}
         \ell(W^*(x);a_i,b_i)
      + \frac{\gamma}{2n}\sum_{i=1}^n(\alpha_i(x_i)-1)^2, \\
    \mathrm{s.t.}\quad
    & W^*(x)
      = \arg\min_{W\in\mathbb R^{d\times C}}
      \frac{1}{n}\sum_{(a_i,\widetilde b_i)\in\widetilde{\mathcal D}_{\rm tr}}
      \alpha_i(x_i)\ell(W;a_i,\widetilde b_i)
      + \frac{\mu}{2}\|W\|_F^2 .
\end{aligned}
\end{equation}
The lower-level problem is strongly convex in $W$ due to the ridge term, while
the upper-level problem is nonconvex in $x$ because of the nonlinear sample-weight
map and the solution mapping $W^*(x)$.

We use the full MNIST dataset \citep{lecun1998gradient}, obtained from OpenML
\citep{vanschoren2014openml}, consisting of $70,000$ images with $784$ features and $10$ classes.  The data are split into $38,500$ training examples, $15,750$ validation examples, and $15,750$ test examples. The
validation and test labels are kept clean, while the training labels are randomly corrupted with $p\in\{0.5,0.9,0.99\}$.  We compare SGHA with penalty-based methods F$^2$BSA \citep{chen2025near} and
F$^2$SA~\citep{kwon2023fully}, the HVP-based methods stocBiO~\citep{ji2021bilevel}
and MRBO~\citep{yang2021provably}, the variance-reduced bilevel method
VRBO~\citep{yang2021provably}, and a baseline without learned data weights. All methods are run for $20,000$ iterations using stochastic mini-batches. For SGHA, we use different hyperparameter settings for different corruption ratios. Specifically, for $p=0.5$, $0.9$ and $0.99$, we set $(\beta,p_w)=(0.16,0.02)$, $(0.24,0.005)$, $(0.32,0.005)$, respectively, while fixing $p_\lambda=10^{-4}$ and $\tau=10^{-3}$.

To quantify whether an algorithm suppresses corrupted examples, we report
\[
    \mathrm{GarbageWeight}(x)
    := \frac{1}{|\mathcal G|}\sum_{i\in\mathcal G}2\sigma(x_i),
\]
where $\mathcal G$ is the set of corrupted training examples. Lower values are better.

\begin{figure}[t]
    \centering
    \includegraphics[width=\linewidth]{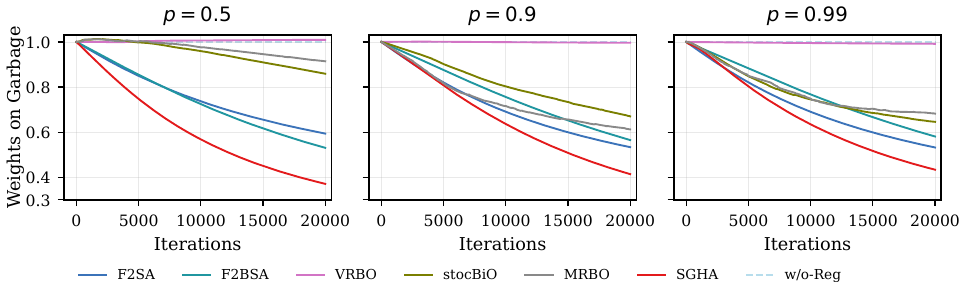}
    \caption{Data hyper-cleaning on the full MNIST dataset under corruption ratios
    $p\in\{0.5,0.9,0.99\}$.  The vertical axis is the average weight assigned to
    corrupted training examples; lower is better.}
    \label{fig:mnist-corruption-sweep}
\end{figure}

Figure~\ref{fig:mnist-corruption-sweep} reports the data hyper-cleaning results, where the dashed line labeled ``w/o Reg'' denotes performance without learning data weights. Across all corruption ratios, SGHA assigns substantially smaller weights to corrupted examples than the competing methods, with the advantage becoming more pronounced as the corruption level increases. Among the baselines, F$^2$BSA performs best overall, but still assigns noticeably larger weights to corrupted samples than SGHA.

\section{Conclusions and Future Directions}

This paper proposes SGHA and Stoc--SGHA, the first single-loop, fully first-order algorithms for NC--SC bilevel optimization that achieve the optimal $\epsilon$-dependence in the deterministic setting and match the best-known $\epsilon$-dependence of existing nested-loop methods in the stochastic setting. Our approach starts from a Lagrangian formulation of the equality-constrained problem induced by the lower-level first-order optimality condition. We then apply Smoothed GDA to a regularized minimax surrogate and replace the resulting Hessian--vector products with finite-difference approximations of first-order gradients. The analysis is tailored to this specific structure, which is key to obtaining the stated complexity guarantees.

Several questions remain open. First, although our algorithms match the
best-known $\epsilon$-dependence in the corresponding settings, their dependence
on $\kappa_y$ and $\bar{\kappa}_y$ remains worse than existing lower bounds and
double-loop penalty methods. Improving the condition-number dependence while
preserving the single-loop, fully first-order structure is an important
direction. Second, the optimal $\epsilon$-dependence for stochastic NC--SC
bilevel optimization under the standard bounded-variance oracle is still
unknown. Existing $\Omega(\kappa_y^{4.5}\epsilon^{-4})$ \citep{chen2026conditionnumberdependencybilevel} lower bound constructs a hard instance with a deterministic upper--level objective and a quadratic lower-level objective, and therefore does not fully capture the difficulty of
upper-level stochasticity and applies to a restricted problem class. We conjecture that, without these special structures, the tight complexity is
$\Omega(\mathrm{poly}(\kappa_y)\epsilon^{-6})$. Finally, it remains interesting
to determine whether the SGHA framework can be extended beyond strongly convex
lower-level problems.


\begingroup
\raggedright
\bibliographystyle{abbrvnat}
\bibliography{SGHA}
\endgroup

\appendix

\section{Properties of Reformulated NC--SC Bilevel Optimization}
\subsection{LICQ of NC-SC Bilevel Optimization}
\paragraph{Proof of Lemma \ref{Lemma: LICQ}}
Since \(g(x,\cdot)\) is \(\mu_g\)-strongly convex, the Jacobian of the constraint mapping with respect to \(y\) is $\nabla_{yy}g(x,y)$, which satisfies
\[
    \nabla_{yy}^2 g(x,y)\succeq \mu_g I.
\]
Hence, \(\nabla_{yy}^2 g(x,y)\) is nonsingular, and the gradients of the equality constraints are linearly independent. Therefore, the full Jacobian of \(\nabla_y g\) has full row rank, which implies LICQ at
every feasible point.

\subsection{Approximation of Hypergradient}
\paragraph{Proof of Theorem \ref{Theorem: Approximation of Hypergradient}}
First, we justify the stated value of $L_{\bar F,y}$. We let
\[
\bar\nabla F(x,y)
=
\nabla_x f(x,y)-\nabla_{xy} g(x,y)\nabla_{yy}^2 g(x,y)^{-1}\nabla_y f(x,y).
\]
For any $y_1,y_2$, by the $l_{f,1}$-smoothness of $f$,
\[
\|\nabla_x f(x,y_1)-\nabla_x f(x,y_2)\|
\le
l_{f,1}\|y_1-y_2\|.
\]
Moreover,
\[
\|\nabla_{xy} g(x,y)\nabla_{yy}^2 g(x,y)^{-1}\|
\le
\frac{l_{g,1}}{\mu_g}
=
\kappa_y,
\]
and hence
\[
\|\nabla_{xy} g(x,y_1)\nabla_{yy} g(x,y_1)^{-1}(\nabla_y f(x,y_1)-\nabla_y f(x,y_2))\|
\le
\kappa_y l_{f,1}\|y_1-y_2\|.
\]
Also, by $A^{-1} - B^{-1} = A^{-1}(B - A)B^{-1}$,
\[
\begin{aligned}
&\|\nabla_{xy} g(x,y_1)\nabla_{yy} g(x,y_1)^{-1}-\nabla_{xy} g(x,y_2)\nabla_{yy} g(x,y_2)^{-1}\|                         \\
&\le
\|\nabla_{xy} g(x,y_1)-\nabla_{xy} g(x,y_2)\|\|\nabla_{yy} g(x,y_1)^{-1}\|
+
\|\nabla_{xy} g(x,y_2)\|\|\nabla_{yy} g(x,y_1)^{-1}-\nabla_{yy} g(x,y_2)^{-1}\|                              \\
&\le
\frac{l_{g,2}}{\mu_g}\|y_1-y_2\|
+
\frac{l_{g,1}}{\mu_g^2}
\|\nabla_{yy} g(x,y_1)-\nabla_{yy} g(x,y_2)\|                                                    \\
&\le
l_{g,2}
\left(
\frac{1}{\mu_g}
+
\frac{l_{g,1}}{\mu_g^2}
\right)
\|y_1-y_2\|.
\end{aligned}
\]
Using $\|\nabla_y f(x,y)\|\le l_{f,0}$, we get
\[
\begin{aligned}
&\|\nabla_{xy} g(x,y_1)\nabla_{yy} g(x,y_1)^{-1}\nabla_y f(x,y_1)
-
\nabla_{xy} g(x,y_2)\nabla_{yy} g(x,y_2)^{-1}\nabla_y f(x,y_2)\|                                \\
&\le
\kappa_y l_{f,1}\|y_1-y_2\|
+
l_{f,0}l_{g,2}
\left(
\frac{1}{\mu_g}
+
\frac{l_{g,1}}{\mu_g^2}
\right)
\|y_1-y_2\|.
\end{aligned}
\]
Combining the above inequalities yields
\[
\|\bar\nabla F(x,y_1)-\bar\nabla F(x,y_2)\|
\le
L_{\bar F,y}\|y_1-y_2\|,
\]
with
\[
L_{\bar F,y}
=
l_{f,1}(1+\kappa_y)
+
l_{f,0}l_{g,2}
\left(
\frac{1}{\mu_g}
+
\frac{l_{g,1}}{\mu_g^2}
\right).
\]
Moreover,
\[
\lambda+\nabla_{yy}^2 g(x,y)^{-1}\nabla_y f(x,y)
=
\nabla_{yy}^2 g(x,y)^{-1}\left(\nabla_y f(x,y)+\nabla_{yy}^2 g(x,y)\lambda\right),
\]
Since $g(x,\cdot)$ is $\mu_g$-strongly convex, we have
\[
\left\|
\lambda+\nabla_{yy}^2 g(x,y)^{-1}\nabla_y f(x,y)
\right\|
\le
\frac{1}{\mu_g}\|\nabla_y f(x,y)+\nabla_{yy}^2 g(x,y)\lambda\|.
\]
Then, we obtain
\[
\begin{aligned}
\bar\nabla F(x,y)
&=
\nabla_x f(x,y)
-
\nabla_{xy} g(x,y)\nabla_{yy}^2 g(x,y)^{-1}\nabla_y f(x,y)                                      \\
&=
\nabla_x f(x,y)+\nabla_{xy} g(x,y)\lambda
-
\nabla_{xy} g(x,y)
\left(
\lambda+\nabla_{yy}^2 g(x,y)^{-1}\nabla_y f(x,y)
\right).
\end{aligned}
\]
Thus,
\[
\begin{aligned}
\|\bar\nabla F(x,y)\|
&\le
\|\nabla_x f(x,y)+\nabla_{xy} g(x,y)\lambda\|
+
\|\nabla_{xy} g(x,y)\|
\left\|
\lambda+\nabla_{yy}^2 g(x,y)^{-1}\nabla_y f(x,y)
\right\|                                                            \\
&\le
\|\nabla_x f(x,y)+\nabla_{xy} g(x,y)\lambda\|
+
\frac{l_{g,1}}{\mu_g}\|
\nabla_y f(x,y)+\nabla_{yy}^2 g(x,y)\lambda\|                                      \\
&=
\|\nabla_x f(x,y)+\nabla_{xy} g(x,y)\lambda\|+\kappa_y\|
\nabla_y f(x,y)+\nabla_{yy}^2 g(x,y)\lambda\|.
\end{aligned}
\]

Next, we transfer the bound from the current $y$ to the lower-level solution
$y^\star(x)$. Since $g(x,\cdot)$ is $\mu_g$-strongly convex,
\[
\nabla_y g(x,y^\star(x))=0
\]
and
\[
\|y-y^\star(x)\|
\le
\frac{1}{\mu_g}
\|\nabla_y g(x,y)-\nabla_y g(x,y^\star(x))\|
=
\frac{1}{\mu_g}\|\nabla_y g(x,y)\|.
\]
Moreover,
\[
\nabla F(x)
=
\bar\nabla F(x,y^\star(x)).
\]
Hence, by the Lipschitz continuity of $\bar\nabla F(x,\cdot)$ with respect to
$y$,
\[
\begin{aligned}
\|\nabla F(x)\|
&=
\|\bar\nabla F(x,y^\star(x))\|                                      \\
&\le
\|\bar\nabla F(x,y)\|
+
\|\bar\nabla F(x,y^\star(x))-\bar\nabla F(x,y)\|                    \\
&\le
\|\nabla_x f(x,y)+\nabla_{xy} g(x,y)\lambda\|
+
\kappa_y\|\nabla_y f(x,y)+\nabla_{yy}^2 g(x,y)\lambda\|
+
L_{\bar F,y}\|y-y^\star(x)\|                                       \\
&\le
\|\nabla_x f(x,y)+\nabla_{xy} g(x,y)\lambda\|
+
\kappa_y\|\nabla_y f(x,y)+\nabla_{yy}^2 g(x,y)\lambda\|
+
\frac{L_{\bar F,y}}{\mu_g}\|\nabla_y g(x,y)\|.
\end{aligned}
\]

\subsection{Boundedness of \texorpdfstring{\(\lambda^*\)}{lambda*}}
\paragraph{Proof of Theorem \ref{Theorem: Boundedness of lambda star}}
At a KKT point $(x^*,y^*,\lambda^*)$, we have the stationarity condition:
\[
\nabla_y \mathcal{L}(x^*,y^*,\lambda^*) = \nabla_y f(x^*,y^*) + (\lambda^*)^{\top}\nabla^2_{yy} g(x^*,y^*)  = 0,
\]
which implies
\[
(\lambda^*)^{\top}\nabla^2_{yy} g(x^*,y^*) = -\nabla_y f(x^*,y^*).
\]
By the strong convexity of $g(x,y)$ with respect to $y$, $\nabla^2_{yy} g(x^*,y^*)$ is positive definite and invertible with $\|[\nabla^2_{yy} g(x^*,y^*)]^{-1}\| \leq \frac{1}{\mu_g}$. Thus, we have
\[
\|\lambda^*\| = \|[\nabla^2_{yy} g(x^*,y^*)]^{-1} \nabla_y f(x^*,y^*)\| \leq \|[\nabla^2_{yy} g(x^*,y^*)]^{-1}\| \cdot \|\nabla_y f(x^*,y^*)\| \leq \frac{l_{f,0}}{\mu_g}.
\]
Thus, $\lambda^*$ is bounded by a constant independent of the specific KKT point.
\section{Deterministic NC-SC Bilevel Optimization}
\subsection{Lower Bound of the Lyapunov Function}
\paragraph{Proof of Lemma \ref{Lemma: Lower Bound of Lyapunov Function}.}
Since \(K(\cdot,\hat w,\lambda)\) is strongly convex and coercive in \(w\)
for each \(\lambda\in\Lambda\), while \(K(w,\hat w,\cdot)\) is continuous
and strongly concave on the compact convex set \(\Lambda\), by Sion's minimax theorem,
\[
P(\hat w)
=
\min_w\max_{\lambda\in\Lambda}K(w,\hat w,\lambda)
=
\max_{\lambda\in\Lambda}\min_wK(w,\hat w,\lambda)
=
\max_{\lambda\in\Lambda}\Psi(\hat w,\lambda).
\]
Then, it is easy to obtain
\begin{align*}
    V =V(w,\hat{w},\lambda)
    &= K(w,\hat{w},\lambda) - 2\Psi(\hat{w},\lambda) + 2P(\hat{w}) \\
    &= P(\hat{w}) + (K(w,\hat{w},\lambda) - \Psi(\hat{w},\lambda))
    + (P(\hat{w}) - \Psi(\hat{w},\lambda)) \\
    &\geq P(\hat{w}),
\end{align*}

Next, by the definition of \(P\), we have
\[
\begin{aligned}
    P(\hat w)
    &=
    \min_w \max_{\lambda\in\Lambda}
    \left\{
        f(w)
        +
        \lambda^\top \nabla_y g(w) +
        \frac{p_w}{2}\|w-\hat w\|^2
        -
        \frac{p_\lambda}{2}\|\lambda\|^2.
    \right\}
\end{aligned}
\]
Since $\Lambda=\{\lambda:\|\lambda\|\le C_\lambda\}$ with \(C_\lambda>0\), we have \(0\in\Lambda\). Therefore, for any fixed \(w\),
\[
\begin{aligned}
    \max_{\lambda\in\Lambda}
    \left\{
        \lambda^\top \nabla_y g(w)
        -
        \frac{p_\lambda}{2}\|\lambda\|^2
    \right\}
    &\ge
    0^\top \nabla_y g(w)
    -
    \frac{p_\lambda}{2}\|0\|^2  \\
    &=0.
\end{aligned}
\]
Consequently,
\[
\begin{aligned}
    P(\hat w)
    &\ge
    \min_w
    \left\{
        f(w)
        +
        \frac{p_w}{2}\|w-\hat w\|^2
    \right\}.
\end{aligned}
\]
Since \(\frac{p_w}{2}\|w-\hat w\|^2\ge 0\), Assumption \ref{Assumption: Bilevel Optimization} implies
\[
    P(\hat w) \ge \underline{f},
    \qquad \forall \hat w.
\]
Together with \(V(w,\hat{w},\lambda) \geq P(\hat w)\), this gives
\[
    V_t\ge \underline{f}.
\]
This completes the proof.

\subsection{Properties of Deterministic Finite Difference Approximation}
\begin{lemma}
\label{Lemma: Deterministic Finite Difference Estimator}
Define the deterministic finite-difference approximation of  $\nabla_{xy}^2g(x,y)\lambda$ and $\nabla_{yy}^2g(x,y)\lambda$ as
\[
\tilde{\nabla}^2_{xy} g_{\lambda}(x,y)
:=
\frac{
\nabla_x g(x,y+\tau\lambda)-\nabla_x g(x,y-\tau\lambda)
}{
2\tau
},
\]
and
\[
\tilde{\nabla}^2_{yy} g_{\lambda}(x,y)
:=
\frac{
\nabla_y g(x,y+\tau\lambda)-\nabla_y g(x,y-\tau\lambda)
}{
2\tau
}.
\]
Under Assumption \ref{Assumption: Bilevel Optimization} and
\[
\lambda\in\Lambda=\{\lambda:\|\lambda\|\le C_\lambda\},
\]
we have
\begin{equation}
\label{Lemma: Dete FD 1}
\left\|
\nabla^2_{xy} g(x,y)\lambda
-
\tilde{\nabla}^2_{xy} g_{\lambda}(x,y)
\right\|
\le
\frac{l_{g,2}C_\lambda^2}{2}\tau,
\end{equation}
and
\begin{equation}
\label{Lemma: Dete FD 2}
\left\|
\nabla^2_{yy} g(x,y)\lambda
-
\tilde{\nabla}^2_{yy} g_{\lambda}(x,y)
\right\|
\le
\frac{l_{g,2}C_\lambda^2}{2}\tau.
\end{equation}
\end{lemma}
\begin{proof}
We first prove \eqref{Lemma: Dete FD 1}. Define
\[
G_x(s):=\nabla_x g(x,y+s\lambda).
\]
Then
\[
G_x'(s)
=
\nabla^2_{xy}g(x,y+s\lambda)\lambda.
\]
Hence,
\[
\tilde{\nabla}^2_{xy} g_{\lambda}(x,y)
=
\frac{G_x(\tau)-G_x(-\tau)}{2\tau}
=
\frac{1}{2\tau}\int_{-\tau}^{\tau}G_x'(s)\,ds.
\]
Therefore,
\[
\begin{aligned}
\left\|
\tilde{\nabla}^2_{xy} g_{\lambda}(x,y)
-
\nabla^2_{xy}g(x,y)\lambda
\right\|
&=
\left\|
\frac{1}{2\tau}
\int_{-\tau}^{\tau}
\left[
G_x'(s)-G_x'(0)
\right]ds
\right\|                                                       \\
&\le
\frac{1}{2\tau}
\int_{-\tau}^{\tau}
\left\|
\left(
\nabla^2_{xy}g(x,y+s\lambda)
-
\nabla^2_{xy}g(x,y)
\right)\lambda
\right\|ds                                                     \\
&\le
\frac{1}{2\tau}
\int_{-\tau}^{\tau}
l_{g,2}|s|\|\lambda\|^2\,ds                                    \\
&=
\frac{l_{g,2}}{2}\tau\|\lambda\|^2                              \\
&\le
\frac{l_{g,2}C_\lambda^2}{2}\tau.
\end{aligned}
\]
This proves \eqref{Lemma: Dete FD 1}. The proof of \eqref{Lemma: Dete FD 2} is identical to that of
\eqref{Lemma: Dete FD 1}.
\end{proof}
\subsection{Preliminary Lemmas}
\begin{lemma}\label{Lemma: Smoothness of L(w,lambda)}
    Under Assumption \ref{Assumption: Bilevel Optimization}, and
\[
\lambda\in\Lambda=\{\lambda:\|\lambda\|\le C_\lambda\},
\]
let $l_{\mathcal L,1}$ be any constant satisfying
\[
l_{\mathcal L,1}
\ge
\sqrt{2}(l_{f,1}+C_{\lambda}l_{g,2}).
\]
Then $\nabla_w \mathcal{L}(w,\lambda)$ satisfies
    \[
    \|\nabla_w \mathcal{L}(w,\lambda) - \nabla_w \mathcal{L}(w^{\prime},\lambda)\| \leq l_{\mathcal L,1}\|w - w^{\prime}\|.
    \]
    \[
    \|\nabla_w \mathcal{L}(w,\lambda) - \nabla_w \mathcal{L}(w,\lambda')\| \leq l_{g,1}\|\lambda - \lambda^{\prime}\|.
    \]
\end{lemma}

\begin{proof}
By the definition of $\mathcal{L}(x,y,\lambda)$, we have
    \begin{align*}
    \left\|\nabla_w \mathcal{L}(w,\lambda) - \nabla_w \mathcal{L}(w^{\prime},\lambda)\right\| &= \left\|\begin{pmatrix}
        \nabla_x f(x,y) - \nabla_x f(x^{\prime},y^{\prime}) \\
        \nabla_y f(x,y) - \nabla_y f(x^{\prime},y^{\prime})
    \end{pmatrix} + \begin{pmatrix}
        \nabla_{xy}^2 g(x,y) - \nabla_{xy}^2 g(x^{\prime},y^{\prime})  \\
        \nabla_{yy}^2 g(x,y) - \nabla_{yy}^2 g(x^{\prime},y^{\prime})
    \end{pmatrix}\lambda\right\|\\
    &\leq (l_{f,1} + C_{\lambda}l_{g,2})\left(\|x - x^{\prime}\| + \|y - y^{\prime}\|\right) \\
    &\leq \sqrt{2}(l_{f,1} + C_{\lambda}l_{g,2})\sqrt{\|x - x^{\prime}\|^2 + \|y - y^{\prime}\|^2} \\
    &= \sqrt{2}(l_{f,1} + C_{\lambda}l_{g,2})\|w - w^{\prime}\|.
    \end{align*}
    Similarly,
    \begin{align*}
    \left\|\nabla_w \mathcal{L}(w,\lambda) - \nabla_w \mathcal{L}(w,\lambda^{\prime})\right\| = \left\|\begin{pmatrix}
        \nabla_{xy}^2 g(x,y)(\lambda - \lambda^{\prime})  \\
    \nabla_{yy}^2 g(x,y)(\lambda - \lambda^{\prime})
    \end{pmatrix}\right\| \leq l_{g,1}\|\lambda - \lambda^{\prime}\|.
    \end{align*}
    The first bound remains valid for every larger deterministic upper
    bound $l_{\mathcal L,1}$, which proves the stated form.
\end{proof}

\begin{lemma}\label{Lemma: SC property of K}
    When $p_w > l_{\mathcal{L},1}$, $K(w,\hat{w},\lambda)$ is $(p_w - l_{\mathcal{L},1})$-strongly convex w.r.t. $w$, and $p_{\lambda}$-strongly concave w.r.t. $\lambda$. Moreover, we have
    \[
    \|w^{\prime} - w^{\prime\prime}\| \leq \frac{1}{p_w - l_{\mathcal{L},1}}\|\nabla_w K(w^{\prime},\hat{w},\lambda) - \nabla_w K(w^{\prime\prime},\hat{w},\lambda)\|.
    \]
\end{lemma}
\begin{proof}
Since $K(w,\hat{w},\lambda)$ is $(p_w - l_{\mathcal{L},1})$-strongly convex with respect to $w$, its gradient satisfies the strong monotonicity property
\begin{equation}\label{Lemma: SC property of K 1}
\left\langle
\nabla_w K(w^{\prime},\hat{w},\lambda) - \nabla_w K(w^{\prime\prime},\hat{w},\lambda),
\, w^{\prime} - w^{\prime\prime}
\right\rangle
\ge
(p_w - l_{\mathcal{L},1}) \|w^{\prime} - w^{\prime\prime}\|^2.
\end{equation}
On the other hand, by the Cauchy--Schwarz inequality, we have
\begin{equation}\label{Lemma: SC property of K 2}
\left\langle
\nabla_w K(w^{\prime},\hat{w},\lambda) - \nabla_w K(w^{\prime\prime},\hat{w},\lambda),
\, w^{\prime} - w^{\prime\prime}
\right\rangle
\le
\|\nabla_w K(w^{\prime},\hat{w},\lambda) - \nabla_w K(w^{\prime\prime},\hat{w},\lambda)\|
\cdot
\|w^{\prime} - w^{\prime\prime}\|.
\end{equation}
Combining \eqref{Lemma: SC property of K 1} and \eqref{Lemma: SC property of K 2}, we have
\[
(p_w - l_{\mathcal{L},1}) \|w^{\prime} - w^{\prime\prime}\|^2
\le
\|\nabla_w K(w^{\prime},\hat{w},\lambda) - \nabla_w K(w^{\prime\prime},\hat{w},\lambda)\|
\cdot
\|w^{\prime} - w^{\prime\prime}\|.
\]
Dividing both sides by $\|w^{\prime} - w^{\prime\prime}\|$ completes the proof.
\end{proof}

\begin{lemma}\label{Lemma: Lipschitz continuity of w star}
When $p_w > l_{\mathcal{L},1}$, we have
\begin{align*}
&\|w^*(\hat{w},\lambda) - w^*(\hat{w}^{\prime},\lambda)\| \le \gamma_1 \|\hat{w} - \hat{w}^{\prime}\|, \\
&\|w^*(\hat{w}) - w^*(\hat{w}^{\prime})\| \le \gamma_1 \|\hat{w} - \hat{w}^{\prime}\|, \\
&\|w^*(\hat{w},\lambda) - w^*(\hat{w},\lambda^{\prime})\| \le \gamma_2 \|\lambda - \lambda^{\prime}\|,
\end{align*}
where $\gamma_1 = \frac{p_w}{p_w -l_{\mathcal{L},1}}$, $\gamma_2 = \frac{l_{g,1}}{p_w -l_{\mathcal{L},1}}$.
\end{lemma}
\begin{proof}
By Lemma \ref{Lemma: SC property of K}, we have
\begin{equation}\label{Lemma: Lipschitz continuity of w star 1}
(p_w - l_{\mathcal{L},1}) \|w^*(\hat{w},\lambda) - w^*(\hat{w}^{\prime},\lambda)\|
\le \|\nabla_w K(w^*(\hat{w},\lambda),\hat{w},\lambda)  - \nabla_w K(w^*(\hat{w}^{\prime},\lambda),\hat{w},\lambda)\|.
\end{equation}
Using the first-order optimality condition
\[
\nabla_w K(w^*(\hat{w},\lambda),\hat{w},\lambda) = 0,
\quad
\nabla_w K(w^*(\hat{w}^{\prime},\lambda),\hat{w}^{\prime},\lambda) = 0,
\]
we have
\begin{align*}
\nabla_w K(w^*(\hat{w},\lambda),\hat{w},\lambda)
- \nabla_w K(w^*(\hat{w}^{\prime},\lambda),\hat{w},\lambda)
&= \nabla_w K(w^*(\hat{w},\lambda),\hat{w},\lambda)  - \nabla_w K(w^*(\hat{w}^{\prime},\lambda),\hat{w}^{\prime},\lambda) \\
&\quad + \nabla_w K(w^*(\hat{w}^{\prime},\lambda),\hat{w}^{\prime},\lambda)
- \nabla_w K(w^*(\hat{w}^{\prime},\lambda),\hat{w},\lambda) \\
&= \nabla_w K(w^*(\hat{w}^{\prime},\lambda),\hat{w}^{\prime},\lambda)
- \nabla_w K(w^*(\hat{w}^{\prime},\lambda),\hat{w},\lambda).
\end{align*}
Then, \eqref{Lemma: Lipschitz continuity of w star 1} reduces to
\begin{equation}\label{Lemma: Lipschitz continuity of w star 2}
(p_w - l_{\mathcal{L},1}) \|w^*(\hat{w},\lambda) - w^*(\hat{w}^{\prime},\lambda)\|
\le \|\nabla_w K(w^*(\hat{w}^{\prime},\lambda),\hat{w}^{\prime},\lambda)
- \nabla_w K(w^*(\hat{w}^{\prime},\lambda),\hat{w},\lambda)\|.
\end{equation}
Using
\[
\nabla_w K(w,\hat{w},\lambda)
= \nabla_w \mathcal{L}(w,\lambda) + p_w(w - \hat{w}),
\]
\eqref{Lemma: Lipschitz continuity of w star 2} reduces to
\[
(p_w - l_{\mathcal{L},1}) \|w^*(\hat{w},\lambda) - w^*(\hat{w}^{\prime},\lambda)\|
\le p_w \|\hat{w} - \hat{w}^{\prime}\|,
\]
which implies
\[
\|w^*(\hat{w},\lambda) - w^*(\hat{w}^{\prime},\lambda)\|
\le \frac{p_w}{p_w - l_{\mathcal{L},1}} \|\hat{w} - \hat{w}^{\prime}\| = \gamma_1\|\hat{w} - \hat{w}^{\prime}\|.
\]
This proves the first claim.

Similarly, since $\Phi(w,\hat{w})$ is still $(p_w - l_{\mathcal{L},1})$-strongly convex with respect to $w$, we have
\begin{equation}\label{Lemma: Lipschitz continuity of w star 4}
(p_w - l_{\mathcal{L},1}) \|w^*(\hat{w}) - w^*(\hat{w}^{\prime})\| \le \|\nabla_w \Phi(w^*(\hat{w}),\hat{w}) - \nabla_w \Phi(w^*(\hat{w}^{\prime}),\hat{w})\|.
\end{equation}
Using the first-order optimality conditions
\[
\nabla_w \Phi(w^*(\hat{w}),\hat{w}) = 0,
\quad
\nabla_w \Phi(w^*(\hat{w}^{\prime}),\hat{w}^{\prime}) = 0,
\]
we have
\begin{align*}
\nabla_w \Phi(w^*(\hat{w}),\hat{w}) - \nabla_w \Phi(w^*(\hat{w}^{\prime}),\hat{w})
&= \nabla_w \Phi(w^*(\hat{w}),\hat{w})
- \nabla_w \Phi(w^*(\hat{w}^{\prime}),\hat{w}^{\prime}) \\
&\quad + \nabla_w \Phi(w^*(\hat{w}^{\prime}),\hat{w}^{\prime})
- \nabla_w \Phi(w^*(\hat{w}^{\prime}),\hat{w}) \\
&= \nabla_w \Phi(w^*(\hat{w}^{\prime}),\hat{w}^{\prime})
- \nabla_w \Phi(w^*(\hat{w}^{\prime}),\hat{w}).
\end{align*}
Then, \eqref{Lemma: Lipschitz continuity of w star 4} reduces to
\begin{align}\label{Lemma: Lipschitz continuity of w star 5}
(p_w - l_{\mathcal{L},1}) \|w^*(\hat{w}) - w^*(\hat{w}^{\prime})\| \le \|\nabla_w \Phi(w^*(\hat{w}^{\prime}),\hat{w}^{\prime})
- \nabla_w \Phi(w^*(\hat{w}^{\prime}),\hat{w})\|.
\end{align}
Using
\[
\nabla_w \Phi(w^*(\hat{w}^{\prime}),\hat{w}^{\prime})
- \nabla_w \Phi(w^*(\hat{w}^{\prime}),\hat{w})
= p_w(\hat{w} - \hat{w}^{\prime}),
\]
\eqref{Lemma: Lipschitz continuity of w star 5} reduces to
\[
(p_w - l_{\mathcal{L},1}) \|w^*(\hat{w}) - w^*(\hat{w}^{\prime})\|
\le p_w \|\hat{w} - \hat{w}^{\prime}\|,
\]
which implies
\[
\|w^*(\hat{w}) - w^*(\hat{w}^{\prime})\|
\le \frac{p_w}{p_w - l_{\mathcal{L},1}} \|\hat{w} - \hat{w}^{\prime}\| = \gamma_1\|\hat{w} - \hat{w}^{\prime}\|.
\]
This proves the second claim.

Similarly, since
\[
\nabla_w K(w^*(\hat{w},\lambda),\hat{w},\lambda) = 0, \quad
\nabla_w K(w^*(\hat{w},\lambda'),\hat{w},\lambda') = 0.
\]
we have
\begin{equation}\label{Lemma: Lipschitz continuity of w star 6}
\nabla_w K(w^*(\hat{w},\lambda),\hat{w},\lambda)
- \nabla_w K(w^*(\hat{w},\lambda'),\hat{w},\lambda) = - \left(\nabla_w K(w^*(\hat{w},\lambda'),\hat{w},\lambda)
- \nabla_w K(w^*(\hat{w},\lambda'),\hat{w},\lambda')\right).
\end{equation}
By taking norms over both sides of \eqref{Lemma: Lipschitz continuity of w star 6}, we have
\begin{equation}\label{Lemma: Lipschitz continuity of w star 7}
\|\nabla_w K(w^*(\hat{w},\lambda),\hat{w},\lambda)
- \nabla_w K(w^*(\hat{w},\lambda'),\hat{w},\lambda)\| = \|\nabla_w K(w^*(\hat{w},\lambda'),\hat{w},\lambda)
- \nabla_w K(w^*(\hat{w},\lambda'),\hat{w},\lambda')\|.
\end{equation}
By strong convexity of $K(w,\hat{w},\lambda)$ w.r.t. $w$, we have
\begin{equation}\label{Lemma: Lipschitz continuity of w star 8}
    (p_w - l_{\mathcal{L},1})\|w^*(\hat{w},\lambda) - w^*(\hat{w},\lambda')\| \leq \|\nabla_w K(w^*(\hat{w},\lambda),\hat{w},\lambda)
- \nabla_w K(w^*(\hat{w},\lambda'),\hat{w},\lambda)\|
\end{equation}
Recall Lemma \ref{Lemma: Smoothness of L(w,lambda)}, we have
\begin{equation}\label{Lemma: Lipschitz continuity of w star 9}
\|\nabla_w K(w^*(\hat{w},\lambda'),\hat{w},\lambda)
- \nabla_w K(w^*(\hat{w},\lambda'),\hat{w},\lambda')\|
\le l_{g,1}\|\lambda - \lambda'\|.
\end{equation}
Plugging \eqref{Lemma: Lipschitz continuity of w star 8} and \eqref{Lemma: Lipschitz continuity of w star 9} into \eqref{Lemma: Lipschitz continuity of w star 7}, we have
\[
(p_w - l_{\mathcal{L},1})
\|w^*(\hat{w},\lambda) - w^*(\hat{w},\lambda')\|
\le l_{g,1}\|\lambda - \lambda'\|,
\]
which implies
\[
\|w^*(\hat{w},\lambda) - w^*(\hat{w},\lambda')\|
\le \frac{l_{g,1}}{p_w - l_{\mathcal{L},1}} \|\lambda - \lambda'\| = \gamma_2\|\lambda - \lambda'\|.
\]
This proves the third claim.
\end{proof}

\begin{lemma}\label{Lemma: Deterministic Iterates Gap}
    Under Assumption \ref{Assumption: Bilevel Optimization} and the update rule of Algorithm \ref{SGHA}, we have
    \begin{align*}
    &\left\|x^{k+1} - x^k\right\|^2 \leq 2\gamma_{x,k}^2\left\|\nabla_x K(w^k,\hat{w}^k,\lambda^k)\right\|^2 + \frac{\gamma_{x,k}^2l_{g,2}^2C_{\lambda}^4\tau^2}{2}, \\
&\left\|y^{k+1} - y^k\right\|^2 \leq 2\gamma_{y,k}^2\left\|\nabla_y K(w^k,\hat{w}^k,\lambda^k)\right\|^2 + \frac{\gamma_{y,k}^2l_{g,2}^2C_{\lambda}^4\tau^2}{2},\\
&\left\|\lambda^{k+1} - \lambda^k\right\|^2 \leq 2\gamma_{\lambda,k}^2\|\nabla_y  g(x^k,y^k)\|^2 + 2\gamma_{\lambda,k}^2p_{\lambda}^2C_{\lambda}^2.
\end{align*}
\end{lemma}
\begin{proof}
According to the update rule of $x^k$ in Algorithm \ref{SGHA} and by \eqref{Lemma: Dete FD 1} in Lemma \ref{Lemma: Deterministic Finite Difference Estimator}, we have
    \begin{align*}
        &\|x^{k+1} - x^k\|^2 \nonumber \\
        &= \gamma_{x,k}^2\left\|\nabla_x f(x^k,y^k) + \tilde{\nabla}_{xy}^2 g_{\lambda^k}(x^k,y^k) + p_w(x^k - \hat{x}^k)\right\|^2\nonumber\\
        &\leq 2\gamma_{x,k}^2\left\|\nabla_x f(x^k,y^k) + \nabla_{xy}^2g(x^k,y^k)\lambda^k + p_w(x^k - \hat{x}^k)\right\|^2 + 2\gamma_{x,k}^2\left\|\nabla_{xy}^2g(x^k,y^k)\lambda^k - \tilde{\nabla}_{xy}^2 g_{\lambda^k}(x^k,y^k)\right\|^2 \\
        &\leq 2\gamma_{x,k}^2\left\|\nabla_x K(w^k,\hat{w}^k,\lambda^k)\right\|^2 + \frac{\gamma_{x,k}^2l_{g,2}^2C_{\lambda}^4\tau^2}{2}.
    \end{align*}
    Similarly, for $y^k$,
\[
\|y^{k+1} - y^k\|^2 \leq 2\gamma_{y,k}^2\left\|\nabla_y K(w^k,\hat{w}^k,\lambda^k)\right\|^2 + \frac{\gamma_{y,k}^2l_{g,2}^2C_{\lambda}^4\tau^2}{2}.
\]
For $\lambda^k$, by nonexpansiveness of projection operator and $\|\lambda^k\| \leq C_{\lambda}$,
\begin{align*}
    \|\lambda^{k+1} - \lambda^k\|^2
    \leq  \gamma_{\lambda,k}^2\|\nabla_y g(x^k,y^k) - p_{\lambda}\lambda^k\|^2 \leq 2\gamma_{\lambda,k}^2\|\nabla_y  g(x^k,y^k)\|^2 + 2\gamma_{\lambda,k}^2p_{\lambda}^2C_{\lambda}^2.
\end{align*}
\end{proof}
\subsection{Intermediate Lemmas for Lyapunov Function}
\begin{lemma}\label{Lemma: Deterministic Primal Descent w lambda}
    Under Assumption \ref{Assumption: Bilevel Optimization}, suppose $\gamma_{w,k} \leq \frac{1}{8(p_w + l_{\mathcal{L},1})}$. Then, we have
    \begin{align*}
        &K(w^k,\hat{w}^k,\lambda^k) - K(w^{k+1},\hat{w}^k,\lambda^{k+1})\\
        &\geq \frac{\gamma_{w,k}}{4}\left\|\nabla_w K(w^k,\hat{w}^k,\lambda^k)\right\|^2  -  \left(\frac{\gamma_{w,k}}{4} + (p_w + l_{\mathcal{L},1})\gamma_{w,k}^2\right)l_{g,2}^2C_{\lambda}^4\tau^2  \\
        &\quad + \langle \nabla_y g(x^k,y^k)- p_{\lambda}\lambda^k, \lambda^k - \lambda^{k+1}\rangle - \left(\frac{p_{\lambda}}{2} + \frac{l_{g,1}^2}{2(p_w+l_{\mathcal{L},1})}\right)\|\lambda^{k+1} - \lambda^k\|^2.
    \end{align*}
\end{lemma}
\begin{proof}
Since $K(w,\hat{w},\lambda)$ is $(p_w + l_{\mathcal{L},1})$-smooth w.r.t. $w$, and $p_{\lambda}$-smooth w.r.t. $\lambda$, we have
    \begin{align*}
        &K(w^k,\hat{w}^k,\lambda^k)  - K(w^{k+1},\hat{w}^k,\lambda^{k+1})\\
        &\geq K(w^k,\hat{w}^k,\lambda^k)  - K(w^{k+1},\hat{w}^k,\lambda^k) + K(w^{k+1},\hat{w}^k,\lambda^k)  - K(w^{k+1},\hat{w}^k,\lambda^{k+1}) \\
        &\geq \langle \nabla_w K(w^k,\hat{w}^k,\lambda^k), w^k - w^{k+1}\rangle + \langle \nabla_y g(x^k,y^k) - p_{\lambda}\lambda^k, \lambda^k  - \lambda^{k+1}\rangle \\
        &\quad + \langle \nabla_y g(x^{k+1},y^{k+1}) - \nabla_y g(x^k,y^k),\lambda^k - \lambda^{k+1}\rangle- \frac{p_w+l_{\mathcal{L},1}}{2}\|w^{k+1} - w^k\|^2 - \frac{p_{\lambda}}{2}\|\lambda^{k+1} - \lambda^k\|^2.
    \end{align*}
    By Young's inequality,
    \begin{align*}
    &\langle \nabla_y g(x^{k+1},y^{k+1}) - \nabla_y g(x^k,y^k),\lambda^k - \lambda^{k+1}\rangle \nonumber\\&\geq - \frac{p_w+l_{\mathcal{L},1}}{2}\|w^{k+1} - w^k\|^2 - \frac{l_{g,1}^2}{2(p_w+l_{\mathcal{L},1})}\|\lambda^{k+1} - \lambda^k\|^2.
\end{align*}
    Then, according to the update rules of $x^k$ and $y^k$, we have
    \begin{align*}
        &K(w^k,\hat{w}^k,\lambda^k)  - K(w^{k+1},\hat{w}^k,\lambda^{k+1})\\
        &\geq \gamma_{x,k}\left\langle \nabla_x K(w^k,\hat{w}^k,\lambda^k), \nabla_x f(x^k,y^k) + \tilde{\nabla}_{xy}^2 g_{\lambda^k}(x^k,y^k) + p_w(x^k - \hat{x}^k)\right\rangle\\
        &\quad + \gamma_{y,k}\left\langle \nabla_y K(w^k,\hat{w}^k,\lambda^k), \nabla_y f(x^k,y^k) + \tilde{\nabla}_{yy}^2 g_{\lambda^k}(x^k,y^k) + p_w(y^k - \hat{y}^k)\right\rangle\\
        &\quad + \langle \nabla_y g(x^k,y^k)- p_{\lambda}\lambda^k, \lambda^k - \lambda^{k+1}\rangle - \frac{p_w+l_{\mathcal{L},1}}{2}\|w^{k+1} - w^k\|^2 - \frac{l_{g,1}^2}{2(p_w+l_{\mathcal{L},1})}\|\lambda^{k+1} - \lambda^k\|^2 \\
        &\quad - \frac{p_w+l_{\mathcal{L},1}}{2}\|w^{k+1} - w^k\|^2 - \frac{p_{\lambda}}{2}\|\lambda^{k+1} - \lambda^k\|^2\\
        &= \gamma_{x,k}\left\|\nabla_x K(w^k,\hat{w}^k,\lambda^k)\right\|^2 + \gamma_{y,k}\left\|\nabla_y K(w^k,\hat{w}^k,\lambda^k)\right\|^2\\
        &\quad - \gamma_{x,k}\left\langle \nabla_x K(w^k,\hat{w}^k,\lambda^k), \nabla_{xy}^2 g(x^k,y^k)\lambda^k -\tilde{\nabla}_{xy}^2 g_{\lambda^k}(x^k,y^k)\right\rangle \\
        &\quad - \gamma_{y,k}\left\langle \nabla_y K(w^k,\hat{w}^k,\lambda^k), \nabla_{yy}^2 g(x^k,y^k)\lambda^k - \tilde{\nabla}_{yy}^2 g_{\lambda^k}(x^k,y^k)\right\rangle \\
        &\quad + \langle \nabla_y g(x^k,y^k)- p_{\lambda}\lambda^k, \lambda^k - \lambda^{k+1}\rangle  - (p_w+l_{\mathcal{L},1})\|w^{k+1} - w^k\|^2 - \left(\frac{p_{\lambda}}{2} + \frac{l_{g,1}^2}{2(p_w+l_{\mathcal{L},1})}\right)\|\lambda^{k+1} - \lambda^k\|^2.
    \end{align*}
    By using that fact that $\langle x,y\rangle \leq \frac{1}{2}\|x\|^2 + \frac{1}{2}\|y\|^2$, we have
    \begin{align*}
        &K(w^k,\hat{w}^k,\lambda^k)  - K(w^{k+1},\hat{w}^k,\lambda^{k+1})\\
        &\geq \frac{\gamma_{x,k}}{2}\left\|\nabla_x K(w^k,\hat{w}^k,\lambda^k)\right\|^2 + \frac{\gamma_{y,k}}{2}\left\|\nabla_y K(w^k,\hat{w}^k,\lambda^k)\right\|^2\\
        &\quad -  \frac{\gamma_{x,k}}{2}\left\|\nabla_{xy}^2 g(x^k,y^k)\lambda^k -\tilde{\nabla}_{xy}^2 g_{\lambda^k}(x^k,y^k)\right\|^2 -  \frac{\gamma_{y,k}}{2}\left\|\nabla_{yy}^2 g(x^k,y^k)\lambda^k -\tilde{\nabla}_{yy}^2g_{\lambda^k}(x^k,y^k)\right\|^2 \\
        &\quad + \langle \nabla_y g(x^k,y^k)- p_{\lambda}\lambda^k, \lambda^k - \lambda^{k+1}\rangle - (p_w+l_{\mathcal{L},1})\|w^{k+1} - w^k\|^2 - \left(\frac{p_{\lambda}}{2} + \frac{l_{g,1}^2}{2(p_w+l_{\mathcal{L},1})}\right)\|\lambda^{k+1} - \lambda^k\|^2.
    \end{align*}
    By Lemma \ref{Lemma: Deterministic Iterates Gap}, and choosing $\gamma_{w,k} = \gamma_{x,k} = \gamma_{y,k} \leq \frac{1}{8(p_w + l_{\mathcal{L},1})}$, we have
    \begin{align*}
        &K(w^k,\hat{w}^k,\lambda^k) - K(w^{k+1},\hat{w}^k,\lambda^{k+1})\\
        &\geq \left(\frac{\gamma_{x,k}}{2} - 2(p_w+l_{\mathcal{L},1})\gamma_{x,k}^2\right)\left\|\nabla_x K(w^k,\hat{w}^k,\lambda^k)\right\|^2 + \left(\frac{\gamma_{y,k}}{2} - 2(p_w+l_{\mathcal{L},1})\gamma_{y,k}^2\right)\left\|\nabla_y K(w^k,\hat{w}^k,\lambda^k)\right\|^2\\
        &\quad - \left(\frac{\gamma_{x,k}}{2} + \frac{\gamma_{y,k}}{2}\right)\frac{l_{g,2}^2C_{\lambda}^4}{4}\tau^2 - \frac{(p_w + l_{\mathcal{L},1})(\gamma_{x,k}^2 + \gamma_{y,k}^2)}{2}l_{g,2}^2C_{\lambda}^4\tau^2 \\
        &\quad+ \langle \nabla_y g(x^k,y^k)- p_{\lambda}\lambda^k, \lambda^k - \lambda^{k+1}\rangle - \left(\frac{p_{\lambda}}{2} + \frac{l_{g,1}^2}{2(p_w+l_{\mathcal{L},1})}\right)\|\lambda^{k+1} - \lambda^k\|^2 \\
        &\geq \frac{\gamma_{w,k}}{4}\left\|\nabla_w K(w^k,\hat{w}^k,\lambda^k)\right\|^2  -  \left(\frac{\gamma_{w,k}}{4} + (p_w + l_{\mathcal{L},1})\gamma_{w,k}^2\right)l_{g,2}^2C_{\lambda}^4\tau^2  \\
        &\quad + \langle \nabla_y g(x^k,y^k)- p_{\lambda}\lambda^k, \lambda^k - \lambda^{k+1}\rangle - \left(\frac{p_{\lambda}}{2} + \frac{l_{g,1}^2}{2(p_w+l_{\mathcal{L},1})}\right)\|\lambda^{k+1} - \lambda^k\|^2.
    \end{align*}
\end{proof}

\begin{lemma}\label{Lemma: Primal Descent hat w}
    The update of $\hat{w}$ in Algorithm \ref{SGHA} yields
    \[
    K(w^{k+1},\hat{w}^{k},\lambda^{k+1}) - K(w^{k+1},\hat{w}^{k+1},\lambda^{k+1}) \geq \frac{p_w}{2\beta}\|\hat{w}^{k+1} - \hat{w}^k\|^2.
    \]
\end{lemma}
\begin{proof}
By the definition of $K(w,\hat{w},\lambda)$ and the update of $\hat{w}$, as $0 < \beta < 1$, we have
    \begin{align*}
        K(w^{k+1},\hat{w}^{k},\lambda^{k+1}) - K(w^{k+1},\hat{w}^{k+1},\lambda^{k+1})
        &= \frac{p_w}{2}\left[\|w^{k+1} - \hat{w}^k\|^2 - \|w^{k+1} - \hat{w}^{k+1}\|^2\right] \\
        &= \frac{p_w}{2}\left[\frac{1}{\beta^2}\|\hat{w}^{k+1} - \hat{w}^k\|^2 - \frac{(1 - \beta)^2}{\beta^2}\|\hat{w}^{k+1} - \hat{w}^k\|^2\right] \\
        &\geq \frac{p_w}{2\beta}\|\hat{w}^{k+1} - \hat{w}^k\|^2.
    \end{align*}
\end{proof}

\begin{lemma}\label{Lemma: Dual Ascent hat w lambda}
    With $l_{\Psi,1} = p_{\lambda} + l_{g,1}\gamma_2$, we have
    \[
    \Psi(\hat{w}^k,\lambda^{k+1}) - \Psi(\hat{w}^k,\lambda^{k}) \geq \left\langle \nabla_{\lambda}\Psi(\hat{w}^k,\lambda^{k}),\lambda^{k+1} - \lambda^k \right\rangle - \frac{l_{\Psi,1}}{2}\|\lambda^{k+1} - \lambda^k\|^2,
    \]
    \[
    \Psi(\hat{w}^{k+1},\lambda^{k+1}) - \Psi(\hat{w}^{k},\lambda^{k+1}) \geq \frac{p_w}{2}(\hat{w}^{k+1} - \hat{w}^k)^{\top}\left[\hat{w}^{k+1} + \hat{w}^k - 2w^*(\hat{w}^{k+1},\lambda^{k+1})\right].
    \]
\end{lemma}
\begin{proof}
By Danskin’s theorem, we have
    \[
    \nabla_{\lambda}\Psi(\hat{w},\lambda) = \nabla_{\lambda}K(w^*(\hat{w},\lambda),\hat{w},\lambda).
    \]
    Thus, by Lemma \ref{Lemma: Lipschitz continuity of w star}, we have
    \begin{align*}
        \|\nabla_{\lambda}\Psi(\hat{w},\lambda^{\prime}) - \nabla_{\lambda}\Psi(\hat{w},\lambda^{\prime\prime})\| &= \|\nabla_{\lambda}K(w^*(\hat{w},\lambda^{\prime}),\hat{w},\lambda^{\prime}) - \nabla_{\lambda}K(w^*(\hat{w},\lambda^{\prime\prime}),\hat{w},\lambda^{\prime\prime}\| \\
        &\leq \|\nabla_{\lambda}K(w^*(\hat{w},\lambda^{\prime}),\hat{w},\lambda^{\prime}) - \nabla_{\lambda}K(w^*(\hat{w},\lambda^{\prime}),\hat{w},\lambda^{\prime\prime}\| \\
        &\quad+ \|\nabla_{\lambda}K(w^*(\hat{w},\lambda^{\prime}),\hat{w},\lambda^{\prime\prime}) - \nabla_{\lambda}K(w^*(\hat{w},\lambda^{\prime\prime}),\hat{w},\lambda^{\prime\prime}\| \\
        &\leq p_{\lambda}\|\lambda^{\prime} - \lambda^{\prime\prime}\| + \|\nabla_y g(w^*(\hat{w},\lambda^{\prime})) - \nabla_y g(w^*(\hat{w},\lambda^{\prime\prime}))\| \\
        &\leq p_{\lambda}\|\lambda^{\prime} - \lambda^{\prime\prime}\| +l_{g,1}\|w^*(\hat{w},\lambda^{\prime}) - w^*(\hat{w},\lambda^{\prime\prime})\| \\
        &\leq (p_{\lambda} + l_{g,1}\gamma_2)\|\lambda^{\prime} - \lambda^{\prime\prime}\|,
    \end{align*}
    which implies $\Psi(\hat{w},\lambda^{\prime})$ is $l_{\Psi,1}$-smooth w.r.t. $\lambda$.

    On the other hand, by the definition of $w^*(\hat{w}^k,\lambda^{k+1})$, we have
    \begin{align*}
        \Psi(\hat{w}^{k+1},\lambda^{k+1}) - \Psi(\hat{w}^{k},\lambda^{k+1})
        &= K(w^*(\hat{w}^{k+1},\lambda^{k+1}),\hat{w}^{k+1},\lambda^{k+1}) - K(w^*(\hat{w}^{k},\lambda^{k+1}),\hat{w}^{k},\lambda^{k+1}) \\
        &\geq K(w^*(\hat{w}^{k+1},\lambda^{k+1}),\hat{w}^{k+1},\lambda^{k+1}) - K(w^*(\hat{w}^{k+1},\lambda^{k+1}),\hat{w}^{k},\lambda^{k+1}) \\
        &= \frac{p_w}{2}\left[\|\hat{w}^{k+1} - w^*(\hat{w}^{k+1},\lambda^{k+1})\|^2 - \|\hat{w}^{k} - w^*(\hat{w}^{k+1},\lambda^{k+1})\|^2\right] \\
        &= \frac{p_w}{2}(\hat{w}^{k+1} - \hat{w}^k)^{\top}[\hat{w}^{k+1} + \hat{w}^k - 2w^*(\hat{w}^{k+1},\lambda^{k+1})].
    \end{align*}
\end{proof}

\begin{lemma}\label{Lemma: Proximal Descent hat w}
    The proximal descent satisfies
    \[
    P(\hat{w}^{k+1}) - P(\hat{w}^k) \leq \frac{p_w}{2}(\hat{w}^{k+1} - \hat{w}^k)^{\top}\left[\hat{w}^{k+1} + \hat{w}^k - 2w^*(\hat{w}^k,\hat{\lambda}^*(\hat{w}^{k+1}))\right].
    \]
\end{lemma}
\begin{proof}
By the definition of $\hat{\lambda}^*(\hat{w}^k)$, we have
    \begin{align*}
        P(\hat{w}^{k+1}) - P(\hat{w}^k) &= \Psi(\hat{w}^{k+1},\hat{\lambda}^*(\hat{w}^{k+1})) - \Psi(\hat{w}^{k},\hat{\lambda}^*(\hat{w}^{k})) \\
        &\leq \Psi(\hat{w}^{k+1},\hat{\lambda}^*(\hat{w}^{k+1})) - \Psi(\hat{w}^{k},\hat{\lambda}^*(\hat{w}^{k+1})) \\
        &= K(w^*(\hat{w}^{k+1},\hat{\lambda}^*(\hat{w}^{k+1})),\hat{w}^{k+1},\hat{\lambda}^*(\hat{w}^{k+1})) - K(w^*(\hat{w}^{k},\hat{\lambda}^*(\hat{w}^{k+1})),\hat{w}^{k},\hat{\lambda}^*(\hat{w}^{k+1})) \\
        &\leq K(w^*(\hat{w}^{k},\hat{\lambda}^*(\hat{w}^{k+1})),\hat{w}^{k+1},\hat{\lambda}^*(\hat{w}^{k+1})) - K(w^*(\hat{w}^{k},\hat{\lambda}^*(\hat{w}^{k+1})),\hat{w}^{k},\hat{\lambda}^*(\hat{w}^{k+1})) \\
        &= \frac{p_w}{2}\left[\left\|w^*(\hat{w}^{k},\hat{\lambda}^*(\hat{w}^{k+1})) - \hat{w}^{k+1}\right\|^2 - \left\|w^*(\hat{w}^{k},\hat{\lambda}^*(\hat{w}^{k+1})) - \hat{w}^{k}\right\|^2\right] \\
        &= \frac{p_w}{2}(\hat{w}^{k+1} - \hat{w}^k)^{\top}\left[\hat{w}^{k+1} + \hat{w}^k - 2w^*(\hat{w}^{k},\hat{\lambda}^*(\hat{w}^{k+1}))\right].
    \end{align*}
\end{proof}

\begin{lemma}\label{Lemma: Dual Ascent - Proximal Descent}
    The following inequality holds:
    \begin{align*}
        &2\Psi(\hat{w}^{k+1},\lambda^{k+1}) - 2\Psi(\hat{w}^{k},\lambda^{k+1}) - 2P(\hat{w}^{k+1}) + 2P(\hat{w}^k) \\
        &\geq - \left(2p_w\gamma_1 + \frac{p_w}{6\beta} + 48p_w\beta \gamma_1^2\right)\|\hat{w}^{k+1} - \hat{w}^k\|^2 - 24p_w\beta \|w^*(\hat{w}^{k}) - w^*(\hat{w}^{k},\lambda^{k})\|^2 - 24p_w\beta\gamma_2^2\|\lambda^{k+1} - \lambda^k\|^2.
    \end{align*}
\end{lemma}
\begin{proof}
Combining Lemma \ref{Lemma: Dual Ascent hat w lambda} and Lemma \ref{Lemma: Proximal Descent hat w}, and by Lemma \ref{Lemma: Lipschitz continuity of w star}, we have
    \begin{align*}
        &2\Psi(\hat{w}^{k+1},\lambda^{k+1}) - 2\Psi(\hat{w}^{k},\lambda^{k+1}) - 2P(\hat{w}^{k+1}) + 2P(\hat{w}^k) \\
        &\geq 2p_w(\hat{w}^{k+1} - \hat{w}^k)^{\top}\left[w^*(\hat{w}^{k},\hat{\lambda}^*(\hat{w}^{k+1})) - w^*(\hat{w}^{k+1},\lambda^{k+1})\right] \\
        &= 2p_w(\hat{w}^{k+1} - \hat{w}^k)^{\top}\left[w^*(\hat{w}^{k},\hat{\lambda}^*(\hat{w}^{k+1})) - w^*(\hat{w}^{k+1},\hat{\lambda}^*(\hat{w}^{k+1}))\right] \\
        &\quad + 2p_w(\hat{w}^{k+1} - \hat{w}^k)^{\top}\left[w^*(\hat{w}^{k+1},\hat{\lambda}^*(\hat{w}^{k+1})) - w^*(\hat{w}^{k+1},\lambda^{k+1})\right] \\
        &\geq -2p_w\gamma_1\|\hat{w}^{k+1} - \hat{w}^k\|^2 + 2p_w(\hat{w}^{k+1} - \hat{w}^k)^{\top}\left[w^*(\hat{w}^{k+1},\hat{\lambda}^*(\hat{w}^{k+1})) - w^*(\hat{w}^{k+1},\lambda^{k+1})\right] \\
        &\geq - \left(2p_w\gamma_1 + \frac{p_w}{6\beta}\right)\|\hat{w}^{k+1} - \hat{w}^k\|^2 - 6p_w\beta \|w^*(\hat{w}^{k+1}) - w^*(\hat{w}^{k+1},\lambda^{k+1})\|^2 \\
        &\geq - \left(2p_w\gamma_1 + \frac{p_w}{6\beta}\right)\|\hat{w}^{k+1} - \hat{w}^k\|^2 - 24p_w\beta \|w^*(\hat{w}^{k}) - w^*(\hat{w}^{k},\lambda^{k})\|^2 - 24p_w\beta \|w^*(\hat{w}^{k}) - w^*(\hat{w}^{k+1})\|^2\\
        &\quad  - 24p_w\beta\|w^*(\hat{w}^{k+1},\lambda^{k+1}) - w^*(\hat{w}^{k},\lambda^{k+1})\|^2 - 24p_w\beta \|w^*(\hat{w}^{k},\lambda^{k+1}) - w^*(\hat{w}^{k},\lambda^{k})\|^2 \\
        &\geq - \left(2p_w\gamma_1 + \frac{p_w}{6\beta} + 48p_w\beta \gamma_1^2\right)\|\hat{w}^{k+1} - \hat{w}^k\|^2 - 24p_w\beta \|w^*(\hat{w}^{k}) - w^*(\hat{w}^{k},\lambda^{k})\|^2 - 24p_w\beta\gamma_2^2\|\lambda^{k+1} - \lambda^k\|^2.
    \end{align*}
\end{proof}

\begin{lemma}\label{Lemma: Primal Descent - Dual Ascent}
Set $p_w=2l_{\mathcal L,1}$. Suppose
\[
\gamma_{w,k}\le \frac{1}{8(p_w+l_{\mathcal L,1})},
\quad
\gamma_{\lambda,k}
\le
\min\left\{
\frac{1}{2\left(2l_{\Psi,1}+p_\lambda + \frac{l_{g,1}^2}{p_w+l_{\mathcal{L},1}}\right)},
\frac{l_{\mathcal L,1}^2}{16l_{g,1}^2}\gamma_{w,k}
\right\}.
\]
Then, we have
\[
\begin{aligned}
&K(w^k,\hat{w}^k,\lambda^k)
-
K(w^{k+1},\hat{w}^k,\lambda^{k+1})
+
2\Psi(\hat{w}^k,\lambda^{k+1})
-
2\Psi(\hat{w}^k,\lambda^k)                                      \\
&\ge
\frac{\gamma_{w,k}}{8}
\left\|
\nabla_wK(w^k,\hat{w}^k,\lambda^k)
\right\|^2
+
\frac{1}{4\gamma_{\lambda,k}}
\left\|
\lambda^{k+1}-\lambda^k
\right\|^2
-
C_{V,k}\tau^2,
\end{aligned}
\]
where
\[
C_{V,k} = \left(
\frac{\gamma_{w,k}}{4}
+
(p_w+l_{\mathcal L,1})\gamma_{w,k}^2
\right)
l_{g,2}^2C_\lambda^4.
\]
\end{lemma}
\begin{proof}
Combining Lemma \ref{Lemma: Deterministic Primal Descent w lambda} and Lemma \ref{Lemma: Dual Ascent hat w lambda} gives
\[
\begin{aligned}
&K(w^k,\hat{w}^k,\lambda^k)
-
K(w^{k+1},\hat{w}^k,\lambda^{k+1})
+
2\Psi(\hat{w}^k,\lambda^{k+1})
-
2\Psi(\hat{w}^k,\lambda^k)                                      \\
&\ge
\frac{\gamma_{w,k}}{4}
\left\|
\nabla_wK(w^k,\hat{w}^k,\lambda^k)
\right\|^2
-
\left(
\frac{\gamma_{w,k}}{4}
+
(p_w+l_{\mathcal L,1})\gamma_{w,k}^2
\right)
l_{g,2}^2C_\lambda^4\tau^2                                  \\
&\quad
+
\left\langle
\nabla_y g(x^k,y^k)-p_\lambda\lambda^k,
\lambda^{k+1}-\lambda^k
\right\rangle                                                    \\
&\quad
+
2\left\langle
\nabla_\lambda\Psi(\hat{w}^k,\lambda^k)
-
\left(
\nabla_y g(x^k,y^k)-p_\lambda\lambda^k
\right),
\lambda^{k+1}-\lambda^k
\right\rangle                                                    \\
&\quad
-
\left(
l_{\Psi,1}+\frac{p_\lambda}{2} + \frac{l_{g,1}^2}{2(p_w+l_{\mathcal{L},1})}
\right)
\left\|
\lambda^{k+1}-\lambda^k
\right\|^2 .
\end{aligned}
\]
Since
\[
\lambda^{k+1}
=
\mathcal{P}_\Lambda\left(
\lambda^k+
\gamma_{\lambda,k}
\left(
\nabla_y g(x^k,y^k)-p_\lambda\lambda^k
\right)
\right),
\]
the first-order optimality condition of the projection gives
\[
\left\langle
\nabla_y g(x^k,y^k)-p_\lambda\lambda^k,
\lambda^{k+1}-\lambda^k
\right\rangle
\ge
\frac{1}{\gamma_{\lambda,k}}
\left\|
\lambda^{k+1}-\lambda^k
\right\|^2 .
\]
Therefore,
\[
\begin{aligned}
&\left\langle
\nabla_y g(x^k,y^k)-p_\lambda\lambda^k,
\lambda^{k+1}-\lambda^k
\right\rangle
-
\left(
l_{\Psi,1}+\frac{p_\lambda}{2}
\right)
\left\|
\lambda^{k+1}-\lambda^k
\right\|^2                                                       \\
&\ge
\left(
\frac{1}{\gamma_{\lambda,k}}
-
l_{\Psi,1}
-
\frac{p_\lambda}{2}
\right)
\left\|
\lambda^{k+1}-\lambda^k
\right\|^2 .
\end{aligned}
\]

Next, by Danskin's theorem,
\[
\nabla_\lambda\Psi(\hat{w}^k,\lambda^k)
=
\nabla_y g(w^\star(\hat{w}^k,\lambda^k))
-
p_\lambda\lambda^k .
\]
Hence, by the Lipschitz continuity of $\nabla_y g$ and Lemma \ref{Lemma: SC property of K},
\[
\begin{aligned}
\left\|
\nabla_\lambda\Psi(\hat{w}^k,\lambda^k)
-
\left(
\nabla_y g(x^k,y^k)-p_\lambda\lambda^k
\right)
\right\|  &=
\left\|
\nabla_y g(w^\star(\hat{w}^k,\lambda^k))
-
\nabla_y g(w^k)
\right\|                                                        \\
&\le
l_{g,1}
\left\|
w^\star(\hat{w}^k,\lambda^k)-w^k
\right\|                                                        \\
&\le
\frac{l_{g,1}}{p_w-l_{\mathcal L,1}}
\left\|
\nabla_wK(w^k,\hat{w}^k,\lambda^k)
\right\|.
\end{aligned}
\]
Using Young's inequality, we obtain
\[
\begin{aligned}
&2\left\langle
\nabla_\lambda\Psi(\hat{w}^k,\lambda^k)
-
\left(
\nabla_y g(x^k,y^k)-p_\lambda\lambda^k
\right),
\lambda^{k+1}-\lambda^k
\right\rangle                                                   \\
&\ge
-
\frac{1}{2\gamma_{\lambda,k}}
\left\|
\lambda^{k+1}-\lambda^k
\right\|^2
-
2\gamma_{\lambda,k}
\left\|
\nabla_\lambda\Psi(\hat{w}^k,\lambda^k)
-
\left(
\nabla_y g(x^k,y^k)-p_\lambda\lambda^k
\right)
\right\|^2                                                      \\
&\ge
-
\frac{1}{2\gamma_{\lambda,k}}
\left\|
\lambda^{k+1}-\lambda^k
\right\|^2
-
\frac{2\gamma_{\lambda,k}l_{g,1}^2}{(p_w-l_{\mathcal L,1})^2}
\left\|
\nabla_wK(w^k,\hat{w}^k,\lambda^k)
\right\|^2 .
\end{aligned}
\]
Since
\[
\gamma_{\lambda,k}
\le
\frac{1}{2\left(2l_{\Psi,1}+p_\lambda + \frac{l_{g,1}^2}{p_w+l_{\mathcal{L},1}}\right)},
\]
we have
\[
\frac{1}{\gamma_{\lambda,k}}
-
l_{\Psi,1}
-
\frac{p_\lambda}{2}
-
\frac{l_{g,1}^2}{2(p_w+l_{\mathcal{L},1})}
-
\frac{1}{2\gamma_{\lambda,k}}
\ge
\frac{1}{4\gamma_{\lambda,k}}.
\]
Therefore,
\[
\begin{aligned}
&K(w^k,\hat{w}^k,\lambda^k)
-
K(w^{k+1},\hat{w}^k,\lambda^{k+1})
+
2\Psi(\hat{w}^k,\lambda^{k+1})
-
2\Psi(\hat{w}^k,\lambda^k)                                      \\
&\ge
\frac{\gamma_{w,k}}{4}
\left\|
\nabla_wK(w^k,\hat{w}^k,\lambda^k)
\right\|^2
-
\frac{2\gamma_{\lambda,k}l_{g,1}^2}{(p_w-l_{\mathcal L,1})^2}
\left\|
\nabla_wK(w^k,\hat{w}^k,\lambda^k)
\right\|^2          +
\frac{1}{4\gamma_{\lambda,k}}
\left\|
\lambda^{k+1}-\lambda^k
\right\|^2                                                       \\
&\quad
-
\left(
\frac{\gamma_{w,k}}{4}
+
(p_w+l_{\mathcal L,1})\gamma_{w,k}^2
\right)
l_{g,2}^2C_\lambda^4\tau^2 .
\end{aligned}
\]
Finally, since $p_w=2l_{\mathcal L,1}$ and
\[
\gamma_{\lambda,k}
\le
\frac{(p_w-l_{\mathcal L,1})^2}{16l_{g,1}^2}\gamma_{w,k} =\frac{l_{\mathcal L,1}^2}{16l_{g,1}^2}\gamma_{w,k},
\]
we have
\[
\frac{2\gamma_{\lambda,k}l_{g,1}^2}{(p_w-l_{\mathcal L,1})^2}
\le
\frac{\gamma_{w,k}}{8}.
\]
Thus,
\[
\begin{aligned}
&K(w^k,\hat{w}^k,\lambda^k)
-
K(w^{k+1},\hat{w}^k,\lambda^{k+1})
+
2\Psi(\hat{w}^k,\lambda^{k+1})
-
2\Psi(\hat{w}^k,\lambda^k)                                      \\
&\ge
\frac{\gamma_{w,k}}{8}
\left\|
\nabla_wK(w^k,\hat{w}^k,\lambda^k)
\right\|^2
+
\frac{1}{4\gamma_{\lambda,k}}
\left\|
\lambda^{k+1}-\lambda^k
\right\|^2      -
\left(
\frac{\gamma_{w,k}}{4}
+
(p_w+l_{\mathcal L,1})\gamma_{w,k}^2
\right)
l_{g,2}^2C_\lambda^4\tau^2 .
\end{aligned}
\]
This completes the proof.
\end{proof}

\begin{proposition}\label{Proposition: Deterministic Lyapunov Function Descent}
Define the Lyapunov function
\[
V_k
=
K(w^k,\hat{w}^k,\lambda^k)
-
2\Psi(\hat{w}^k,\lambda^k)
+
2P(\hat{w}^k).
\]
Set $p_w=2l_{\mathcal L,1}$. Suppose
\[
\gamma_{w,k}
\le
\min\left\{
\frac{1}{8p_w\beta},
\frac{1}{8(p_w+l_{\mathcal L,1})}
\right\},\quad \gamma_{\lambda,k}
\le
\min\left\{
\frac{1}{2(2l_{\Psi,1}+p_\lambda + \frac{l_{g,1}^2}{p_w+l_{\mathcal{L},1}})},
\frac{l_{\mathcal L,1}^2}{16l_{g,1}^2}\gamma_{w,k},
\frac{1}{1536p_w\beta\gamma_2^2}
\right\},
\]
and
\[
\beta
\le
\frac{\sqrt{5}-1}{48\gamma_1}.
\]
Then, we have
\[
\begin{aligned}
V_k-V_{k+1}
&\geq
\frac{1}{32\gamma_{w,k}}
\left\|
w^k-w^+(\hat{w}^k,\lambda^k)
\right\|^2
+
\frac{1}{32\gamma_{\lambda,k}}
\left\|
\lambda^+(\hat{w}^k)-\lambda^k
\right\|^2         +
\frac{p_w\beta}{8}
\left\|
w^k-\hat{w}^k
\right\|^2                                          \\
&\quad
-
48p_w\beta
\left\|
w^\star(\hat{w}^k)
-
w^\star(\hat{w}^k,\lambda^+(\hat{w}^k))
\right\|^2
-
C_{\tau,k}\tau^2,
\end{aligned}
\]
where
\[
C_{\tau,k}
:=
\left(
\frac{\gamma_{w,k}}{4}
+
(p_w+l_{\mathcal L,1})\gamma_{w,k}^2
\right)
l_{g,2}^2C_\lambda^4 + \frac{p_{w}\beta\gamma_{w,k}^2l_{g,2}^2C_{\lambda}^4}{4}.
\]
\end{proposition}
\begin{proof}
Combining Lemma \ref{Lemma: Primal Descent hat w}, Lemma \ref{Lemma: Dual Ascent - Proximal Descent} and Lemma \ref{Lemma: Primal Descent - Dual Ascent}, and using Lemma \ref{Lemma: Deterministic Iterates Gap}, we have
\[
\begin{aligned}
V_k-V_{k+1}
&=
K(w^{k+1},\hat{w}^k,\lambda^{k+1})
-
K(w^{k+1},\hat{w}^{k+1},\lambda^{k+1})                         \\
&\quad+
2\bigl(
\Psi(\hat{w}^{k+1},\lambda^{k+1})
-
\Psi(\hat{w}^k,\lambda^{k+1})
\bigr)
-
2\bigl(
P(\hat{w}^{k+1})-P(\hat{w}^k)
\bigr)                                                        \\
&\quad+
K(w^k,\hat{w}^k,\lambda^k)
-
K(w^{k+1},\hat{w}^k,\lambda^{k+1})        +
2\left(
\Psi(\hat{w}^k,\lambda^{k+1})
-
\Psi(\hat{w}^k,\lambda^k)
\right)                                                        \\
&\geq
\left(
\frac{p_w}{2\beta}
-
2p_w\gamma_1
-
\frac{p_w}{6\beta}
-
48p_w\beta\gamma_1^2
\right)
\left\|
\hat{w}^{k+1}-\hat{w}^k
\right\|^2                                                 +
\left(
\frac{1}{4\gamma_{\lambda,k}}
-
24p_w\beta\gamma_2^2
\right)
\left\|
\lambda^{k+1}-\lambda^k
\right\|^2                                                     \\
&\quad-
24p_w\beta
\left\|
w^\star(\hat{w}^k)
-
w^\star(\hat{w}^k,\lambda^k)
\right\|^2  +
\frac{\gamma_{w,k}}{8}
\left\|
\nabla_wK(w^k,\hat{w}^k,\lambda^k)
\right\|^2
-
C_{V,k}\tau^2.
\end{aligned}
\]
By
\[
\beta\le \frac{\sqrt{5}-1}{48\gamma_1},
\]
we have
\[
\frac{p_w}{2\beta}
-
2p_w\gamma_1
-
\frac{p_w}{6\beta}
-
48p_w\beta\gamma_1^2
\ge
\frac{p_w}{4\beta}.
\]
Moreover, by
\[
\gamma_{\lambda,k}
\le
\frac{1}{192p_w\beta\gamma_2^2},
\]
we have
\[
\frac{1}{4\gamma_{\lambda,k}}
-
24p_w\beta\gamma_2^2
\ge
\frac{1}{8\gamma_{\lambda,k}}.
\]
Therefore,
\[
\begin{aligned}
V_k-V_{k+1}
&\geq
\frac{\gamma_{w,k}}{8}
\left\|
\nabla_wK(w^k,\hat{w}^k,\lambda^k)
\right\|^2
+
\frac{p_w}{4\beta}
\left\|
\hat{w}^{k+1}-\hat{w}^k
\right\|^2           +
\frac{1}{8\gamma_{\lambda,k}}
\left\|
\lambda^{k+1}-\lambda^k
\right\|^2                                          \\
&\quad
-
24p_w\beta
\left\|
w^\star(\hat{w}^k)
-
w^\star(\hat{w}^k,\lambda^k)
\right\|^2
-
C_{V,k}\tau^2.
\end{aligned}
\]
By the update rule of $\hat{w}^k$,
\[
\hat{w}^{k+1}-\hat{w}^k
=
\beta(w^{k+1}-\hat{w}^k).
\]
Using $\|a\|^2\ge \frac12\|a-b\|^2-\|b\|^2$, we obtain
\[
\begin{aligned}
\frac{p_w}{4\beta}
\left\|
\hat{w}^{k+1}-\hat{w}^k
\right\|^2
&=
\frac{p_w\beta}{4}
\left\|
w^{k+1}-\hat{w}^k
\right\|^2                                                     \\
&\ge
\frac{p_w\beta}{8}
\left\|
w^k-\hat{w}^k
\right\|^2
-
\frac{p_w\beta}{4}
\left\|
w^{k+1}-w^k
\right\|^2 .
\end{aligned}
\]
By Lemma \ref{Lemma: Deterministic Iterates Gap} and $\gamma_{x,k}= \gamma_{y,k} = \gamma_{w,k}$,
\[
\left\|
w^{k+1}-w^k
\right\|^2
\le
2\gamma_{w,k}^2
\left\|
\nabla_wK(w^k,\hat{w}^k,\lambda^k)
\right\|^2
+
\gamma_{w,k}^2l_{g,2}^2C_{\lambda}^4\tau^2.
\]
Denote $C_{w,k} = \frac{p_{w}\beta\gamma_{w,k}^2l_{g,2}^2C_{\lambda}^4}{4}$, and $C_{\tau,k} = C_{V,k} + C_{w,k}$. Thus,
\[
\begin{aligned}
V_k-V_{k+1}
&\geq
\left(
\frac{\gamma_{w,k}}{8}
-
\frac{p_w\beta\gamma_{w,k}^2}{2}
\right)
\left\|
\nabla_wK(w^k,\hat{w}^k,\lambda^k)
\right\|^2 +
\frac{1}{8\gamma_{\lambda,k}}
\left\|
\lambda^{k+1}-\lambda^k
\right\|^2
+
\frac{p_w\beta}{8}
\left\|
w^k-\hat{w}^k
\right\|^2                                                     \\
&\quad-
24p_w\beta
\left\|
w^\star(\hat{w}^k)
-
w^\star(\hat{w}^k,\lambda^k)
\right\|^2
-
C_{\tau,k}\tau^2.
\end{aligned}
\]
Since
\[
\gamma_{w,k}
\le
\frac{1}{8p_w\beta},
\]
we have
\[
\frac{\gamma_{w,k}}{8}
-
\frac{p_w\beta\gamma_{w,k}^2}{2}
\ge
\frac{\gamma_{w,k}}{16}.
\]
Therefore,
\[
\begin{aligned}
V_k-V_{k+1}
&\geq
\frac{\gamma_{w,k}}{16}
\left\|
\nabla_wK(w^k,\hat{w}^k,\lambda^k)
\right\|^2
+
\frac{1}{8\gamma_{\lambda,k}}
\left\|
\lambda^{k+1}-\lambda^k
\right\|^2  +
\frac{p_w\beta}{8}
\left\|
w^k-\hat{w}^k
\right\|^2                                                   \\
&\quad-
24p_w\beta
\left\|
w^\star(\hat{w}^k)
-
w^\star(\hat{w}^k,\lambda^k)
\right\|^2
-
C_{\tau,k}\tau^2.
\end{aligned}
\]

Next, we convert the actual projected dual step $(\lambda^{k+1}-\lambda^k)$ into the ghost projected dual step
$(\lambda^+(\hat{w}^k)-\lambda^k)$. By non-expansiveness of $\mathcal{P}_\Lambda$,
\[
\begin{aligned}
&\left\|
\lambda^{k+1}-\lambda^+(\hat{w}^k)
\right\|                                                       \\
&=
\left\|
\mathcal{P}_\Lambda\left(
\lambda^k+\gamma_{\lambda,k}
\left(
\nabla_y g(w^k)-p_\lambda\lambda^k
\right)
\right)
-
\mathcal{P}_\Lambda\left(
\lambda^k+\gamma_{\lambda,k}
\left(
\nabla_y g(w^\star(\hat{w}^k,\lambda^k))-p_\lambda\lambda^k
\right)
\right)
\right\|                                                       \\
&\le
\gamma_{\lambda,k}
\left\|
\nabla_y g(w^k)
-
\nabla_y g(w^\star(\hat{w}^k,\lambda^k))
\right\|                                                       \\
&\le
\gamma_{\lambda,k}l_{g,1}
\left\|
w^k-w^\star(\hat{w}^k,\lambda^k)
\right\|                                                       \\
&\le
\frac{\gamma_{\lambda,k}l_{g,1}}{p_w-l_{\mathcal L,1}}
\left\|
\nabla_wK(w^k,\hat{w}^k,\lambda^k)
\right\|.
\end{aligned}
\]
Hence, using $\|a - b\|^2 \geq \frac{1}{2}\|a\|^2 - \|b\|^2$,
\[
\begin{aligned}
\left\|
\lambda^{k+1}-\lambda^k
\right\|^2
&\ge
\frac12
\left\|
\lambda^+(\hat{w}^k)-\lambda^k
\right\|^2
-
\left\|
\lambda^{k+1}-\lambda^+(\hat{w}^k)
\right\|^2                                                     \\
&\ge
\frac12
\left\|
\lambda^+(\hat{w}^k)-\lambda^k
\right\|^2
-
\frac{\gamma_{\lambda,k}^2l_{g,1}^2}{(p_w-l_{\mathcal L,1})^2}
\left\|
\nabla_wK(w^k,\hat{w}^k,\lambda^k)
\right\|^2 .
\end{aligned}
\]
Therefore,
\[
\begin{aligned}
V_k-V_{k+1}
\ge\;&
\left(
\frac{\gamma_{w,k}}{16}
-
\frac{\gamma_{\lambda,k}l_{g,1}^2}{8(p_w-l_{\mathcal L,1})^2}
\right)
\left\|
\nabla_wK(w^k,\hat{w}^k,\lambda^k)
\right\|^2  +
\frac{1}{16\gamma_{\lambda,k}}
\left\|
\lambda^+(\hat{w}^k)-\lambda^k
\right\|^2
+
\frac{p_w\beta}{8}
\left\|
w^k-\hat{w}^k
\right\|^2                                                     \\
&-
24p_w\beta
\left\|
w^\star(\hat{w}^k)
-
w^\star(\hat{w}^k,\lambda^k)
\right\|^2
-
C_{\tau,k}\tau^2.
\end{aligned}
\]
Since
\[
p_w=2l_{\mathcal L,1},
\quad
\gamma_{\lambda,k}
\le\frac{l_{\mathcal L,1}^2}{16l_{g,1}^2}\gamma_{w,k}\leq  \frac{l_{\mathcal L,1}^2}{4l_{g,1}^2}\gamma_{w,k} =
\frac{(p_w-l_{\mathcal L,1})^2}{4l_{g,1}^2}\gamma_{w,k},
\]
we have
\[
\frac{\gamma_{w,k}}{16}
-
\frac{\gamma_{\lambda,k}l_{g,1}^2}{8(p_w-l_{\mathcal L,1})^2}
\ge
\frac{\gamma_{w,k}}{32}.
\]
Thus,
\[
\begin{aligned}
V_k-V_{k+1}
&\geq
\frac{\gamma_{w,k}}{32}
\left\|
\nabla_wK(w^k,\hat{w}^k,\lambda^k)
\right\|^2
+
\frac{1}{16\gamma_{\lambda,k}}
\left\|
\lambda^+(\hat{w}^k)-\lambda^k
\right\|^2    +
\frac{p_w\beta}{8}
\left\|
w^k-\hat{w}^k
\right\|^2                                                 \\
&\quad  -
24p_w\beta
\left\|
w^\star(\hat{w}^k)
-
w^\star(\hat{w}^k,\lambda^k)
\right\|^2
-
C_{\tau,k}\tau^2.
\end{aligned}
\]

Finally, by Lemma \ref{Lemma: Lipschitz continuity of w star},
\[
\begin{aligned}
&\left\|
w^\star(\hat{w}^k)
-
w^\star(\hat{w}^k,\lambda^k)
\right\|^2                                                     \\
&\le
2
\left\|
w^\star(\hat{w}^k)
-
w^\star(\hat{w}^k,\lambda^+(\hat{w}^k))
\right\|^2
+
2
\left\|
w^\star(\hat{w}^k,\lambda^+(\hat{w}^k))
-
w^\star(\hat{w}^k,\lambda^k)
\right\|^2                                                     \\
&\le
2
\left\|
w^\star(\hat{w}^k)
-
w^\star(\hat{w}^k,\lambda^+(\hat{w}^k))
\right\|^2
+
2\gamma_2^2
\left\|
\lambda^+(\hat{w}^k)-\lambda^k
\right\|^2 .
\end{aligned}
\]
Hence,
\[
\begin{aligned}
V_k-V_{k+1}
\ge\;&
\frac{\gamma_{w,k}}{32}
\left\|
\nabla_wK(w^k,\hat{w}^k,\lambda^k)
\right\|^2
+
\left(
\frac{1}{16\gamma_{\lambda,k}}
-
48p_w\beta\gamma_2^2
\right)
\left\|
\lambda^+(\hat{w}^k)-\lambda^k
\right\|^2                                                     \\
&+
\frac{p_w\beta}{8}
\left\|
w^k-\hat{w}^k
\right\|^2
-
48p_w\beta
\left\|
w^\star(\hat{w}^k)
-
w^\star(\hat{w}^k,\lambda^+(\hat{w}^k))
\right\|^2
-
C_{\tau,k}\tau^2.
\end{aligned}
\]
By
\[
\gamma_{\lambda,k}
\le
\frac{1}{1536p_w\beta\gamma_2^2},
\]
we have
\[
\frac{1}{16\gamma_{\lambda,k}}
-
48p_w\beta\gamma_2^2
\ge
\frac{1}{32\gamma_{\lambda,k}}.
\]
Moreover, by the definition
\[
w^+(\hat{w}^k,\lambda^k)
=
w^k-\gamma_{w,k}\nabla_wK(w^k,\hat{w}^k,\lambda^k),
\]
we have
\[
\frac{\gamma_{w,k}}{32}
\left\|
\nabla_wK(w^k,\hat{w}^k,\lambda^k)
\right\|^2
=
\frac{1}{32\gamma_{w,k}}
\left\|
w^k-w^+(\hat{w}^k,\lambda^k)
\right\|^2 .
\]
Therefore,
\[
\begin{aligned}
V_k-V_{k+1}
&\geq
\frac{1}{32\gamma_{w,k}}
\left\|
w^k-w^+(\hat{w}^k,\lambda^k)
\right\|^2
+
\frac{1}{32\gamma_{\lambda,k}}
\left\|
\lambda^+(\hat{w}^k)-\lambda^k
\right\|^2       +
\frac{p_w\beta}{8}
\left\|
w^k-\hat{w}^k
\right\|^2                                              \\
&\quad  -
48p_w\beta
\left\|
w^\star(\hat{w}^k)
-
w^\star(\hat{w}^k,\lambda^+(\hat{w}^k))
\right\|^2
-
C_{\tau,k}\tau^2.
\end{aligned}
\]
This completes the proof.
\end{proof}
\subsection{Oracle Complexity under Deterministic Setting}
\begin{lemma}
\label{Lemma: Error Bound 1}
Under Assumption \ref{Assumption: Bilevel Optimization}, suppose that
\[
\Lambda=\{\lambda:\|\lambda\|\le C_\lambda\},
\]
and let $p_w>l_{\mathcal L,1}$ and $\gamma_{\lambda,k}>0$. Define
\[
\lambda^+(\hat w^k)
=
\mathcal P_\Lambda
\left(
\lambda^k+\gamma_{\lambda,k}\nabla_\lambda\Psi(\hat w^k,\lambda^k)
\right).
\]
Then
\[
\left\|
w^\star(\hat w^k)
-
w^\star(\hat w^k,\lambda^+(\hat w^k))
\right\|
\le
\gamma_3
\left\|
\lambda^+(\hat w^k)-\lambda^k
\right\|,
\]
where
\[
\gamma_3
:=
\frac{
l_{\Psi,1}+\gamma_{\lambda,k}^{-1}
}{
\sqrt{
2(p_w-l_{\mathcal L,1})p_\lambda
+
\frac{(p_w-l_{\mathcal L,1})^2\mu_g^2}
{(p_w+l_{\mathcal L,1})^2}
}
}.
\]
\end{lemma}
\begin{proof}
By Danskin's theorem,
\[
\nabla_\lambda\Psi(\hat w^k,\lambda)
=
\nabla_y g(w^\star(\hat w^k,\lambda))-p_\lambda\lambda.
\]
We first show that $-\nabla_\lambda\Psi(\hat w^k,\lambda)$ is $\left(p_\lambda+
\frac{(p_w-l_{\mathcal L,1})\mu_g^2}{(p_w+l_{\mathcal L,1})^2}\right)$-strongly monotone on $\Lambda$. Since $K(\cdot,\hat w^k,\lambda)$ is
$(p_w-l_{\mathcal L,1})$-strongly convex and
$(p_w+l_{\mathcal L,1})$-smooth with respect to $w$, we have
\[
K(w^*(\hat{w}^k,\lambda_2),\hat w^k,\lambda_1)-K(w^*(\hat{w}^k,\lambda_1),\hat w^k,\lambda_1)
\ge
\frac{p_w-l_{\mathcal L,1}}{2}\|w^*(\hat{w}^k,\lambda_1) - w^*(\hat{w}^k,\lambda_2)\|^2,
\]
and
\[
K(w^*(\hat{w}^k,\lambda_1),\hat w^k,\lambda_2)-K(w^*(\hat{w}^k,\lambda_2),\hat w^k,\lambda_2)
\ge
\frac{p_w-l_{\mathcal L,1}}{2}\|w^*(\hat{w}^k,\lambda_1) - w^*(\hat{w}^k,\lambda_2)\|^2.
\]
Adding the two inequalities, and for any fixed $w$, we have
\[
K(w,\hat w^k,\lambda_1)-K(w,\hat w^k,\lambda_2)
=
(\lambda_1 - \lambda_2)^\top\nabla_y g(w)
-
\frac{p_\lambda}{2}
\left(
\|\lambda_1\|^2-\|\lambda_2\|^2
\right).
\]
Then, we obtain
\[
-\left\langle \nabla_y g(w^*(\hat{w}^k,\lambda_1))-\nabla_y g(w^*(\hat{w}^k,\lambda_2)),\lambda_1 - \lambda_2\right\rangle
\ge
(p_w-l_{\mathcal L,1})\|w^*(\hat{w}^k,\lambda_1) - w^*(\hat{w}^k,\lambda_2)\|^2.
\]
Hence,
\[
\begin{aligned}
&\left\langle
-\nabla_\lambda\Psi(\hat w^k,\lambda_1)-(-\nabla_\lambda\Psi(\hat w^k,\lambda_2)),
\lambda_1 - \lambda_2
\right\rangle\\
&=
p_\lambda\|\lambda_1 - \lambda_2\|^2
-
\left\langle
\nabla_y g(w^*(\hat{w}^k,\lambda_1))-\nabla_y g(w^*(\hat{w}^k,\lambda_2)),\lambda_1 - \lambda_2
\right\rangle                                      \\
&\ge
p_\lambda\|\lambda_1 - \lambda_2\|^2
+
(p_w-l_{\mathcal L,1})\|w^*(\hat{w}^k,\lambda_1) - w^*(\hat{w}^k,\lambda_2)\|^2.
\end{aligned}
\]
By the first-order optimality of
$w^*(\hat{w}^k,\lambda_1)$ and $w^*(\hat{w}^k,\lambda_2)$,
\[
\nabla_w K(w^*(\hat{w}^k,\lambda_1),\hat w^k,\lambda_1)=0,
\quad
\nabla_w K(w^*(\hat{w}^k,\lambda_2),\hat w^k,\lambda_2)=0.
\]
Moreover,
\[
\nabla_w K(w,\hat w^k,\lambda_1)
-
\nabla_w K(w,\hat w^k,\lambda_2)
=
J(w)^\top(\lambda_1 - \lambda_2),
\]
where
\[
J(w):=\nabla_w(\nabla_y g(w)).
\]
Thus,
\[
\nabla_w K(w^*(\hat{w}^k,\lambda_2),\hat w^k,\lambda_1)
=
J(w^*(\hat{w}^k,\lambda_2))^\top(\lambda_1 - \lambda_2).
\]
Since $K(\cdot,\hat w^k,\lambda_1)$ is
$(p_w+l_{\mathcal L,1})$-smooth and
\[
\nabla_w K(w^*(\hat{w}^k,\lambda_1),\hat w^k,\lambda_1)=0,
\]
we have
\begin{align*}
    \|J(w^*(\hat{w}^k,\lambda_2))^\top(\lambda_1 - \lambda_2)\|
&=
\|\nabla_w K(w^*(\hat{w}^k,\lambda_2),\hat w^k,\lambda_1)
-
\nabla_w K(w^*(\hat{w}^k,\lambda_1),\hat w^k,\lambda_1)\| \\
&\le
(p_w+l_{\mathcal L,1})\|w^*(\hat{w}^k,\lambda_1) - w^*(\hat{w}^k,\lambda_2)\|.
\end{align*}
On the other hand, by the strong convexity of $g(x,\cdot)$,
\[
J(w)J(w)^\top\succeq \mu_g^2 I.
\]
Indeed, for any $v\in\mathbb R^{d_y}$,
\[
\|J(w)^\top v\|^2
=
\|\nabla_{xy}^2g(w) v\|^2
+
\|\nabla_{yy}^2g(w)v\|^2
\ge
\mu_g^2\|v\|^2.
\]
Therefore,
\[
\|J(w^*(\hat{w}^k,\lambda_2))^\top(\lambda_1 - \lambda_2)\|
\ge
\mu_g\|\lambda_1 - \lambda_2\|.
\]
Combining the last two displays gives
\[
\|w^*(\hat{w}^k,\lambda_1) - w^*(\hat{w}^k,\lambda_2)\|
\ge
\frac{\mu_g}{p_w+l_{\mathcal L,1}}
\|\lambda_1 - \lambda_2\|.
\]
Consequently,
\[
\left\langle
-\nabla_\lambda\Psi(\hat w^k,\lambda_1)-(-\nabla_\lambda\Psi(\hat w^k,\lambda_2)),
\lambda_1-\lambda_2
\right\rangle
\ge
\left(
p_\lambda
+
\frac{(p_w-l_{\mathcal L,1})\mu_g^2}{(p_w+l_{\mathcal L,1})^2}
\right)
\|\lambda_1-\lambda_2\|^2.
\]
Thus, $-\nabla_\lambda\Psi(\hat w,\lambda)$ is $\left(
p_\lambda
+
\frac{(p_w-l_{\mathcal L,1})\mu_g^2}{(p_w+l_{\mathcal L,1})^2}
\right)$-strongly monotone on $\Lambda$.

Next, let
\[
\hat\lambda^\star(\hat w^k)
=
\arg\max_{\lambda\in\Lambda}\Psi(\hat w^k,\lambda).
\]
Equivalently,
\[
0\in -\nabla_\lambda\Psi(\hat{w}^k,\hat\lambda^\star(\hat{w}^k))+N_\Lambda(\hat\lambda^\star(\hat w^k)).
\]
Since $-\nabla_\lambda\Psi(\hat w,\lambda)$ is $\left(
p_\lambda
+
\frac{(p_w-l_{\mathcal L,1})\mu_g^2}{(p_w+l_{\mathcal L,1})^2}
\right)$-strongly monotone, the error bound for strongly monotone variational inequalities of $-\nabla_\lambda\Psi(\hat w,\lambda)$ gives, for any $\lambda\in\Lambda$,
\[
\|\lambda-\hat\lambda^\star(\hat{w}^k)\|
\le
\frac{1}{p_\lambda
+
\frac{(p_w-l_{\mathcal L,1})\mu_g^2}{(p_w+l_{\mathcal L,1})^2}}
\operatorname{dist}
\left(
0,
-\nabla_\lambda\Psi(\hat{w}^k,\lambda)+N_\Lambda(\lambda)
\right).
\]
We apply this variational inequality at $\lambda=\lambda^+(\hat w^k)$. By the projection optimality condition for
\[
\lambda^+(\hat w^k)
=
\mathcal P_\Lambda
\left(
\lambda^k + \gamma_{\lambda,k}\nabla_\lambda\Psi(\hat{w}^k,\lambda^k)
\right),
\]
we have
\[
\frac{1}{\gamma_{\lambda,k}}
(\lambda^k-\lambda^+(\hat w^k))
+ \nabla_\lambda\Psi(\hat{w}^k,\lambda^k)
\in
N_\Lambda(\lambda^+(\hat w^k)).
\]
Therefore,
\[
-\nabla_\lambda\Psi(\hat{w}^k,\lambda^+(\hat w^k))
+
\frac{1}{\gamma_{\lambda,k}}
(\lambda^k-\lambda^+(\hat w^k))
+ \nabla_\lambda\Psi(\hat{w}^k,\lambda^k)
\in
-\nabla_\lambda\Psi(\hat{w}^k,\lambda^+(\hat w^k))+N_\Lambda(\lambda^+(\hat w^k)).
\]
Hence,
\[
\begin{aligned}
&\operatorname{dist}
\left(
0,
-\nabla_\lambda\Psi(\hat{w}^k,\lambda^+(\hat w^k))+N_\Lambda(\lambda^+(\hat w^k))
\right)                                                       \\
&\le
\left\|
-\nabla_\lambda\Psi(\hat{w}^k,\lambda^+(\hat w^k))-(-\nabla_\lambda\Psi(\hat{w}^k,\lambda^k))
+
\frac{1}{\gamma_{\lambda,k}}
(\lambda^k-\lambda^+(\hat w^k))
\right\|                                                      \\
&\le
\left(
l_{\Psi,1}
+
\gamma_{\lambda,k}^{-1}
\right)
\|\lambda^+(\hat w^k)-\lambda^k\|.
\end{aligned}
\]
We now use the joint primal--dual coercivity retained in the preceding
calculation. The optimality conditions give 
\[
\nabla_\lambda\Psi
(\hat w^k,\hat\lambda^\star(\hat w^k))
\in N_\Lambda(\hat\lambda^\star(\hat w^k)),\quad \frac{\lambda^k-\lambda^+(\hat w^k)}{\gamma_{\lambda,k}}
+
\nabla_\lambda\Psi(\hat w^k,\lambda^k)
\in N_\Lambda(\lambda^+(\hat w^k)).
\]
By monotonicity of
$N_\Lambda$ and the joint monotonicity inequality proved above,
\begin{align*}
&
p_\lambda
\left\|
\lambda^+(\hat w^k)-\hat\lambda^\star(\hat w^k)
\right\|^2
+
(p_w-l_{\mathcal L,1})
\left\|
w^\star(\hat w^k,\lambda^+(\hat w^k))
-
w^\star(\hat w^k,\hat\lambda^\star(\hat w^k))
\right\|^2                                                     \\
&\le
\left\langle
-\nabla_\lambda\Psi
(\hat w^k,\lambda^+(\hat w^k))+\frac{\lambda^k-\lambda^+(\hat w^k)}{\gamma_{\lambda,k}}
+
\nabla_\lambda\Psi(\hat w^k,\lambda^k),\,
\lambda^+(\hat w^k)-\hat\lambda^\star(\hat w^k)
\right\rangle                                                  \\
&\le
\left(
l_{\Psi,1}+\gamma_{\lambda,k}^{-1}
\right)
\left\|
\lambda^+(\hat w^k)-\lambda^k
\right\|
\left\|
\lambda^+(\hat w^k)-\hat\lambda^\star(\hat w^k)
\right\|.
\end{align*}
Furthermore, the lower response bound proved above implies
\[
\left\|
w^\star(\hat w^k,\lambda^+(\hat w^k))
-
w^\star(\hat w^k,\hat\lambda^\star(\hat w^k))
\right\|
\ge
\frac{\mu_g}{p_w+l_{\mathcal L,1}}
\left\|
\lambda^+(\hat w^k)-\hat\lambda^\star(\hat w^k)
\right\|.
\]
Then, we have
\begin{align*}
&\left(p_\lambda
\left\|
\lambda^+(\hat w^k)-\hat\lambda^\star(\hat w^k)
\right\|^2
+
(p_w-l_{\mathcal L,1})
\left\|
w^\star(\hat w^k,\lambda^+(\hat w^k))
-
w^\star(\hat w^k,\hat\lambda^\star(\hat w^k))
\right\|^2\right)^2 \\  
&\ge
\left(
2(p_w-l_{\mathcal L,1})p_\lambda
+
\frac{(p_w-l_{\mathcal L,1})^2\mu_g^2}
{(p_w+l_{\mathcal L,1})^2}
\right)\left\|
\lambda^+(\hat w^k)-\hat\lambda^\star(\hat w^k)
\right\|^2\left\|
w^\star(\hat w^k,\lambda^+(\hat w^k))
-
w^\star(\hat w^k,\hat\lambda^\star(\hat w^k))
\right\|^2.
\end{align*}
Combining two inequalities and
dividing by
$\|\lambda^+(\hat w^k)-\hat\lambda^\star(\hat w^k)\|$ when it is nonzero
gives
\[
\left\|
w^\star(\hat w^k)
-
w^\star(\hat w^k,\lambda^+(\hat w^k))
\right\|
\le
\frac{
l_{\Psi,1}+\gamma_{\lambda,k}^{-1}
}{
\sqrt{
2(p_w-l_{\mathcal L,1})p_\lambda
+
\frac{(p_w-l_{\mathcal L,1})^2\mu_g^2}
{(p_w+l_{\mathcal L,1})^2}
}
}
\left\|
\lambda^+(\hat w^k)-\lambda^k
\right\|.
\]
This proves the claim.
\end{proof}

\paragraph{Proof of Proposition \ref{Proposition: Recursion of Lyapunov Fucntion}.}
By Proposition \ref{Proposition: Deterministic Lyapunov Function Descent}, we know that
    \begin{align*}
        V_k - V_{k+1}
        &\geq \frac{1}{32\gamma_{w,k}}\|w^k - w^+(\hat{w}^k,\lambda^k)\|^2  + \frac{1}{32\gamma_{\lambda,k}}\|\lambda^+(\hat{w}^{k}) - \lambda^k\|^2 + \frac{p_w\beta}{8}\|w^{k} - \hat{w}^k\|^2\\
        &\quad - 48p_w\beta \|w^*(\hat{w}^{k}) - w^*(\hat{w}^{k},\lambda^{+}(\hat{w}^k))\|^2 - C_{\tau,k}\tau^2.
    \end{align*}
    By Lemma \ref{Lemma: Error Bound 1}, and choosing $\beta \leq \frac{1}{3072p_w\gamma_{\lambda,k}\gamma_3^2}$, we have
    \begin{align*}
    48p_w\beta \|w^*(\hat{w}^{k}) - w^*(\hat{w}^{k},\lambda^{+}(\hat{w}^k))\|^2 \leq 48p_w\beta\gamma_3^2\|\lambda^+(\hat{w}^{k}) - \lambda^k\|^2\leq \frac{1}{64\gamma_{\lambda,k}}\|\lambda^+(\hat{w}^{k}) - \lambda^k\|^2,
    \end{align*}
    which implies the conclusion.

\paragraph{Proof of Lemma \ref{Lemma: Deterministic Stationary Point Bridge}.}
By the definition of $w^+(\hat{w}^k,\lambda^k)$,
\[
w^+(\hat{w}^k,\lambda^k)
=
w^k-\gamma_{w,k}\nabla_wK(w^k,\hat{w}^k,\lambda^k).
\]
Hence
\[
\|\nabla_wK(w^k,\hat{w}^k,\lambda^k)\|
\le
\gamma_{w,k}^{-1}
\|w^k-w^+(\hat{w}^k,\lambda^k)\|
\le
\gamma_{w,k}^{-1}\epsilon_{\mathcal{L}}.
\]
Since
\[
\nabla_wK(w^k,\hat{w}^k,\lambda^k)
=
\nabla_w\mathcal{L}(w^k,\lambda^k)+p_w(w^k-\hat{w}^k),
\]
we have
\[
\|\nabla_w\mathcal{L}(w^k,\lambda^k)\|
\le
\|\nabla_wK(w^k,\hat{w}^k,\lambda^k)\|
+
p_w\|w^k-\hat{w}^k\|
\le
\left(\gamma_{w,k}^{-1}+\bar{\kappa}_yp_w\right)\epsilon_{\mathcal{L}}.
\]

Next, we prove that the projected dual step is inactive. From the bound on
$\nabla_y\mathcal{L}(x^k,y^k,\lambda^k)$, we have
\[
\left\|
\nabla_y f(x^k,y^k)+\nabla_{yy}^2g(x^k,y^k)\lambda^k
\right\|
\le
\left(\gamma_{w,k}^{-1}+\bar{\kappa}_yp_w\right)\epsilon_{\mathcal{L}}.
\]
Using $\nabla_{yy}^2g(x^k,y^k)\succeq \mu_g I$ and
$\|\nabla_y f(x^k,y^k)\|\le l_{f,0}$, we obtain
\[
\mu_g\|\lambda^k\|
\le
\|\nabla_{yy}^2g(x^k,y^k)\lambda^k\|
\le
l_{f,0}
+
\left(\gamma_{w,k}^{-1}+\bar{\kappa}_yp_w\right)\epsilon_{\mathcal{L}}.
\]
By the choice of $\epsilon_{\mathcal{L}}$,
\[
\|\lambda^k\|
\le
\frac{l_{f,0}}{\mu_g}
+
\frac14
\left(
C_\lambda-\frac{l_{f,0}}{\mu_g}
\right)
=
C_\lambda
-
\frac34
\left(
C_\lambda-\frac{l_{f,0}}{\mu_g}
\right).
\]
Thus $\lambda^k$ lies in the interior of $\Lambda$ with distance at least $\frac34
\left(
C_\lambda-\frac{l_{f,0}}{\mu_g}
\right)$ from the boundary. Since
\[
\|\lambda^+(\hat{w}^k)\| - \|\lambda^k\| \leq \|\lambda^+(\hat{w}^k)-\lambda^k\|
\le
\bar{\kappa}_y\epsilon_{\mathcal{L}}
\le
\frac14
\left(
C_\lambda-\frac{l_{f,0}}{\mu_g}
\right),
\]
we have
\[
\|\lambda^+(\hat{w}^k)\| \leq C_{\lambda} - \frac{1}{2}\left(
C_\lambda-\frac{l_{f,0}}{\mu_g}
\right) < C_{\lambda},
\]
which implies the projection in the definition of $\lambda^+(\hat{w}^k)$ is inactive. Therefore,
\[
\lambda^+(\hat{w}^k)
=
\lambda^k
+
\gamma_{\lambda,k}\nabla_\lambda\Psi(\hat{w}^k,\lambda^k).
\]
By Danskin's theorem,
\[
\nabla_\lambda\Psi(\hat{w}^k,\lambda^k)
=
\nabla_y g(w^\star(\hat{w}^k,\lambda^k))-p_\lambda\lambda^k.
\]
Hence
\[
\left\|
\nabla_y g(w^\star(\hat{w}^k,\lambda^k))-p_\lambda\lambda^k
\right\|
=
\gamma_{\lambda,k}^{-1}
\|\lambda^+(\hat{w}^k)-\lambda^k\|
\le
\bar{\kappa}_y\gamma_{\lambda,k}^{-1}\epsilon_{\mathcal{L}}.
\]
It remains to transfer this bound from $w^\star(\hat{w}^k,\lambda^k)$ to $w^k$.
By Lemma \ref{Lemma: SC property of K} and the optimality condition
\[
\nabla_wK(w^\star(\hat{w}^k,\lambda^k),\hat{w}^k,\lambda^k)=0,
\]
we have
\[
\begin{aligned}
\|w^k-w^\star(\hat{w}^k,\lambda^k)\|
&\le
\frac{1}{p_w-l_{\mathcal L,1}}
\|\nabla_wK(w^k,\hat{w}^k,\lambda^k)\|                                      \\
&\le
\frac{1}{(p_w-l_{\mathcal L,1})\gamma_{w,k}}\epsilon_{\mathcal{L}}.
\end{aligned}
\]
Therefore,
\[
\begin{aligned}
\|\nabla_y g(w^k)-p_\lambda\lambda^k\|
&\le
\|\nabla_y g(w^k)-\nabla_y g(w^\star(\hat{w}^k,\lambda^k))\|
+
\|\nabla_y g(w^\star(\hat{w}^k,\lambda^k))-p_\lambda\lambda^k\|              \\
&\le
\frac{l_{g,1}}{(p_w-l_{\mathcal L,1})\gamma_{w,k}}\epsilon_{\mathcal{L}}
+
\bar{\kappa}_y\gamma_{\lambda,k}^{-1}\epsilon_{\mathcal{L}}.
\end{aligned}
\]
Since $\|\lambda^k\|\le C_\lambda$, we get
\begin{align*}
    \|\nabla_\lambda \mathcal{L}(x^k,y^k,\lambda^k)\|
&=
\|\nabla_y g(x^k,y^k)\| \\
&\le
\|\nabla_y g(w^k)-p_\lambda\lambda^k\|
+
p_\lambda\|\lambda^k\| \\
&\leq \left(
\frac{l_{g,1}}{(p_w-l_{\mathcal L,1})\gamma_{w,k}}
+
\bar{\kappa}_y\gamma_{\lambda,k}^{-1}
\right)\epsilon_{\mathcal L}
+
p_\lambda C_\lambda
.
\end{align*}
This completes the proof.

\paragraph{Proof of Theorem \ref{Theorem: Deterministic Sample Complexity}.}

We first verify that the explicit parameters in the theorem satisfy the
conditions of Proposition~\ref{Proposition: Deterministic Lyapunov Function Descent}
and Lemma~\ref{Lemma: Error Bound 1}.
Under $C_\lambda=2\bar\kappa_y$, the basic smoothness bound is
\begin{align*}
\sqrt{2}\bigl(l_{f,1}+2\bar\kappa_y l_{g,2}\bigr) \le
\sqrt{2}\bar L(1+2\bar\kappa_y) \le
3\sqrt{2}\bar L\bar\kappa_y
<
8\bar L\bar\kappa_y
=:
L_0.
\end{align*}
Thus, in the generic results above we choose the admissible smoothness
upper bound $l_{\mathcal L,1}:=L_0$. In particular, the theorem's choice
\[
p_w=16\bar L\bar\kappa_y=2L_0
\]
has the required form.

Define $\epsilon_{\mathcal L} = \frac{\epsilon}{2^{11}(1+\bar L)\bar \kappa_y^3}$. The choice
$
p_\lambda
=
\frac18
\min\left\{
\mu_g,\frac{\epsilon_{\mathcal L}}{\bar\kappa_y}
\right\}
$
satisfies
\[
p_\lambda
\le
\frac{\mu_g}{8}
=
\frac{\bar L}{8\bar\kappa_y},
\qquad
p_\lambda C_\lambda
\le
\frac{\epsilon_{\mathcal L}}{4}.
\]
Moreover, the auxiliary constants used in the appendix satisfy
\[
\gamma_1
=
\frac{p_w}{p_w-L_0}
=
2,
\quad
\gamma_2
=
\frac{l_{g,1}}{p_w-L_0}
=
\frac{l_{g,1}}{L_0}
\le
\frac{1}{8\bar\kappa_y}.
\]
Also,
\begin{align*}
l_{\Psi,1}
=
p_\lambda+l_{g,1}\gamma_2 =
p_\lambda+\frac{l_{g,1}^2}{L_0} \le
\frac{\bar L}{8\bar\kappa_y}
+
\frac{\bar L}{8\bar\kappa_y}
=
\frac{\bar L}{4\bar\kappa_y}.
\end{align*}

For the constant $\gamma_3$ in
Lemma~\ref{Lemma: Error Bound 1}, its denominator satisfies
\begin{align*}
\sqrt{
2(p_w-L_0)p_\lambda
+
\frac{(p_w-L_0)^2\mu_g^2}{(p_w+L_0)^2}
}
\ge
\frac{(p_w-L_0)\mu_g}{p_w+L_0}
=
\frac{\mu_g}{3}
=
\frac{\bar L}{3\bar\kappa_y}.
\end{align*}
Since
$
\gamma_\lambda^{-1}
=
\frac{2^8\bar L}{\bar\kappa_y}$,
we obtain
\begin{align*}
\gamma_3
=
\frac{
l_{\Psi,1}+\gamma_\lambda^{-1}
}{
\sqrt{
2(p_w-L_0)p_\lambda
+
\frac{(p_w-L_0)^2\mu_g^2}{(p_w+L_0)^2}
}
} \le
\frac{
\frac{\bar L}{4\bar\kappa_y}
+
\frac{2^8\bar L}{\bar\kappa_y}
}{
\bar L/(3\bar\kappa_y)
} =
\frac{3075}{4}.
\end{align*}

We next verify the stepsize conditions in Proposition~\ref{Proposition: Deterministic Lyapunov Function Descent}. The primal stepsize satisfies
\[
\gamma_w
=
\frac{1}{2^8\bar L\bar\kappa_y}
\le
\frac{1}{8(p_w+L_0)}
=
\frac{1}{192\bar L\bar\kappa_y}.
\]
Since $\beta\le 1$, it also satisfies
$
\gamma_w
\le
\frac{1}{8p_w\beta}.
$

For the first multiplier-stepsize condition,
\begin{align*}
2\left(
2l_{\Psi,1}
+p_\lambda
+\frac{l_{g,1}^2}{p_w+L_0}
\right)
\le
2\left(
\frac{\bar L}{2\bar\kappa_y}
+
\frac{\bar L}{8\bar\kappa_y}
+
\frac{\bar L}{24\bar\kappa_y}
\right) =
\frac{4\bar L}{3\bar\kappa_y}.
\end{align*}
Hence,
\[
\gamma_\lambda \le \frac{3\bar\kappa_y}{4\bar L} \le \frac{1}{
2\left(
2l_{\Psi,1}
+p_\lambda+
l_{g,1}^2/(p_w+L_0)
\right)}.
\]
For the second condition,
\[
\frac{L_0^2}{16l_{g,1}^2}\gamma_w
\ge
\frac{\bar\kappa_y}{64\bar L}
\ge
\frac{\bar\kappa_y}{2^8\bar L}
=
\gamma_\lambda.
\]
For the third condition, as $\beta\le 2/3$,
\[
\frac{1}{1536p_w\beta\gamma_2^2}
\ge
\frac{\bar\kappa_y}{384\bar L\beta}
\ge
\frac{\bar\kappa_y}{2^8\bar L}
=
\gamma_\lambda.
\]

The parameter $\beta$ satisfies
\[
\beta
=
\frac{1}{2^{30}\bar\kappa_y^2}
\le
\frac{\sqrt{5}-1}{96}
=
\frac{\sqrt{5}-1}{48\gamma_1}.
\]
Furthermore,
\begin{align*}
3072p_w\gamma_\lambda\gamma_3^2
\le
3072
\left(16\bar L\bar\kappa_y\right)
\left(\frac{\bar\kappa_y}{2^8\bar L}\right)
\left(\frac{3075}{4}\right)^2 =
12(3075)^2\bar\kappa_y^2 <
2^{30}\bar\kappa_y^2
=
\beta^{-1}.
\end{align*}
Thus,
$
\beta
\le
\frac{1}{3072p_w\gamma_\lambda\gamma_3^2}.
$ Therefore, all the conditions of
Proposition~\ref{Proposition: Deterministic Lyapunov Function Descent}, Lemma~\ref{Lemma: Error Bound 1} and Proposition~\ref{Proposition: Recursion of Lyapunov Fucntion} are satisfied.

Because the stepsizes are constant, $C_{\tau, k} \equiv C_\tau$ in Proposition~\ref{Proposition: Deterministic Lyapunov Function Descent} is a constant.
Proposition~\ref{Proposition: Deterministic Lyapunov Function Descent}
and Lemma~\ref{Lemma: Error Bound 1} give
\begin{align}
V_k-V_{k+1}
\ge\;
\frac{1}{32\gamma_w}
\left\|
w^k-w^+(\hat w^k,\lambda^k)
\right\|^2
+
\frac{1}{64\gamma_\lambda}
\left\|
\lambda^+(\hat w^k)-\lambda^k
\right\|^2
+
\frac{p_w\beta}{8}
\left\|
w^k-\hat w^k
\right\|^2
-
C_\tau\tau^2.
\label{eq:explicit-deterministic-descent}
\end{align}

Define
\[
a
:=
\min\left\{
\frac{1}{32\gamma_w},
\frac{\bar\kappa_y^2}{64\gamma_\lambda},
\frac{\bar\kappa_y^2p_w\beta}{8}
\right\}.
\]
Under the explicit parameter choices,
\[
\frac{1}{32\gamma_w}
=
8\bar L\bar\kappa_y,
\quad
\frac{\bar\kappa_y^2}{64\gamma_\lambda}
=
4\bar L\bar\kappa_y,
\quad
\frac{\bar\kappa_y^2p_w\beta}{8}
=
\frac{\bar L\bar\kappa_y}{2^{29}}.
\]
Since $\bar\kappa_y\ge1$, it follows that
\[
a
=
\frac{\bar L\bar\kappa_y}{2^{29}}.
\]

Then \eqref{eq:explicit-deterministic-descent} implies
\[
V_k-V_{k+1}
\ge
a\max\left\{
\|w^k-w^+(\hat w^k,\lambda^k)\|^2,\,
\bar\kappa_y^{-2}
\|\lambda^+(\hat w^k)-\lambda^k\|^2,\,
\bar\kappa_y^{-2}
\|w^k-\hat w^k\|^2
\right\}
-C_\tau\tau^2.
\]

We next bound $C_\tau$. Using
$l_{g,2}\le\bar L$, $C_\lambda=2\bar\kappa_y$, and $\beta\le1$,
we obtain
\begin{align*}
C_\tau
\le
\left(
\frac{1}{2^{10}\bar L\bar\kappa_y}
+
\frac{3}{2^{13}\bar L\bar\kappa_y}
+
\frac{1}{2^{14}\bar L\bar\kappa_y}
\right)
\bar L^2(2\bar\kappa_y)^4 =
\frac{23}{2^{10}}
\bar L\bar\kappa_y^3.
\end{align*}

Summing the descent inequality from $k=0$ to $K-1$ and using
$V_K\ge\underline f$ gives
\[
a\sum_{k=0}^{K-1}
\max\left\{
\|w^k-w^+(\hat w^k,\lambda^k)\|^2,\,
\bar\kappa_y^{-2}
\|\lambda^+(\hat w^k)-\lambda^k\|^2,\,
\bar\kappa_y^{-2}
\|w^k-\hat w^k\|^2
\right\}
\le
V_0-\underline f
+
KC_\tau\tau^2.
\]
Consequently, there exists an index
$k_\star\in\{0,\ldots,K-1\}$ such that
\[
\max\left\{
\|w^{k_\star}-w^+(\hat w^{k_\star},\lambda^{k_\star})\|^2,\,
\bar\kappa_y^{-2}
\|\lambda^+(\hat w^{k_\star})-\lambda^{k_\star}\|^2,\,
\bar\kappa_y^{-2}
\|w^{k_\star}-\hat w^{k_\star}\|^2
\right\}
\le
\frac{V_0-\underline f}{aK}
+
\frac{C_\tau}{a}\tau^2.
\]

Denote
$
\Delta_V:=V_0-\underline f$.
With
\[
K
\ge
\frac{2^{30}\Delta_V}
{\bar L\bar\kappa_y\epsilon_{\mathcal L}^2},
\]
we have
$
\frac{\Delta_V}{aK}
\le
\frac12\epsilon_{\mathcal L}^2$.
Moreover, because
$
\tau
=
\frac{\epsilon_{\mathcal L}}
{2^{13}\bar\kappa_y}$,
we have
\begin{align*}
\frac{C_\tau}{a}\tau^2
\le
\frac{
(23/2^{10})\bar L\bar\kappa_y^3
}{
\bar L\bar\kappa_y/2^{29}
}
\frac{\epsilon_{\mathcal L}^2}
{2^{26}\bar\kappa_y^2} =
\frac{23}{128}\epsilon_{\mathcal L}^2.
\end{align*}
Therefore,
\[
\begin{aligned}
\max\biggl\{
\|w^{k_\star}-w^+(\hat w^{k_\star},\lambda^{k_\star})\|^2, \bar\kappa_y^{-2}
\|\lambda^+(\hat w^{k_\star})-\lambda^{k_\star}\|^2, \bar\kappa_y^{-2}
\|w^{k_\star}-\hat w^{k_\star}\|^2
\biggr\}
&\le
\left(
\frac12+\frac{23}{128}
\right)
\epsilon_{\mathcal L}^2\\
&=
\frac{87}{128}\epsilon_{\mathcal L}^2 \\
&<
\epsilon_{\mathcal L}^2.
\end{aligned}
\]
It follows that
\[
\|w^{k_\star}-w^+(\hat w^{k_\star},\lambda^{k_\star})\|
\le
\epsilon_{\mathcal L}, \quad
\|\lambda^+(\hat w^{k_\star})-\lambda^{k_\star}\|
\le
\bar\kappa_y\epsilon_{\mathcal L},
\quad
\|w^{k_\star}-\hat w^{k_\star}\|
\le
\bar\kappa_y\epsilon_{\mathcal L}.
\]

We now verify the interiority conditions in
Lemma~\ref{Lemma: Deterministic Stationary Point Bridge}. Since
$l_{f,0}\le\bar L$,
\[
C_\lambda-\frac{l_{f,0}}{\mu_g}
=
2\bar\kappa_y-\frac{l_{f,0}}{\mu_g}
\ge
\bar\kappa_y.
\]
Furthermore,
\begin{align*}
\gamma_w^{-1}+\bar\kappa_yp_w
=
2^8\bar L\bar\kappa_y
+
16\bar L\bar\kappa_y^2 \le
272\bar L\bar\kappa_y^2.
\end{align*}
Using
$
\epsilon_{\mathcal L}
=
\frac{\epsilon}
{2^{11}(1+\bar L)\bar\kappa_y^3}$ and
$0<\epsilon\le1$,
we obtain
\begin{align*}
\left(
\gamma_w^{-1}+\bar\kappa_yp_w
\right)
\epsilon_{\mathcal L}
\le
\frac{17}{128}
\frac{\bar L}{1+\bar L}
\frac{\epsilon}{\bar\kappa_y} <
\frac{\bar L}{4} =
\frac{\mu_g\bar\kappa_y}{4} \le
\frac{\mu_g}{4}
\left(
C_\lambda-\frac{l_{f,0}}{\mu_g}
\right).
\end{align*}
Similarly,
\[
\bar\kappa_y\epsilon_{\mathcal L}
=
\frac{\epsilon}
{2^{11}(1+\bar L)\bar\kappa_y^2}
<
\frac14
\left(
C_\lambda-\frac{l_{f,0}}{\mu_g}
\right).
\]
Hence, Lemma~\ref{Lemma: Deterministic Stationary Point Bridge}
applies at $k_\star$.

For the primal component, it gives
\begin{align*}
\|\nabla_w\mathcal L(w^{k_\star},\lambda^{k_\star})\|
\le
\left(
\gamma_w^{-1}+\bar\kappa_yp_w
\right)
\epsilon_{\mathcal L} \le
272\bar L\bar\kappa_y^2
\epsilon_{\mathcal L}.
\end{align*}
For the multiplier component,
\begin{align*}
\frac{l_{g,1}}{(p_w-L_0)\gamma_w}
+
\bar\kappa_y\gamma_\lambda^{-1}
=
32l_{g,1}+2^8\bar L \le
288\bar L.
\end{align*}
Since $p_\lambda C_\lambda\le\epsilon_{\mathcal L}/4$,
Lemma~\ref{Lemma: Deterministic Stationary Point Bridge} yields
\[
\|\nabla_\lambda
\mathcal L(w^{k_\star},\lambda^{k_\star})\|
\le
\left(
288\bar L+\frac14
\right)
\epsilon_{\mathcal L}.
\]

It remains to translate these bounds to stationarity of the bilevel
hyper-objective. The constant in
Theorem~\ref{Theorem: Approximation of Hypergradient} satisfies
\begin{align*}
L_{\bar F,y}
&=
l_{f,1}(1+\kappa_y)
+
l_{f,0}l_{g,2}
\left(
\frac{1}{\mu_g}
+
\frac{l_{g,1}}{\mu_g^2}
\right) \le
\bar L(1+\bar\kappa_y)^2 \le
4\bar L\bar\kappa_y^2.
\end{align*}
Consequently,
$
\frac{L_{\bar F,y}}{\mu_g}
\le
4\bar\kappa_y^3$.
Using
Theorem~\ref{Theorem: Approximation of Hypergradient}, we obtain
\begin{align*}
\|\nabla F(x^{k_\star})\|
&\le
\|\nabla_x\mathcal L(w^{k_\star},\lambda^{k_\star})\|
+
\kappa_y
\|\nabla_y\mathcal L(w^{k_\star},\lambda^{k_\star})\|
+
\frac{L_{\bar F,y}}{\mu_g}
\|\nabla_\lambda
\mathcal L(w^{k_\star},\lambda^{k_\star})\|\\
&\le
(1+\bar\kappa_y)
272\bar L\bar\kappa_y^2
\epsilon_{\mathcal L}
+
4\bar\kappa_y^3
\left(
288\bar L+\frac14
\right)
\epsilon_{\mathcal L}\\
&\le
(1696\bar L+1)
\bar\kappa_y^3
\epsilon_{\mathcal L} \le
2^{11}(1+\bar L)
\bar\kappa_y^3
\epsilon_{\mathcal L} =
\epsilon.
\end{align*}
Thus, $x^{k_\star}$ is an $\epsilon$-stationary point of $F$.
Finally, substituting
$
\epsilon_{\mathcal L}
=
\frac{\epsilon}
{2^{11}(1+\bar L)\bar\kappa_y^3}
$
into the iteration bound gives
\begin{align*}
K
\le
1+
\frac{
2^{52}(1+\bar L)^2\Delta_V
}{
\bar L
}
\bar\kappa_y^5\epsilon^{-2}.
\end{align*}
This completes the proof.

\section{Stochastic NC-SC Bilevel Optimization}
\subsection{Filtration}

For the stochastic analysis, define the pre-sampling filtration
\[
    \mathcal F_k^-
    :=
    \sigma\left(
        w^0,\hat w^0,\lambda^0,
        \{\mathcal B_f^t,\mathcal B_g^t\}_{t=0}^{k-1}
    \right).
\]
Thus, \(w^k,\hat w^k,\lambda^k\) and the stepsizes are
\(\mathcal F_k^-\)-measurable. The fresh mini-batches
\(\mathcal B_f^k\) and \(\mathcal B_g^k\) are independent of
\(\mathcal F_k^-\). After sampling and performing the \(k\)-th update, define
\[
    \mathcal F_k^+
    :=
    \sigma(\mathcal F_k^-,\mathcal B_f^k,\mathcal B_g^k),
    \quad
    \mathcal F_{k+1}^-=\mathcal F_k^+.
\]
Then, \(w^{k+1},\hat w^{k+1},\lambda^{k+1}\) are
\(\mathcal F_k^+\)-measurable. Moreover,
\[
\mathbb E\!\left[
\nabla f(x^k,y^k;\mathcal B_f^k)
\mid \mathcal F_k^-
\right]
=
\nabla f(x^k,y^k),
\quad
\mathbb E\!\left[
\nabla g(x^k,y^k;\mathcal B_g^k)
\mid \mathcal F_k^-
\right]
=
\nabla g(x^k,y^k).
\]
\subsection{Properties of Stochastic Finite Difference Approximation}
\begin{lemma}\label{Lemma: Stochastic Finite Difference Estimator}
Define the stochastic finite-difference approximation of  $\nabla_{xy}^2g(x,y)\lambda$ and $\nabla_{yy}^2g(x,y)\lambda$ as
\[
\tilde{\nabla}_{xy}^2 g_{\lambda}(x,y;\mathcal B_g)
    :=
    \frac{1}{B}\sum_{i=1}^B
    \frac{
    \nabla_x g(x,y+\tau\lambda;\zeta_{i})
    -
    \nabla_x g(x,y-\tau\lambda;\zeta_{i})
    }{2\tau},
\]
and
\[
\tilde{\nabla}_{yy}^2 g_{\lambda}(x,y;\mathcal B_g)
    :=
    \frac{1}{B}\sum_{i=1}^B
    \frac{
    \nabla_y g(x,y+\tau\lambda;\zeta_{i})
    -
    \nabla_y g(x,y-\tau\lambda;\zeta_{i})
    }{2\tau},
\]
where $\mathcal{B}_f = \{\xi_1,\ldots,\xi_B\}$, $\mathcal{B}_g = \{\zeta_1,\ldots,\zeta_B\}$ are mini-batches with the same cardinality $B$.

Under Assumptions \ref{Assumption: Bilevel Optimization} and \ref{Assumption: Stochastic First-order Oracle}, with
\[
\lambda\in\Lambda=\{\lambda:\|\lambda\|\le C_\lambda\},
\]
we have
\begin{equation}\label{Lemma: Stoc FD 1}
\mathbb{E}\left[
\left\|
\nabla_{xy}^2 g(x,y)\lambda
-
\tilde{\nabla}_{xy}^2 g_{\lambda}(x,y;\mathcal{B}_g)
\right\|^2
\right]
\le
\frac{l_{g,2}^2C_{\lambda}^4}{4}\tau^2
+
\frac{\sigma_g^2}{B\tau^2},
\end{equation}
and
\begin{equation}\label{Lemma: Stoc FD 2}
\mathbb{E}\left[
\left\|
\nabla_{yy}^2 g(x,y)\lambda
-
\tilde{\nabla}_{yy}^2 g_{\lambda}(x,y;\mathcal{B}_g)
\right\|^2
\right]
\le
\frac{l_{g,2}^2C_{\lambda}^4}{4}\tau^2
+
\frac{\sigma_g^2}{B\tau^2}.
\end{equation}
\end{lemma}
\begin{proof}
We first prove \eqref{Lemma: Stoc FD 1}. By the definition of
$\tilde{\nabla}_{xy}^2 g_{\lambda}(x,y;\mathcal{B}_g)$ and the unbiasedness of
the stochastic gradient oracle,
\[
\begin{aligned}
\mathbb{E}\left[
\left\|
\nabla_{xy}^2 g(x,y)\lambda
-
\tilde{\nabla}_{xy}^2 g_{\lambda}(x,y;\mathcal{B}_g)
\right\|^2
\right]
&=
\left\|
\nabla_{xy}^2 g(x,y)\lambda
-
\mathbb{E}\left[
\tilde{\nabla}_{xy}^2 g_{\lambda}(x,y;\mathcal{B}_g)
\right]
\right\|^2                                                     \\
&\quad
+
\mathbb{E}\left[
\left\|
\tilde{\nabla}_{xy}^2 g_{\lambda}(x,y;\mathcal{B}_g)
-
\mathbb{E}\left[
\tilde{\nabla}_{xy}^2 g_{\lambda}(x,y;\mathcal{B}_g)
\right]
\right\|^2
\right].
\end{aligned}
\]
For the bias term, we have
\[
\frac{
\nabla_x g(x,y+\tau\lambda)
-
\nabla_x g(x,y-\tau\lambda)
}{
2\tau
}
=
\frac{1}{2\tau}
\int_{-\tau}^{\tau}
\nabla_{xy}^2g(x,y+s\lambda)\lambda\,ds .
\]
Therefore,
\[
\begin{aligned}
\left\|
\nabla_{xy}^2 g(x,y)\lambda
-
\mathbb{E}\left[
\tilde{\nabla}_{xy}^2 g_{\lambda}(x,y;\mathcal{B}_g)
\right]
\right\|
&=
\left\|
\frac{1}{2\tau}
\int_{-\tau}^{\tau}
\left(
\nabla_{xy}^2g(x,y)-
\nabla_{xy}^2g(x,y+s\lambda)
\right)
\lambda\,ds
\right\|      \\
&\le
\frac{1}{2\tau}
\int_{-\tau}^{\tau}
l_{g,2}|s|\|\lambda\|^2\,ds                                    \\
&=
\frac{l_{g,2}}{2}\tau\|\lambda\|^2                              \\
&\le
\frac{l_{g,2}C_\lambda^2}{2}\tau .
\end{aligned}
\]
Thus,
\[
\left\|
\nabla_{xy}^2 g(x,y)\lambda
-
\mathbb{E}\left[
\tilde{\nabla}_{xy}^2 g_{\lambda}(x,y;\mathcal{B}_g)
\right]
\right\|^2
\le
\frac{l_{g,2}^2C_\lambda^4}{4}\tau^2.
\]
For the variance term, by the independence of samples in $\mathcal{B}_g$,
\[
\begin{aligned}
&\mathbb{E}\left[
\left\|
\tilde{\nabla}_{xy}^2 g_{\lambda}(x,y;\mathcal{B}_g)
-
\mathbb{E}\left[
\tilde{\nabla}_{xy}^2 g_{\lambda}(x,y;\mathcal{B}_g)
\right]
\right\|^2
\right]                                                        \\
&=
\frac{1}{B}
\mathbb{E}\left[
\left\|
\frac{
\left(
\nabla_x g(x,y+\tau\lambda;\zeta)
-
\nabla_x g(x,y+\tau\lambda)
\right)
-
\left(
\nabla_x g(x,y-\tau\lambda;\zeta)
-
\nabla_x g(x,y-\tau\lambda)
\right)
}{
2\tau
}
\right\|^2
\right]                                                        \\
&\le
\frac{1}{B}
\cdot
\frac{1}{4\tau^2}
\cdot
\left(2\mathbb{E}\left[
\left\|
\nabla_x g(x,y+\tau\lambda;\zeta)
-
\nabla_x g(x,y+\tau\lambda)
\right\|^2
\right]
+
2\mathbb{E}\left[
\left\|
\nabla_x g(x,y-\tau\lambda;\zeta)
-
\nabla_x g(x,y-\tau\lambda)
\right\|^2
\right]   \right)                                                     \\
&\le
\frac{\sigma_g^2}{B\tau^2}.
\end{aligned}
\]
Combining the bias and variance bounds gives
\[
\mathbb{E}\left[
\left\|
\nabla_{xy}^2 g(x,y)\lambda
-
\tilde{\nabla}_{xy}^2 g_{\lambda}(x,y;\mathcal{B}_g)
\right\|^2
\right]
\le
\frac{l_{g,2}^2C_{\lambda}^4}{4}\tau^2
+
\frac{\sigma_g^2}{B\tau^2}.
\]

The proof of \eqref{Lemma: Stoc FD 2} is similar. Indeed, by unbiasedness,
\[
\mathbb{E}\left[
\tilde{\nabla}_{yy}^2 g_{\lambda}(x,y;\mathcal{B}_g)
\right]
=
\frac{
\nabla_y g(x,y+\tau\lambda)
-
\nabla_y g(x,y-\tau\lambda)
}{
2\tau
}.
\]
Using the same integral argument and the Hessian Lipschitz continuity of $g$,
\[
\left\|
\nabla_{yy}^2 g(x,y)\lambda
-
\mathbb{E}\left[
\tilde{\nabla}_{yy}^2 g_{\lambda}(x,y;\mathcal{B}_g)
\right]
\right\|^2
\le
\frac{l_{g,2}^2C_\lambda^4}{4}\tau^2.
\]
Moreover, by the independence of samples in $\mathcal{B}_g$ and the variance bound in
Assumption \ref{Assumption: Stochastic First-order Oracle},
\[
\mathbb{E}\left[
\left\|
\tilde{\nabla}_{yy}^2 g_{\lambda}(x,y;\mathcal{B}_g)
-
\mathbb{E}\left[
\tilde{\nabla}_{yy}^2 g_{\lambda}(x,y;\mathcal{B}_g)
\right]
\right\|^2
\right]
\le
\frac{\sigma_g^2}{B\tau^2}.
\]
Therefore,
\[
\mathbb{E}\left[
\left\|
\nabla_{yy}^2 g(x,y)\lambda
-
\tilde{\nabla}_{yy}^2 g_{\lambda}(x,y;\mathcal{B}_g)
\right\|^2
\right]
\le
\frac{l_{g,2}^2C_{\lambda}^4}{4}\tau^2
+
\frac{\sigma_g^2}{B\tau^2}.
\]
This completes the proof.
\end{proof}

\begin{lemma}\label{Lemma: Stochastic Iterates Gap}
Under Assumptions \ref{Assumption: Bilevel Optimization} and \ref{Assumption: Stochastic First-order Oracle}, with the update rule of Algorithm \ref{Stoc-SGHA}, we have
\begin{align*}
\mathbb{E}[\|x^{k+1}-x^k\|^2]
\le
2\gamma_{x,k}^2
\mathbb{E}\left[
\left\|
\nabla_x K(w^k,\hat w^k,\lambda^k)
\right\|^2
\right]
+
\frac{2\gamma_{x,k}^2\sigma_f^2}{B} +
2\gamma_{x,k}^2
\left(
\frac{l_{g,2}^2C_{\lambda}^4}{4}\tau^2
+
\frac{\sigma_g^2}{B\tau^2}
\right),
\end{align*}
\begin{align*}
\mathbb{E}[\|y^{k+1}-y^k\|^2]
\le
2\gamma_{y,k}^2
\mathbb{E}\left[
\left\|
\nabla_y K(w^k,\hat w^k,\lambda^k)
\right\|^2
\right]
+
\frac{2\gamma_{y,k}^2\sigma_f^2}{B} +
2\gamma_{y,k}^2
\left(
\frac{l_{g,2}^2C_{\lambda}^4}{4}\tau^2
+
\frac{\sigma_g^2}{B\tau^2}
\right),
\end{align*}
\[
\mathbb{E}[\|\lambda^{k+1}-\lambda^k\|^2]
\le
\gamma_{\lambda,k}^2
\mathbb{E}\left[
\left\|
\nabla_y g(x^k,y^k)-p_\lambda\lambda^k
\right\|^2
\right]
+
\frac{\gamma_{\lambda,k}^2\sigma_g^2}{B}.
\]
In particular, since $\|\lambda^k\|\le C_\lambda$, we also have
\[
\mathbb{E}[\|\lambda^{k+1}-\lambda^k\|^2]
\le
2\gamma_{\lambda,k}^2
\mathbb{E}\left[
\left\|
\nabla_y g(x^k,y^k)
\right\|^2
\right]
+
2\gamma_{\lambda,k}^2p_\lambda^2C_\lambda^2
+
\frac{\gamma_{\lambda,k}^2\sigma_g^2}{B}.
\]
\end{lemma}
\begin{proof}
We first prove the bound for $x^{k+1}-x^k$. By the update rule of Algorithm \ref{Stoc-SGHA},
\[
x^{k+1}-x^k
=
-\gamma_{x,k}
\left(
\nabla_x f(x^k,y^k;\mathcal B_f^k)
+
\tilde{\nabla}_{xy}^2 g_{\lambda^k}(x^k,y^k;\mathcal B_g^k)
+
p_w(x^k-\hat x^k)
\right).
\]
Hence,
\begin{align*}
\mathbb{E}\left[
\|x^{k+1}-x^k\|^2\mid \mathcal F_k^-
\right]
&=
\gamma_{x,k}^2
\mathbb{E}\left[
\left\|
\nabla_x f(x^k,y^k;\mathcal B_f^k)
+
\tilde{\nabla}_{xy}^2 g_{\lambda^k}(x^k,y^k;\mathcal B_g^k)
+
p_w(x^k-\hat x^k)
\right\|^2
\mid \mathcal F_k^-
\right]                                                         \\
&\le
2\gamma_{x,k}^2
\mathbb{E}\left[
\left\|
\nabla_x f(x^k,y^k;\mathcal B_f^k)
+
\nabla_{xy}^2 g(x^k,y^k)\lambda^k
+
p_w(x^k-\hat x^k)
\right\|^2
\mid \mathcal F_k^-
\right]                                                         \\
&\quad
+
2\gamma_{x,k}^2
\mathbb{E}\left[
\left\|
\tilde{\nabla}_{xy}^2 g_{\lambda^k}(x^k,y^k;\mathcal B_g^k)
-
\nabla_{xy}^2 g(x^k,y^k)\lambda^k
\right\|^2
\mid \mathcal F_k^-
\right].
\end{align*}
For the first term, by the variance bound of the stochastic gradient of $f$,
\begin{align*}
&\mathbb{E}\left[
\left\|
\nabla_x f(x^k,y^k;\mathcal B_f^k)
+
\nabla_{xy}^2 g(x^k,y^k)\lambda^k
+
p_w(x^k-\hat x^k)
\right\|^2
\mid \mathcal F_k^-
\right]                                                         \\
&\le
\left\|
\nabla_x f(x^k,y^k)
+
\nabla_{xy}^2 g(x^k,y^k)\lambda^k
+
p_w(x^k-\hat x^k)
\right\|^2
+
\frac{\sigma_f^2}{B}                                           \\
&=
\left\|
\nabla_x K(w^k,\hat w^k,\lambda^k)
\right\|^2
+
\frac{\sigma_f^2}{B}.
\end{align*}
For the second term, by Lemma \ref{Lemma: Stochastic Finite Difference Estimator},
\[
\mathbb{E}\left[
\left\|
\tilde{\nabla}_{xy}^2 g_{\lambda^k}(x^k,y^k;\mathcal B_g^k)
-
\nabla_{xy}^2 g(x^k,y^k)\lambda^k
\right\|^2
\mid \mathcal F_k^-
\right]
\le
\frac{l_{g,2}^2C_{\lambda}^4}{4}\tau^2
+
\frac{\sigma_g^2}{B\tau^2}.
\]
Combining the above inequalities gives
\begin{align*}
\mathbb{E}\left[
\|x^{k+1}-x^k\|^2\mid \mathcal F_k^-
\right]
\le
2\gamma_{x,k}^2
\left\|
\nabla_x K(w^k,\hat w^k,\lambda^k)
\right\|^2
+
\frac{2\gamma_{x,k}^2\sigma_f^2}{B}
+
2\gamma_{x,k}^2
\left(
\frac{l_{g,2}^2C_{\lambda}^4}{4}\tau^2
+
\frac{\sigma_g^2}{B\tau^2}
\right).
\end{align*}
Taking total expectation over both sides yields
\begin{align*}
\mathbb{E}[\|x^{k+1}-x^k\|^2]
\le
2\gamma_{x,k}^2
\mathbb{E}\left[
\left\|
\nabla_x K(w^k,\hat w^k,\lambda^k)
\right\|^2
\right]
+
\frac{2\gamma_{x,k}^2\sigma_f^2}{B}
+
2\gamma_{x,k}^2
\left(
\frac{l_{g,2}^2C_{\lambda}^4}{4}\tau^2
+
\frac{\sigma_g^2}{B\tau^2}
\right).
\end{align*}
Similarly,
\begin{align*}
\mathbb{E}[\|y^{k+1}-y^k\|^2]
\le
2\gamma_{y,k}^2
\mathbb{E}\left[
\left\|
\nabla_y K(w^k,\hat w^k,\lambda^k)
\right\|^2
\right]
+
\frac{2\gamma_{y,k}^2\sigma_f^2}{B}
+
2\gamma_{y,k}^2
\left(
\frac{l_{g,2}^2C_{\lambda}^4}{4}\tau^2
+
\frac{\sigma_g^2}{B\tau^2}
\right).
\end{align*}

Finally, for $\lambda^k$, by the update rule,
\[
\lambda^{k+1}
=
\mathcal P_\Lambda\left(
\lambda^k
+
\gamma_{\lambda,k}
\left(
\nabla_y g(x^k,y^k;\mathcal B_g^k)
-
p_\lambda\lambda^k
\right)
\right).
\]
Since $\lambda^k\in\Lambda$, we have $\mathcal P_\Lambda(\lambda^k)=\lambda^k$. By the non-expansiveness of the projection operator,
\begin{align*}
\mathbb{E}[\|\lambda^{k+1}-\lambda^k\|^2\mid \mathcal F_k^-]
&=
\mathbb{E}\left[
\left\|
\mathcal P_\Lambda\left(
\lambda^k
+
\gamma_{\lambda,k}
\left(
\nabla_y g(x^k,y^k;\mathcal B_g^k)
-
p_\lambda\lambda^k
\right)
\right)
-
\mathcal P_\Lambda(\lambda^k)
\right\|^2\mid \mathcal F_k^-
\right]                                                        \\
&\le
\gamma_{\lambda,k}^2
\mathbb{E}\left[
\left\|
\nabla_y g(x^k,y^k;\mathcal B_g^k)
-
p_\lambda\lambda^k
\right\|^2\mid \mathcal F_k^-
\right].
\end{align*}
By the unbiasedness and variance bound of the stochastic gradient of $g$,
\[
\mathbb{E}\left[
\left\|
\nabla_y g(x^k,y^k;\mathcal B_g^k)
-
p_\lambda\lambda^k
\right\|^2\mid \mathcal F_k^-
\right]
\le
\left\|
\nabla_y g(x^k,y^k)
-
p_\lambda\lambda^k
\right\|^2
+
\frac{\sigma_g^2}{B}.
\]
Therefore, taking total expectation yields
\[
\mathbb{E}[\|\lambda^{k+1}-\lambda^k\|^2]
\le
\gamma_{\lambda,k}^2
\mathbb{E}\left[
\left\|
\nabla_y g(x^k,y^k)-p_\lambda\lambda^k
\right\|^2
\right]
+
\frac{\gamma_{\lambda,k}^2\sigma_g^2}{B}.
\]
Furthermore, since $\|\lambda^k\|\le C_\lambda$,
\[
\left\|
\nabla_y g(x^k,y^k)-p_\lambda\lambda^k
\right\|^2
\le
2\left\|
\nabla_y g(x^k,y^k)
\right\|^2
+
2p_\lambda^2C_\lambda^2.
\]
Thus,
\[
\mathbb{E}[\|\lambda^{k+1}-\lambda^k\|^2]
\le
2\gamma_{\lambda,k}^2
\mathbb{E}\left[
\left\|
\nabla_y g(x^k,y^k)
\right\|^2
\right]
+
2\gamma_{\lambda,k}^2p_\lambda^2C_\lambda^2
+
\frac{\gamma_{\lambda,k}^2\sigma_g^2}{B}.
\]
This completes the proof.
\end{proof}
\subsection{Intermediate Lemmas for Lyapunov Function under Stochastic Setting}
\begin{lemma}\label{Lemma: Stochastic Primal Descent w lambda}
Under Assumptions \ref{Assumption: Bilevel Optimization} and \ref{Assumption: Stochastic First-order Oracle}, suppose $\gamma_{w,k}\le \frac{1}{8(p_w+l_{\mathcal L,1})}$. Then, we have
\begin{align*}
&\mathbb{E}\left[
K(w^k,\hat w^k,\lambda^k)
\right]
-
\mathbb{E}\left[
K(w^{k+1},\hat w^k,\lambda^{k+1})
\right]                                                        \\
&\ge
\frac{\gamma_{w,k}}{4}
\mathbb{E}\left[
\left\|
\nabla_wK(w^k,\hat w^k,\lambda^k)
\right\|^2
\right]                                                 \\
&\quad
-
\left(
\gamma_{w,k}
+
4(p_w+l_{\mathcal L,1})\gamma_{w,k}^2
\right)
\left(
\frac{l_{g,2}^2C_\lambda^4}{4}\tau^2
+
\frac{\sigma_g^2}{B\tau^2}
\right)      -
\frac{4(p_w+l_{\mathcal L,1})\gamma_{w,k}^2}{B}
\sigma_f^2                                                   \\
&\quad
+
\mathbb{E}\left[
\left\langle
\nabla_y g(x^k,y^k)-p_\lambda\lambda^k,
\lambda^k-\lambda^{k+1}
\right\rangle
\right]
-
\left(\frac{p_\lambda}{2} + \frac{l_{g,1}^2}{2(p_w+l_{\mathcal{L},1})}\right)
\mathbb{E}\left[
\|\lambda^{k+1}-\lambda^k\|^2
\right].
\end{align*}
\end{lemma}
\begin{proof}
Since $K(w,\hat w,\lambda)$ is $(p_w+l_{\mathcal L,1})$-smooth with respect to $w$
and $p_\lambda$-smooth with respect to $\lambda$, using the same descent
decomposition as in Lemma \ref{Lemma: Deterministic Primal Descent w lambda}, we have
\begin{align}
&K(w^k,\hat w^k,\lambda^k)
-
K(w^{k+1},\hat w^k,\lambda^{k+1})  \nonumber\\
&\ge
\left\langle
\nabla_wK(w^k,\hat w^k,\lambda^k),
w^k-w^{k+1}
\right\rangle
+
\left\langle
\nabla_y g(x^k,y^k)-p_\lambda\lambda^k,
\lambda^k-\lambda^{k+1}
\right\rangle                                                   \nonumber\\
&\quad + \langle \nabla_y g(x^{k+1},y^{k+1}) - \nabla_y g(x^k,y^k),\lambda^k - \lambda^{k+1}\rangle\nonumber\\
&\quad -
\frac{p_w+l_{\mathcal L,1}}{2}
\|w^{k+1}-w^k\|^2
-
\frac{p_\lambda}{2}
\|\lambda^{k+1}-\lambda^k\|^2.
\label{Lemma: Stochastic Primal Descent 1}
\end{align}
Taking conditional expectation with respect to $\mathcal F_k^-$ and using the
unbiasedness of $\nabla f(x^k,y^k;\mathcal B_f^k)$, we get
\begin{align*}
&\mathbb{E}\left[
\left\langle
\nabla_wK(w^k,\hat w^k,\lambda^k),
w^k-w^{k+1}
\right\rangle
\mid \mathcal F_k^-
\right]                                                        \\
&=
\gamma_{x,k}
\left\|
\nabla_xK(w^k,\hat w^k,\lambda^k)
\right\|^2
+
\gamma_{y,k}
\left\|
\nabla_yK(w^k,\hat w^k,\lambda^k)
\right\|^2                                                     \\
&\quad
+
\gamma_{x,k}
\left\langle
\nabla_xK(w^k,\hat w^k,\lambda^k),
\mathbb{E}\left[
\tilde{\nabla}_{xy}^2g_{\lambda^k}(x^k,y^k;\mathcal B_g^k)
-
\nabla_{xy}^2g(x^k,y^k)\lambda^k
\mid \mathcal F_k^-
\right]
\right\rangle                                                   \\
&\quad
+
\gamma_{y,k}
\left\langle
\nabla_yK(w^k,\hat w^k,\lambda^k),
\mathbb{E}\left[
\tilde{\nabla}_{yy}^2g_{\lambda^k}(x^k,y^k;\mathcal B_g^k)
-
\nabla_{yy}^2g(x^k,y^k)\lambda^k
\mid \mathcal F_k^-
\right]
\right\rangle .
\end{align*}
Using Young's inequality, we have
\begin{align*}
\mathbb{E}\left[
\left\langle
\nabla_wK(w^k,\hat w^k,\lambda^k),
w^k-w^{k+1}
\right\rangle
\mid \mathcal F_k^-
\right]
&\ge
\frac{\gamma_{x,k}}{2}
\left\|
\nabla_xK(w^k,\hat w^k,\lambda^k)
\right\|^2
+
\frac{\gamma_{y,k}}{2}
\left\|
\nabla_yK(w^k,\hat w^k,\lambda^k)
\right\|^2                                                     \\
&\quad
-
\frac{\gamma_{x,k}}{2}
\left\|
\mathbb{E}\left[
\tilde{\nabla}_{xy}^2g_{\lambda^k}(x^k,y^k;\mathcal B_g^k)
-
\nabla_{xy}^2g(x^k,y^k)\lambda^k
\mid \mathcal F_k^-
\right]
\right\|^2                                                     \\
&\quad
-
\frac{\gamma_{y,k}}{2}
\left\|
\mathbb{E}\left[
\tilde{\nabla}_{yy}^2g_{\lambda^k}(x^k,y^k;\mathcal B_g^k)
-
\nabla_{yy}^2g(x^k,y^k)\lambda^k
\mid \mathcal F_k^-
\right]
\right\|^2,
\end{align*}
and
\begin{align}\label{Lemma: Stochastic Primal Descent 2}
    &\langle \nabla_y g(x^{k+1},y^{k+1}) - \nabla_y g(x^k,y^k),\lambda^k - \lambda^{k+1}\rangle \nonumber\\&\geq - \frac{p_w+l_{\mathcal{L},1}}{2}\|w^{k+1} - w^k\|^2 - \frac{l_{g,1}^2}{2(p_w+l_{\mathcal{L},1})}\|\lambda^{k+1} - \lambda^k\|^2.
\end{align}
By Jensen's inequality and Lemma \ref{Lemma: Stochastic Finite Difference Estimator},
\begin{align}
\mathbb{E}\left[
\left\langle
\nabla_wK(w^k,\hat w^k,\lambda^k),
w^k-w^{k+1}
\right\rangle
\mid \mathcal F_k^-
\right]
&\ge
\frac{\gamma_{x,k}}{2}
\left\|
\nabla_xK(w^k,\hat w^k,\lambda^k)
\right\|^2
+
\frac{\gamma_{y,k}}{2}
\left\|
\nabla_yK(w^k,\hat w^k,\lambda^k)
\right\|^2                                                     \nonumber\\
&\quad
-
\frac{\gamma_{x,k}+\gamma_{y,k}}{2}
\left(
\frac{l_{g,2}^2C_\lambda^4}{4}\tau^2
+
\frac{\sigma_g^2}{B\tau^2}
\right).
\label{Lemma: Stochastic Primal Descent 3}
\end{align}
Moreover, by Lemma \ref{Lemma: Stochastic Iterates Gap},
\begin{align}
\mathbb{E}\left[
\|w^{k+1}-w^k\|^2
\mid \mathcal F_k^-
\right]
&\le
2\gamma_{x,k}^2
\left\|
\nabla_xK(w^k,\hat w^k,\lambda^k)
\right\|^2
+
2\gamma_{y,k}^2
\left\|
\nabla_yK(w^k,\hat w^k,\lambda^k)
\right\|^2                                     \nonumber                \\
&\quad
+
\frac{2(\gamma_{x,k}^2+\gamma_{y,k}^2)\sigma_f^2}{B}
+
2(\gamma_{x,k}^2+\gamma_{y,k}^2)
\left(
\frac{l_{g,2}^2C_\lambda^4}{4}\tau^2
+
\frac{\sigma_g^2}{B\tau^2}
\right).
\label{Lemma: Stochastic Primal Descent 4}
\end{align}
Substituting \eqref{Lemma: Stochastic Primal Descent 2}, \eqref{Lemma: Stochastic Primal Descent 3} and \eqref{Lemma: Stochastic Primal Descent 4} into
\eqref{Lemma: Stochastic Primal Descent 1}, and then taking total expectation yields
\begin{align*}
&\mathbb{E}\left[
K(w^k,\hat w^k,\lambda^k)
\right]
-
\mathbb{E}\left[
K(w^{k+1},\hat w^k,\lambda^{k+1})
\right]                                                        \\
&\ge
\left(
\frac{\gamma_{x,k}}{2}
-
2(p_w+l_{\mathcal L,1})\gamma_{x,k}^2
\right)
\mathbb{E}\left[
\left\|
\nabla_xK(w^k,\hat w^k,\lambda^k)
\right\|^2
\right]                                                        \\
&\quad
+
\left(
\frac{\gamma_{y,k}}{2}
-
2(p_w+l_{\mathcal L,1})\gamma_{y,k}^2
\right)
\mathbb{E}\left[
\left\|
\nabla_yK(w^k,\hat w^k,\lambda^k)
\right\|^2
\right]                                                        \\
&\quad
-
\left(
\gamma_{w,k}
+
4(p_w+l_{\mathcal L,1})\gamma_{w,k}^2
\right)
\left(
\frac{l_{g,2}^2C_\lambda^4}{4}\tau^2
+
\frac{\sigma_g^2}{B\tau^2}
\right)
-
\frac{2(p_w+l_{\mathcal L,1})(\gamma_{x,k}^2+\gamma_{y,k}^2)}{B}
\sigma_f^2                                                     \\
&\quad
+
\mathbb{E}\left[
\left\langle
\nabla_y g(x^k,y^k)-p_\lambda\lambda^k,
\lambda^k-\lambda^{k+1}
\right\rangle
\right]
-
\left(\frac{p_\lambda}{2} + \frac{l_{g,1}^2}{2(p_w+l_{\mathcal{L},1})}\right)
\mathbb{E}\left[
\|\lambda^{k+1}-\lambda^k\|^2
\right].
\end{align*}
By choosing $\gamma_{w,k} = \gamma_{x,k} =
\gamma_{y,k}\le \frac{1}{8(p_w+l_{\mathcal L,1})}$, we have
we have
\begin{align*}
&\mathbb{E}\left[
K(w^k,\hat w^k,\lambda^k)
\right]
-
\mathbb{E}\left[
K(w^{k+1},\hat w^k,\lambda^{k+1})
\right]                                                        \\
&\ge
\frac{\gamma_{w,k}}{4}
\mathbb{E}\left[
\left\|
\nabla_wK(w^k,\hat w^k,\lambda^k)
\right\|^2
\right]                                                      \\
&\quad
-
\left(
\gamma_{w,k}
+
4(p_w+l_{\mathcal L,1})\gamma_{w,k}^2
\right)
\left(
\frac{l_{g,2}^2C_\lambda^4}{4}\tau^2
+
\frac{\sigma_g^2}{B\tau^2}
\right)      -
\frac{4(p_w+l_{\mathcal L,1})\gamma_{w,k}^2}{B}
\sigma_f^2                                                   \\
&\quad
+
\mathbb{E}\left[
\left\langle
\nabla_y g(x^k,y^k)-p_\lambda\lambda^k,
\lambda^k-\lambda^{k+1}
\right\rangle
\right]
-
\left(\frac{p_\lambda}{2} + \frac{l_{g,1}^2}{2(p_w+l_{\mathcal{L},1})}\right)
\mathbb{E}\left[
\|\lambda^{k+1}-\lambda^k\|^2
\right].
\end{align*}
This completes the proof.
\end{proof}

\begin{lemma}\label{Lemma: Stochastic Primal Descent - Dual Ascent}
Suppose the conditions in Lemma \ref{Lemma: Stochastic Primal Descent w lambda} hold. Set $p_w=2l_{\mathcal L,1}$. Suppose further that
\[
\gamma_{\lambda,k}
\le
\min\left\{
\frac{1}{2\left(2l_{\Psi,1}+p_\lambda + \frac{l_{g,1}^2}{p_w+l_{\mathcal{L},1}}\right)},
\frac{l_{\mathcal L,1}^2}{16l_{g,1}^2}\gamma_{w,k}
\right\}.
\]
Then, we have
\begin{align*}
&\mathbb{E}\left[
K(w^k,\hat w^k,\lambda^k)
\right]
-
\mathbb{E}\left[
K(w^{k+1},\hat w^k,\lambda^{k+1})
\right]
+
2\mathbb{E}\left[
\Psi(\hat w^k,\lambda^{k+1})-\Psi(\hat w^k,\lambda^k)
\right]                                                        \\
&\ge
\frac{\gamma_{w,k}}{8}
\mathbb{E}\left[
\left\|
\nabla_wK(w^k,\hat w^k,\lambda^k)
\right\|^2
\right]
+
\frac{1}{4\gamma_{\lambda,k}}
\mathbb{E}\left[
\left\|
\lambda^{k+1}-\lambda^k
\right\|^2
\right]                                                        \\
&\quad
-
\frac{4(p_w+l_{\mathcal L,1})\gamma_{w,k}^2}{B}
\sigma_f^2
-
\left(
\gamma_{w,k}
+
4(p_w+l_{\mathcal L,1})\gamma_{w,k}^2
\right)
\left(
\frac{l_{g,2}^2C_\lambda^4}{4}\tau^2
+
\frac{\sigma_g^2}{B\tau^2}
\right)
-
\frac{\gamma_{\lambda,k}\sigma_g^2}{B}.
\end{align*}
\end{lemma}
\begin{proof}
Combining Lemma \ref{Lemma: Stochastic Primal Descent w lambda} and Lemma \ref{Lemma: Dual Ascent hat w lambda}, we have
\begin{align*}
&\mathbb{E}\left[
K(w^k,\hat w^k,\lambda^k)
\right]
-
\mathbb{E}\left[
K(w^{k+1},\hat w^k,\lambda^{k+1})
\right]
+
2\mathbb{E}\left[
\Psi(\hat w^k,\lambda^{k+1})-\Psi(\hat w^k,\lambda^k)
\right]                                     \\
&\ge
\frac{\gamma_{w,k}}{4}
\mathbb{E}\left[
\left\|
\nabla_wK(w^k,\hat w^k,\lambda^k)
\right\|^2
\right]  -
\left(
l_{\Psi,1}+\frac{p_\lambda}{2} + \frac{l_{g,1}^2}{2(p_w+l_{\mathcal{L},1})}
\right)
\mathbb{E}\left[
\left\|
\lambda^{k+1}-\lambda^k
\right\|^2
\right]    \\
&\quad + \mathbb{E}\left[
\left\langle
\nabla_y g(x^k,y^k)-p_\lambda\lambda^k,
\lambda^{k+1}-\lambda^k
\right\rangle
\right]             +
2\mathbb{E}\left[
\left\langle
\nabla_\lambda\Psi(\hat w^k,\lambda^k)
-
\left(
\nabla_y g(x^k,y^k)-p_\lambda\lambda^k
\right),
\lambda^{k+1}-\lambda^k
\right\rangle
\right]    \\
&\quad -
\frac{4(p_w+l_{\mathcal L,1})\gamma_{w,k}^2}{B}
\sigma_f^2
-
\left(
\gamma_{w,k}
+
4(p_w+l_{\mathcal L,1})\gamma_{w,k}^2
\right)
\left(
\frac{l_{g,2}^2C_\lambda^4}{4}\tau^2
+
\frac{\sigma_g^2}{B\tau^2}
\right).
\end{align*}
By the update rule,
\[
\lambda^{k+1}
=
\mathcal P_\Lambda\left(
\lambda^k+
\gamma_{\lambda,k}
\left(
\nabla_y g(x^k,y^k;\mathcal B_g^k)-p_\lambda\lambda^k
\right)
\right).
\]
The optimality condition of the projection gives
\[
\left\langle
\nabla_y g(x^k,y^k;\mathcal B_g^k)-p_\lambda\lambda^k,
\lambda^{k+1}-\lambda^k
\right\rangle
\ge
\frac{1}{\gamma_{\lambda,k}}
\left\|
\lambda^{k+1}-\lambda^k
\right\|^2 .
\]
Therefore,
\begin{align*}
\left\langle
\nabla_y g(x^k,y^k)-p_\lambda\lambda^k,
\lambda^{k+1}-\lambda^k
\right\rangle
&=
\left\langle
\nabla_y g(x^k,y^k;\mathcal B_g^k)-p_\lambda\lambda^k,
\lambda^{k+1}-\lambda^k
\right\rangle                                                   \\
&\quad
-
\left\langle
\nabla_y g(x^k,y^k;\mathcal B_g^k)-\nabla_y g(x^k,y^k),
\lambda^{k+1}-\lambda^k
\right\rangle                                                   \\
&\ge
\frac{1}{\gamma_{\lambda,k}}
\left\|
\lambda^{k+1}-\lambda^k
\right\|^2
-
\left\langle
\nabla_y g(x^k,y^k;\mathcal B_g^k)-\nabla_y g(x^k,y^k),
\lambda^{k+1}-\lambda^k
\right\rangle .
\end{align*}
Define
\[
\bar{\lambda}^{k+1} = \mathcal{P}_{\Lambda}(\lambda^k + \gamma_{\lambda,k}(\nabla_y g(x^k,y^k) - p_{\lambda}\lambda^k)).
\]
Taking conditional expectation with respect to $\mathcal F_k^-$, and using the non-expansiveness of
$\mathcal P_\Lambda$, we have
\begin{align*}
&\mathbb{E}\left[
\left\langle
\nabla_y g(x^k,y^k;\mathcal B_g^k)-\nabla_y g(x^k,y^k),
\lambda^{k+1}-\lambda^k
\right\rangle
\mid \mathcal F_k^-
\right]                                                        \\
&= \mathbb{E}\left[
\left\langle
\nabla_y g(x^k,y^k;\mathcal B_g^k)-\nabla_y g(x^k,y^k),
\lambda^{k+1}-\bar{\lambda}^{k+1}
\right\rangle
\mid \mathcal F_k^-
\right] \\
&\quad +    \mathbb{E}\left[
\left\langle
\nabla_y g(x^k,y^k;\mathcal B_g^k)-\nabla_y g(x^k,y^k),
\bar{\lambda}^{k+1} - \lambda^{k}
\right\rangle
\mid \mathcal F_k^-
\right]                                            \\
&\le
\gamma_{\lambda,k}
\mathbb{E}\left[
\left\|
\nabla_y g(x^k,y^k;\mathcal B_g^k)-\nabla_y g(x^k,y^k)
\right\|^2
\mid \mathcal F_k^-
\right]        + 0                                                \\
&\le
\frac{\gamma_{\lambda,k}\sigma_g^2}{B}.
\end{align*}
Thus,
\begin{align*}
\mathbb{E}\left[
\left\langle
\nabla_y g(x^k,y^k)-p_\lambda\lambda^k,
\lambda^{k+1}-\lambda^k
\right\rangle
\right]                             \ge
\frac{1}{\gamma_{\lambda,k}}
\mathbb{E}\left[
\left\|
\lambda^{k+1}-\lambda^k
\right\|^2
\right]
-
\frac{\gamma_{\lambda,k}\sigma_g^2}{B}.
\end{align*}
Next, by Young's inequality,
\begin{align*}
&2\mathbb{E}\left[
\left\langle
\nabla_\lambda\Psi(\hat w^k,\lambda^k)
-
\left(
\nabla_y g(x^k,y^k)-p_\lambda\lambda^k
\right),
\lambda^{k+1}-\lambda^k
\right\rangle
\right]                                                        \\
&\ge
-
\frac{1}{2\gamma_{\lambda,k}}
\mathbb{E}\left[
\left\|
\lambda^{k+1}-\lambda^k
\right\|^2
\right]
-
2\gamma_{\lambda,k}
\mathbb{E}\left[
\left\|
\nabla_\lambda\Psi(\hat w^k,\lambda^k)
-
\left(
\nabla_y g(x^k,y^k)-p_\lambda\lambda^k
\right)
\right\|^2
\right].
\end{align*}
By Danskin's theorem,
\[
\nabla_\lambda\Psi(\hat w^k,\lambda^k)
=
\nabla_y g(w^\star(\hat w^k,\lambda^k))-p_\lambda\lambda^k.
\]
Hence,
\[
\begin{aligned}
\left\|
\nabla_\lambda\Psi(\hat w^k,\lambda^k)
-
\left(
\nabla_y g(x^k,y^k)-p_\lambda\lambda^k
\right)
\right\|
&=
\left\|
\nabla_y g(w^\star(\hat w^k,\lambda^k))
-
\nabla_y g(w^k)
\right\|                                                        \\
&\le
l_{g,1}
\left\|
w^\star(\hat w^k,\lambda^k)-w^k
\right\|                                                        \\
&\le
\frac{l_{g,1}}{p_w-l_{\mathcal L,1}}
\left\|
\nabla_wK(w^k,\hat w^k,\lambda^k)
\right\|.
\end{aligned}
\]
Since
\[
\gamma_{\lambda,k}
\le
\frac{1}{2\left(2l_{\Psi,1}+p_\lambda+ \frac{l_{g,1}^2}{p_w+l_{\mathcal{L},1}}\right)},
\]
we have
\[
\frac{1}{\gamma_{\lambda,k}}
-
l_{\Psi,1}
-
\frac{p_\lambda}{2}
- \frac{l_{g,1}^2}{2(p_w+l_{\mathcal{L},1})}
-
\frac{1}{2\gamma_{\lambda,k}}
\ge
\frac{1}{4\gamma_{\lambda,k}}.
\]
Therefore, we get
\begin{align*}
&\mathbb{E}\left[
K(w^k,\hat w^k,\lambda^k)
\right]
-
\mathbb{E}\left[
K(w^{k+1},\hat w^k,\lambda^{k+1})
\right]
+
2\mathbb{E}\left[
\Psi(\hat w^k,\lambda^{k+1})-\Psi(\hat w^k,\lambda^k)
\right]                                                        \\
&\ge
\left(
\frac{\gamma_{w,k}}{4}
-
\frac{2\gamma_{\lambda,k}l_{g,1}^2}{(p_w-l_{\mathcal L,1})^2}
\right)
\mathbb{E}\left[
\left\|
\nabla_wK(w^k,\hat w^k,\lambda^k)
\right\|^2
\right]
+
\frac{1}{4\gamma_{\lambda,k}}
\mathbb{E}\left[
\left\|
\lambda^{k+1}-\lambda^k
\right\|^2
\right]                                                        \\
&\quad
-
\frac{4(p_w+l_{\mathcal L,1})\gamma_{w,k}^2}{B}
\sigma_f^2
-
\left(
\gamma_{w,k}
+
4(p_w+l_{\mathcal L,1})\gamma_{w,k}^2
\right)
\left(
\frac{l_{g,2}^2C_\lambda^4}{4}\tau^2
+
\frac{\sigma_g^2}{B\tau^2}
\right)
-
\frac{\gamma_{\lambda,k}\sigma_g^2}{B}.
\end{align*}
Finally, since
\[
p_w=2l_{\mathcal L,1},
\]
we have
\[
(p_w-l_{\mathcal L,1})^2=l_{\mathcal L,1}^2.
\]
By
\[
\gamma_{\lambda,k}
\le
\frac{(p_w-l_{\mathcal L,1})^2}{16l_{g,1}^2}\gamma_{w,k} = \frac{l_{\mathcal L,1}^2}{16l_{g,1}^2}\gamma_{w,k},
\]
we have
\[
\frac{2\gamma_{\lambda,k}l_{g,1}^2}{(p_w-l_{\mathcal L,1})^2}
\le
\frac{\gamma_{w,k}}{8}.
\]
Thus,
\begin{align*}
&\mathbb{E}\left[
K(w^k,\hat w^k,\lambda^k)
\right]
-
\mathbb{E}\left[
K(w^{k+1},\hat w^k,\lambda^{k+1})
\right]
+
2\mathbb{E}\left[
\Psi(\hat w^k,\lambda^{k+1})-\Psi(\hat w^k,\lambda^k)
\right]                                                        \\
&\ge
\frac{\gamma_{w,k}}{8}
\mathbb{E}\left[
\left\|
\nabla_wK(w^k,\hat w^k,\lambda^k)
\right\|^2
\right]
+
\frac{1}{4\gamma_{\lambda,k}}
\mathbb{E}\left[
\left\|
\lambda^{k+1}-\lambda^k
\right\|^2
\right]                                                        \\
&\quad
-
\frac{4(p_w+l_{\mathcal L,1})\gamma_{w,k}^2}{B}
\sigma_f^2
-
\left(
\gamma_{w,k}
+
4(p_w+l_{\mathcal L,1})\gamma_{w,k}^2
\right)
\left(
\frac{l_{g,2}^2C_\lambda^4}{4}\tau^2
+
\frac{\sigma_g^2}{B\tau^2}
\right)
-
\frac{\gamma_{\lambda,k}\sigma_g^2}{B}.
\end{align*}
This completes the proof.
\end{proof}

\begin{proposition}\label{Proposition: Stochastic Lyapunov Function Descent}
Define the Lyapunov function
\[
V_k
=
K(w^k,\hat w^k,\lambda^k)
-
2\Psi(\hat w^k,\lambda^k)
+
2P(\hat w^k).
\]
Suppose
\[
\gamma_{w,k}
\le
\min\left\{
\frac{1}{8p_w\beta},
\frac{1}{8(p_w+l_{\mathcal L,1})}
\right\},\quad \gamma_{\lambda,k}
\le
\min\left\{\frac{1}{2\left(2l_{\Psi,1}+p_\lambda + \frac{l_{g,1}^2}{p_w+l_{\mathcal{L},1}}\right)},
\frac{l_{\mathcal L,1}^2}{16l_{g,1}^2}\gamma_{w,k},
\frac{1}{1536p_w\beta\gamma_2^2}
\right\},
\]
and
\[
\beta\le \frac{\sqrt{5}-1}{48\gamma_1}.
\]
Then, we have
\begin{align*}
&\mathbb{E}[V_k]-\mathbb{E}[V_{k+1}]                                      \\
&\ge
\frac{1}{32\gamma_{w,k}}
\mathbb{E}\left[
\left\|
w^k-w^+(\hat w^k,\lambda^k)
\right\|^2
\right]
+
\frac{1}{32\gamma_{\lambda,k}}
\mathbb{E}\left[
\left\|
\lambda^+(\hat w^k)-\lambda^k
\right\|^2
\right]                                                                  \\
&\quad
+
\frac{p_w\beta}{8}
\mathbb{E}\left[
\left\|
w^k-\hat w^k
\right\|^2
\right]
-
48p_w\beta
\mathbb{E}\left[
\left\|
w^\star(\hat w^k)
-
w^\star(\hat w^k,\lambda^+(\hat w^k))
\right\|^2
\right]                                                                  \\
&\quad
-
\left(
\gamma_{w,k}
+
4(p_w+l_{\mathcal L,1})\gamma_{w,k}^2
+
p_w\beta\gamma_{w,k}^2
\right)
\left(
\frac{l_{g,2}^2C_\lambda^4}{4}\tau^2
+
\frac{\sigma_g^2}{B\tau^2}
\right)                                                                 \\
&\quad
-
\frac{
4(p_w+l_{\mathcal L,1})\gamma_{w,k}^2
\sigma_f^2
}{B} - \frac{
p_w\beta\gamma_{w,k}^2\sigma_f^2
}{B}
-
\frac{9\gamma_{\lambda,k}\sigma_g^2}{8B}.
\end{align*}
\end{proposition}
\begin{proof}
Combining the stochastic versions of Lemma \ref{Lemma: Primal Descent hat w} and
Lemma \ref{Lemma: Dual Ascent - Proximal Descent}, and
Lemma \ref{Lemma: Stochastic Primal Descent - Dual Ascent}, we have
\begin{align*}
&\mathbb{E}[V_k]-\mathbb{E}[V_{k+1}] \\
&=
\mathbb{E}[K(w^{k+1},\hat w^k,\lambda^{k+1})]
-
\mathbb{E}[K(w^{k+1},\hat w^{k+1},\lambda^{k+1})]                         \\
&\quad
+
2\mathbb{E}\left[
\Psi(\hat w^{k+1},\lambda^{k+1})
-
\Psi(\hat w^k,\lambda^{k+1})
\right]
-
2\mathbb{E}\left[
P(\hat w^{k+1})-P(\hat w^k)
\right]                                                                  \\
&\quad
+
\mathbb{E}[K(w^k,\hat w^k,\lambda^k)]
-
\mathbb{E}[K(w^{k+1},\hat w^k,\lambda^{k+1})]
+
2\mathbb{E}\left[
\Psi(\hat w^k,\lambda^{k+1})
-
\Psi(\hat w^k,\lambda^k)
\right]                                                                  \\
&\ge
\left(
\frac{p_w}{2\beta}
-
2p_w\gamma_1
-
\frac{p_w}{6\beta}
-
48p_w\beta\gamma_1^2
\right)
\mathbb{E}\left[
\|\hat w^{k+1}-\hat w^k\|^2
\right]
+
\left(
\frac{1}{4\gamma_{\lambda,k}}
-
24p_w\beta\gamma_2^2
\right)
\mathbb{E}\left[
\|\lambda^{k+1}-\lambda^k\|^2
\right]                                                                  \\
&\quad
-
24p_w\beta
\mathbb{E}\left[
\left\|
w^\star(\hat w^k)
-
w^\star(\hat w^k,\lambda^k)
\right\|^2
\right]
+
\frac{\gamma_{w,k}}{8}
\mathbb{E}\left[
\left\|
\nabla_wK(w^k,\hat w^k,\lambda^k)
\right\|^2
\right]                                                                  \\
&\quad
-
\left(
\gamma_{w,k}
+
4(p_w+l_{\mathcal L,1})\gamma_{w,k}^2
\right)
\left(
\frac{l_{g,2}^2C_\lambda^4}{4}\tau^2
+
\frac{\sigma_g^2}{B\tau^2}
\right)
-
\frac{
4(p_w+l_{\mathcal L,1})\gamma_{w,k}^2\sigma_f^2
}{B}
-
\frac{\gamma_{\lambda,k}\sigma_g^2}{B}.
\end{align*}
By
\[
\beta\le \frac{\sqrt{5}-1}{48\gamma_1},
\]
we have
\[
\frac{p_w}{2\beta}
-
2p_w\gamma_1
-
\frac{p_w}{6\beta}
-
48p_w\beta\gamma_1^2
\ge
\frac{p_w}{4\beta}.
\]
Moreover, by
\[
\gamma_{\lambda,k}
\le
\frac{1}{192p_w\beta\gamma_2^2},
\]
we have
\[
\frac{1}{4\gamma_{\lambda,k}}
-
24p_w\beta\gamma_2^2
\ge
\frac{1}{8\gamma_{\lambda,k}}.
\]
Therefore,
\begin{align}
&\mathbb{E}[V_k]-\mathbb{E}[V_{k+1}]                                      \nonumber\\
&\ge
\frac{\gamma_{w,k}}{8}
\mathbb{E}\left[
\left\|
\nabla_wK(w^k,\hat w^k,\lambda^k)
\right\|^2
\right]
+
\frac{p_w}{4\beta}
\mathbb{E}\left[
\|\hat w^{k+1}-\hat w^k\|^2
\right]
+
\frac{1}{8\gamma_{\lambda,k}}
\mathbb{E}\left[
\|\lambda^{k+1}-\lambda^k\|^2
\right]                                                                  \nonumber\\
&\quad
-
24p_w\beta
\mathbb{E}\left[
\left\|
w^\star(\hat w^k)
-
w^\star(\hat w^k,\lambda^k)
\right\|^2
\right]                                                                  \nonumber\\
&\quad
-
\left(
\gamma_{w,k}
+
4(p_w+l_{\mathcal L,1})\gamma_{w,k}^2
\right)
\left(
\frac{l_{g,2}^2C_\lambda^4}{4}\tau^2
+
\frac{\sigma_g^2}{B\tau^2}
\right)
-
\frac{
4(p_w+l_{\mathcal L,1})\gamma_{w,k}^2\sigma_f^2
}{B}
-
\frac{\gamma_{\lambda,k}\sigma_g^2}{B}.
\label{Proposition: Stochastic Lyapunov Function Descent 1}
\end{align}
By the update rule of $\hat w^k$,
\[
\hat w^{k+1}-\hat w^k
=
\beta(w^{k+1}-\hat w^k).
\]
Using $\|a\|^2\ge \frac12\|a-b\|^2-\|b\|^2$,
we obtain
\[
\begin{aligned}
\frac{p_w}{4\beta}
\|\hat w^{k+1}-\hat w^k\|^2
&=
\frac{p_w\beta}{4}
\|w^{k+1}-\hat w^k\|^2                                      \\
&\ge
\frac{p_w\beta}{8}
\|w^k-\hat w^k\|^2
-
\frac{p_w\beta}{4}
\|w^{k+1}-w^k\|^2 .
\end{aligned}
\]
By Lemma \ref{Lemma: Stochastic Iterates Gap},
\begin{align*}
\mathbb{E}\left[
\|w^{k+1}-w^k\|^2
\right]
&\le
2\gamma_{w,k}^2
\mathbb{E}\left[
\|\nabla_wK(w^k,\hat w^k,\lambda^k)\|^2
\right]
+
\frac{4\gamma_{w,k}^2\sigma_f^2}{B}
+
4\gamma_{w,k}^2
\left(
\frac{l_{g,2}^2C_\lambda^4}{4}\tau^2
+
\frac{\sigma_g^2}{B\tau^2}
\right).
\end{align*}
Substituting this estimate into
\eqref{Proposition: Stochastic Lyapunov Function Descent 1}, we get
\begin{align}
&\mathbb{E}[V_k]-\mathbb{E}[V_{k+1}]                                      \nonumber\\
&\ge
\left(
\frac{\gamma_{w,k}}{8}
-
\frac{p_w\beta\gamma_{w,k}^2}{2}
\right)
\mathbb{E}\left[
\|\nabla_wK(w^k,\hat w^k,\lambda^k)\|^2
\right]
+
\frac{1}{8\gamma_{\lambda,k}}
\mathbb{E}\left[
\|\lambda^{k+1}-\lambda^k\|^2
\right]                                                                  \nonumber\\
&\quad
+
\frac{p_w\beta}{8}
\mathbb{E}\left[
\|w^k-\hat w^k\|^2
\right]
-
24p_w\beta
\mathbb{E}\left[
\left\|
w^\star(\hat w^k)
-
w^\star(\hat w^k,\lambda^k)
\right\|^2
\right]                                                                  \nonumber\\
&\quad
-
\left(
\gamma_{w,k}
+
4(p_w+l_{\mathcal L,1})\gamma_{w,k}^2
+
p_w\beta\gamma_{w,k}^2
\right)
\left(
\frac{l_{g,2}^2C_\lambda^4}{4}\tau^2
+
\frac{\sigma_g^2}{B\tau^2}
\right)                                                                 \nonumber\\
&\quad
-
\frac{
4(p_w+l_{\mathcal L,1})\gamma_{w,k}^2
\sigma_f^2
}{B} - \frac{
p_w\beta\gamma_{w,k}^2\sigma_f^2
}{B}
-
\frac{\gamma_{\lambda,k}\sigma_g^2}{B}.
\label{Proposition: Stochastic Lyapunov Function Descent 2}
\end{align}
Since
\[
\gamma_{w,k}\le \frac{1}{8p_w\beta},
\]
we have
\[
\frac{\gamma_{w,k}}{8}
-
\frac{p_w\beta\gamma_{w,k}^2}{2}
\ge
\frac{\gamma_{w,k}}{16}.
\]
Thus,
\begin{align}
&\mathbb{E}[V_k]-\mathbb{E}[V_{k+1}]                                      \nonumber\\
&\ge
\frac{\gamma_{w,k}}{16}
\mathbb{E}\left[
\|\nabla_wK(w^k,\hat w^k,\lambda^k)\|^2
\right]
+
\frac{1}{8\gamma_{\lambda,k}}
\mathbb{E}\left[
\|\lambda^{k+1}-\lambda^k\|^2
\right]                                                                  \nonumber\\
&\quad
+
\frac{p_w\beta}{8}
\mathbb{E}\left[
\|w^k-\hat w^k\|^2
\right]
-
24p_w\beta
\mathbb{E}\left[
\left\|
w^\star(\hat w^k)
-
w^\star(\hat w^k,\lambda^k)
\right\|^2
\right]                                                                  \nonumber\\
&\quad
-
\left(
\gamma_{w,k}
+
4(p_w+l_{\mathcal L,1})\gamma_{w,k}^2
+
p_w\beta\gamma_{w,k}^2
\right)
\left(
\frac{l_{g,2}^2C_\lambda^4}{4}\tau^2
+
\frac{\sigma_g^2}{B\tau^2}
\right)                                                                 \nonumber\\
&\quad
-
\frac{
4(p_w+l_{\mathcal L,1})\gamma_{w,k}^2
\sigma_f^2
}{B} - \frac{
p_w\beta\gamma_{w,k}^2\sigma_f^2
}{B}
-
\frac{\gamma_{\lambda,k}\sigma_g^2}{B}.
\label{Proposition: Stochastic Lyapunov Function Descent 3}
\end{align}

Next, we convert the actual projected dual step into the ghost projected dual step.
By the non-expansiveness of $\mathcal P_\Lambda$,
\begin{align*}
\mathbb{E}\left[
\|\lambda^{k+1}-\lambda^+(\hat w^k)\|^2 \mid \mathcal{F}_k^-
\right]
&\le
\gamma_{\lambda,k}^2
\mathbb{E}\left[
\|\nabla_y g(w^k)-\nabla_y g(w^\star(\hat w^k,\lambda^k))\|^2 \mid \mathcal{F}_k^-
\right]
+
\frac{\gamma_{\lambda,k}^2\sigma_g^2}{B}                                 \\
&\le
\frac{\gamma_{\lambda,k}^2l_{g,1}^2}{(p_w-l_{\mathcal L,1})^2}
\mathbb{E}\left[
\|\nabla_wK(w^k,\hat w^k,\lambda^k)\|^2\mid \mathcal{F}_k^-
\right]
+
\frac{\gamma_{\lambda,k}^2\sigma_g^2}{B}.
\end{align*}
Taking total expectation yields
\[
\mathbb{E}\left[
\|\lambda^{k+1}-\lambda^+(\hat w^k)\|^2\right] \le
\frac{\gamma_{\lambda,k}^2l_{g,1}^2}{(p_w-l_{\mathcal L,1})^2}
\mathbb{E}\left[
\|\nabla_wK(w^k,\hat w^k,\lambda^k)\|^2
\right]
+
\frac{\gamma_{\lambda,k}^2\sigma_g^2}{B}.
\]
Therefore,
\begin{align*}
\mathbb{E}\left[
\|\lambda^{k+1}-\lambda^k\|^2
\right]
&\ge
\frac12
\mathbb{E}\left[
\|\lambda^+(\hat w^k)-\lambda^k\|^2
\right]
-
\mathbb{E}\left[
\|\lambda^{k+1}-\lambda^+(\hat w^k)\|^2
\right]                                                               \\
&\ge
\frac12
\mathbb{E}\left[
\|\lambda^+(\hat w^k)-\lambda^k\|^2
\right]
-
\frac{\gamma_{\lambda,k}^2l_{g,1}^2}{(p_w-l_{\mathcal L,1})^2}
\mathbb{E}\left[
\|\nabla_wK(w^k,\hat w^k,\lambda^k)\|^2
\right]
-
\frac{\gamma_{\lambda,k}^2\sigma_g^2}{B}.
\end{align*}
Substituting the above inequality into
\eqref{Proposition: Stochastic Lyapunov Function Descent 3}, we get
\begin{align}
&\mathbb{E}[V_k]-\mathbb{E}[V_{k+1}]                                      \nonumber\\
&\ge
\left(
\frac{\gamma_{w,k}}{16}
-
\frac{\gamma_{\lambda,k}l_{g,1}^2}{8(p_w-l_{\mathcal L,1})^2}
\right)
\mathbb{E}\left[
\|\nabla_wK(w^k,\hat w^k,\lambda^k)\|^2
\right]         +
\frac{1}{16\gamma_{\lambda,k}}
\mathbb{E}\left[
\|\lambda^+(\hat w^k)-\lambda^k\|^2
\right]                                                         \nonumber\\
&\quad
+
\frac{p_w\beta}{8}
\mathbb{E}\left[
\|w^k-\hat w^k\|^2
\right]
-
24p_w\beta
\mathbb{E}\left[
\left\|
w^\star(\hat w^k)
-
w^\star(\hat w^k,\lambda^k)
\right\|^2
\right]                                                                  \nonumber\\
&\quad
-
\left(
\gamma_{w,k}
+
4(p_w+l_{\mathcal L,1})\gamma_{w,k}^2
+
p_w\beta\gamma_{w,k}^2
\right)
\left(
\frac{l_{g,2}^2C_\lambda^4}{4}\tau^2
+
\frac{\sigma_g^2}{B\tau^2}
\right)                                                                 \nonumber\\
&\quad
-
\frac{
4(p_w+l_{\mathcal L,1})\gamma_{w,k}^2
\sigma_f^2
}{B} - \frac{
p_w\beta\gamma_{w,k}^2\sigma_f^2
}{B}
-
\frac{9\gamma_{\lambda,k}\sigma_g^2}{8B}.
\label{Proposition: Stochastic Lyapunov Function Descent 4}
\end{align}
Since
\[
\gamma_{\lambda,k}
\le \frac{(p_w-l_{\mathcal L,1})^2}{16l_{g,1}^2}\gamma_{w,k} \leq
\frac{(p_w-l_{\mathcal L,1})^2}{4l_{g,1}^2}\gamma_{w,k},
\]
we have
\[
\frac{\gamma_{w,k}}{16}
-
\frac{\gamma_{\lambda,k}l_{g,1}^2}{8(p_w-l_{\mathcal L,1})^2}
\ge
\frac{\gamma_{w,k}}{32}.
\]
Thus,
\begin{align*}
&\mathbb{E}[V_k]-\mathbb{E}[V_{k+1}]               \\
&\ge
\frac{\gamma_{w,k}}{32}
\mathbb{E}\left[
\|\nabla_wK(w^k,\hat w^k,\lambda^k)\|^2
\right]
+
\frac{1}{16\gamma_{\lambda,k}}
\mathbb{E}\left[
\|\lambda^+(\hat w^k)-\lambda^k\|^2
\right]
+
\frac{p_w\beta}{8}
\mathbb{E}\left[
\|w^k-\hat w^k\|^2
\right]                                \\
&\quad
-
24p_w\beta
\mathbb{E}\left[
\left\|
w^\star(\hat w^k)
-
w^\star(\hat w^k,\lambda^k)
\right\|^2
\right]                                \\
&\quad
-
\left(
\gamma_{w,k}
+
4(p_w+l_{\mathcal L,1})\gamma_{w,k}^2
+
p_w\beta\gamma_{w,k}^2
\right)
\left(
\frac{l_{g,2}^2C_\lambda^4}{4}\tau^2
+
\frac{\sigma_g^2}{B\tau^2}
\right)                                                                 \nonumber\\
&\quad
-
\frac{
4(p_w+l_{\mathcal L,1})\gamma_{w,k}^2
\sigma_f^2
}{B} - \frac{
p_w\beta\gamma_{w,k}^2\sigma_f^2
}{B}
-
\frac{9\gamma_{\lambda,k}\sigma_g^2}{8B}.
\end{align*}
Moreover,
\begin{align*}
\mathbb{E}\left[
\left\|
w^\star(\hat w^k)
-
w^\star(\hat w^k,\lambda^k)
\right\|^2
\right]                             \le
2\mathbb{E}\left[
\left\|
w^\star(\hat w^k)
-
w^\star(\hat w^k,\lambda^+(\hat w^k))
\right\|^2
\right]
+
2\gamma_2^2
\mathbb{E}\left[
\left\|
\lambda^+(\hat w^k)-\lambda^k
\right\|^2
\right].
\end{align*}
Finally, we obtain
\begin{align*}
&\mathbb{E}[V_k]-\mathbb{E}[V_{k+1}]                                      \\
&\ge
\frac{\gamma_{w,k}}{32}
\mathbb{E}\left[
\|\nabla_wK(w^k,\hat w^k,\lambda^k)\|^2
\right]
+
\left(
\frac{1}{16\gamma_{\lambda,k}}
-
48p_w\beta\gamma_2^2
\right)
\mathbb{E}\left[
\|\lambda^+(\hat w^k)-\lambda^k\|^2
\right]                                                                  \\
&\quad
+
\frac{p_w\beta}{8}
\mathbb{E}\left[
\|w^k-\hat w^k\|^2
\right]
-
48p_w\beta
\mathbb{E}\left[
\left\|
w^\star(\hat w^k)
-
w^\star(\hat w^k,\lambda^+(\hat w^k))
\right\|^2
\right]                                                                  \\
&\quad
-
\left(
\gamma_{w,k}
+
4(p_w+l_{\mathcal L,1})\gamma_{w,k}^2
+
p_w\beta\gamma_{w,k}^2
\right)
\left(
\frac{l_{g,2}^2C_\lambda^4}{4}\tau^2
+
\frac{\sigma_g^2}{B\tau^2}
\right)                                                                 \nonumber\\
&\quad
-
\frac{
4(p_w+l_{\mathcal L,1})\gamma_{w,k}^2
\sigma_f^2
}{B} - \frac{
p_w\beta\gamma_{w,k}^2\sigma_f^2
}{B}
-
\frac{9\gamma_{\lambda,k}\sigma_g^2}{8B}.
\end{align*}
By
\[
\gamma_{\lambda,k}
\le
\frac{1}{1536p_w\beta\gamma_2^2},
\]
we have
\[
\frac{1}{16\gamma_{\lambda,k}}
-
48p_w\beta\gamma_2^2
\ge
\frac{1}{32\gamma_{\lambda,k}}.
\]
Moreover, by the definition
\[
w^+(\hat w^k,\lambda^k)
=
w^k-\gamma_{w,k}\nabla_wK(w^k,\hat w^k,\lambda^k),
\]
we have
\[
\frac{\gamma_{w,k}}{32}
\left\|
\nabla_wK(w^k,\hat w^k,\lambda^k)
\right\|^2
=
\frac{1}{32\gamma_{w,k}}
\left\|
w^k-w^+(\hat w^k,\lambda^k)
\right\|^2.
\]
Therefore,
\begin{align*}
&\mathbb{E}[V_k]-\mathbb{E}[V_{k+1}]                                      \\
&\ge
\frac{1}{32\gamma_{w,k}}
\mathbb{E}\left[
\left\|
w^k-w^+(\hat w^k,\lambda^k)
\right\|^2
\right]
+
\frac{1}{32\gamma_{\lambda,k}}
\mathbb{E}\left[
\left\|
\lambda^+(\hat w^k)-\lambda^k
\right\|^2
\right]                                                                  \\
&\quad
+
\frac{p_w\beta}{8}
\mathbb{E}\left[
\left\|
w^k-\hat w^k
\right\|^2
\right]
-
48p_w\beta
\mathbb{E}\left[
\left\|
w^\star(\hat w^k)
-
w^\star(\hat w^k,\lambda^+(\hat w^k))
\right\|^2
\right]                                                                  \\
&\quad
-
\left(
\gamma_{w,k}
+
4(p_w+l_{\mathcal L,1})\gamma_{w,k}^2
+
p_w\beta\gamma_{w,k}^2
\right)
\left(
\frac{l_{g,2}^2C_\lambda^4}{4}\tau^2
+
\frac{\sigma_g^2}{B\tau^2}
\right)                                                                 \nonumber\\
&\quad
-
\frac{
4(p_w+l_{\mathcal L,1})\gamma_{w,k}^2
\sigma_f^2
}{B} - \frac{
p_w\beta\gamma_{w,k}^2\sigma_f^2
}{B}
-
\frac{9\gamma_{\lambda,k}\sigma_g^2}{8B}.
\end{align*}
This completes the proof.
\end{proof}

\subsection{Oracle Complexity under Stochastic Setting}
\begin{proposition}\label{Proposition: Stochastic Lyapunov Function Descent After Error Bound}
Suppose the conditions in Proposition \ref{Proposition: Stochastic Lyapunov Function Descent} hold. Suppose further that
\[
\beta
\le
\frac{1}{3072p_w\gamma_{\lambda,k}\gamma_3^2},
\]
where
\[
\gamma_3
:=
\frac{
l_{\Psi,1}+\gamma_{\lambda,k}^{-1}
}{
\sqrt{
2(p_w-l_{\mathcal L,1})p_\lambda
+
\frac{(p_w-l_{\mathcal L,1})^2\mu_g^2}
{(p_w+l_{\mathcal L,1})^2}
}
}.
\]
Then, we have
\begin{align*}
&\mathbb{E}[V_k]-\mathbb{E}[V_{k+1}]                                      \\
&\ge
\frac{1}{32\gamma_{w,k}}
\mathbb{E}\left[
\left\|
w^k-w^+(\hat w^k,\lambda^k)
\right\|^2
\right]
+
\frac{1}{64\gamma_{\lambda,k}}
\mathbb{E}\left[
\left\|
\lambda^+(\hat w^k)-\lambda^k
\right\|^2
\right]
+
\frac{p_w\beta}{8}
\mathbb{E}\left[
\left\|
w^k-\hat w^k
\right\|^2
\right]                                                                  \\
&\quad
-
\left(
\gamma_{w,k}
+
4(p_w+l_{\mathcal L,1})\gamma_{w,k}^2
+
p_w\beta\gamma_{w,k}^2
\right)
\left(
\frac{l_{g,2}^2C_\lambda^4}{4}\tau^2
+
\frac{\sigma_g^2}{B\tau^2}
\right)                                                                 \nonumber\\
&\quad
-
\frac{
4(p_w+l_{\mathcal L,1})\gamma_{w,k}^2
\sigma_f^2
}{B} - \frac{
p_w\beta\gamma_{w,k}^2\sigma_f^2
}{B}
-
\frac{9\gamma_{\lambda,k}\sigma_g^2}{8B}.
\end{align*}
\end{proposition}
\begin{proof}
By Proposition \ref{Proposition: Stochastic Lyapunov Function Descent}, we have
\begin{align*}
&\mathbb{E}[V_k]-\mathbb{E}[V_{k+1}]                                      \\
&\ge
\frac{1}{32\gamma_{w,k}}
\mathbb{E}\left[
\left\|
w^k-w^+(\hat w^k,\lambda^k)
\right\|^2
\right]
+
\frac{1}{32\gamma_{\lambda,k}}
\mathbb{E}\left[
\left\|
\lambda^+(\hat w^k)-\lambda^k
\right\|^2
\right]                                                                  \\
&\quad
+
\frac{p_w\beta}{8}
\mathbb{E}\left[
\left\|
w^k-\hat w^k
\right\|^2
\right]                             -
48p_w\beta
\mathbb{E}\left[
\left\|
w^\star(\hat w^k)
-
w^\star(\hat w^k,\lambda^+(\hat w^k))
\right\|^2
\right]                                                                  \\
&\quad
-
\left(
\gamma_{w,k}
+
4(p_w+l_{\mathcal L,1})\gamma_{w,k}^2
+
p_w\beta\gamma_{w,k}^2
\right)
\left(
\frac{l_{g,2}^2C_\lambda^4}{4}\tau^2
+
\frac{\sigma_g^2}{B\tau^2}
\right)                                                                 \nonumber\\
&\quad
-
\frac{
4(p_w+l_{\mathcal L,1})\gamma_{w,k}^2
\sigma_f^2
}{B} - \frac{
p_w\beta\gamma_{w,k}^2\sigma_f^2
}{B}
-
\frac{9\gamma_{\lambda,k}\sigma_g^2}{8B}.
\end{align*}
Since Lemma \ref{Lemma: Error Bound 1} is deterministic conditional on
$(\hat w^k,\lambda^k)$, it holds pathwise. Hence,
\[
\left\|
w^\star(\hat w^k)
-
w^\star(\hat w^k,\lambda^+(\hat w^k))
\right\|
\le
\gamma_3
\left\|
\lambda^+(\hat w^k)-\lambda^k
\right\|.
\]
Squaring both sides and taking expectation gives
\[
\mathbb{E}\left[
\left\|
w^\star(\hat w^k)
-
w^\star(\hat w^k,\lambda^+(\hat w^k))
\right\|^2
\right]
\le
\gamma_3^2
\mathbb{E}\left[
\left\|
\lambda^+(\hat w^k)-\lambda^k
\right\|^2
\right].
\]
Therefore, by the choice
\[
\beta
\le
\frac{1}{3072p_w\gamma_{\lambda,k}\gamma_3^2},
\]
we have
\[
48p_w\beta
\mathbb{E}\left[
\left\|
w^\star(\hat w^k)
-
w^\star(\hat w^k,\lambda^+(\hat w^k))
\right\|^2
\right]                             \le
48p_w\beta\gamma_3^2
\mathbb{E}\left[
\left\|
\lambda^+(\hat w^k)-\lambda^k
\right\|^2
\right] \leq
\frac{1}{64\gamma_{\lambda,k}}
\mathbb{E}\left[
\left\|
\lambda^+(\hat w^k)-\lambda^k
\right\|^2
\right].
\]
Substituting this estimate into the previous descent inequality gives
\begin{align*}
&\mathbb{E}[V_k]-\mathbb{E}[V_{k+1}]                                      \\
&\ge
\frac{1}{32\gamma_{w,k}}
\mathbb{E}\left[
\left\|
w^k-w^+(\hat w^k,\lambda^k)
\right\|^2
\right]
+
\frac{1}{64\gamma_{\lambda,k}}
\mathbb{E}\left[
\left\|
\lambda^+(\hat w^k)-\lambda^k
\right\|^2
\right]
+
\frac{p_w\beta}{8}
\mathbb{E}\left[
\left\|
w^k-\hat w^k
\right\|^2
\right]                                                                  \\
&\quad
-
\left(
\gamma_{w,k}
+
4(p_w+l_{\mathcal L,1})\gamma_{w,k}^2
+
p_w\beta\gamma_{w,k}^2
\right)
\left(
\frac{l_{g,2}^2C_\lambda^4}{4}\tau^2
+
\frac{\sigma_g^2}{B\tau^2}
\right)                                                                 \nonumber\\
&\quad
-
\frac{
4(p_w+l_{\mathcal L,1})\gamma_{w,k}^2
\sigma_f^2
}{B} - \frac{
p_w\beta\gamma_{w,k}^2\sigma_f^2
}{B}
-
\frac{9\gamma_{\lambda,k}\sigma_g^2}{8B}.
\end{align*}
\end{proof}

\begin{lemma}
\label{Lemma: Stochastic Proximal Residual Bound}
Under Assumptions \ref{Assumption: Bilevel Optimization} and \ref{Assumption: Stochastic First-order Oracle}, suppose that $\Lambda=\{\lambda:\|\lambda\|\le C_\lambda\}$ with $C_{\lambda} > \frac{l_{f,0}}{\mu_g} $. Set $p_w=2l_{\mathcal L,1}$.
Further choose constant stepsizes
\[
\gamma_{w,k} \leq \min\left\{
\frac{1}{8p_w\beta},
\frac{1}{8(p_w+l_{\mathcal L,1})}
\right\},
\quad
\gamma_{\lambda,k} \le
\min\left\{
\frac{1}{2\left(2l_{\Psi,1}+p_\lambda + \frac{l_{g,1}^2}{p_w+l_{\mathcal{L},1}}\right)},
\frac{l_{\mathcal L,1}^2}{16l_{g,1}^2}\gamma_{w,k},
\frac{1}{1536p_w\beta\gamma_2^2}
\right\}.
\]
Moreover, choose the averaging parameter
\[
\beta
\le
\min\left\{
\frac{\sqrt{5}-1}{48\gamma_1},
\frac{1}{3072p_w\gamma_{\lambda,k}\gamma_3^2}
\right\}.
\]
Define
\[
a
:=
\min\left\{
\frac{1}{32\gamma_{w,k}},
\frac{\bar{\kappa}_y^2}{64\gamma_{\lambda,k}},
\frac{\bar{\kappa}_y^2p_w\beta}{8}
\right\}.
\]
Also define
\[
C_g
:=
\gamma_{w,k}
+
4(p_w+l_{\mathcal L,1})\gamma_{w,k}^2
+
p_w\beta\gamma_{w,k}^2,
\]
and
\[
C_f
:=
4(p_w+l_{\mathcal L,1})\gamma_{w,k}^2 + p_w\beta\gamma_{w,k}^2.
\]
Without loss of generality, take the declared Hessian--Lipschitz upper
bound $l_{g,2}>0$ (it may always be enlarged). Choose
\[
\tau^2
=
\frac{4a}{5C_g l_{g,2}^2C_\lambda^4}
\epsilon_{\mathcal L}^2
\]
and the positive integer batch size
\[
B
=
\max\left\{
1,
\left\lceil
\max\left\{
\frac{5C_g\sigma_g^2}{a\tau^2\epsilon_{\mathcal L}^2},
\frac{5C_f\sigma_f^2}{a\epsilon_{\mathcal L}^2},
\frac{45\gamma_{\lambda,k}\sigma_g^2}{8a\epsilon_{\mathcal L}^2}
\right\}
\right\rceil
\right\},
\]
and run Algorithm~\ref{Stoc-SGHA} for
\[
K
=
\max\left\{
1,
\left\lceil
\frac{5\left(\mathbb E[V_0]-\underline f\right)}
{a\epsilon_{\mathcal L}^2}
\right\rceil
\right\}
\]
iterations. More generally, the same conclusion holds if $\tau^2$ is
chosen no larger than its displayed value and $B$ and $K$ are chosen as
positive integers no smaller than their displayed values, with the
right-hand side for $B$ evaluated at the selected $\tau$. Then there
exists an index $k\in\{0,\ldots,K-1\}$ such that
\[
\max\left\{
\mathbb E\left[
\|w^k-w^+(\hat w^k,\lambda^k)\|^2
\right],
\bar{\kappa}_y^{-2}\mathbb E\left[
\|\lambda^+(\hat w^k)-\lambda^k\|^2
\right],
\bar{\kappa}_y^{-2}\mathbb E\left[
\|w^k-\hat w^k\|^2
\right]
\right\}
\le
\epsilon_{\mathcal L}^2.
\]
\end{lemma}
\begin{proof}
By Proposition \ref{Proposition: Stochastic Lyapunov Function Descent After Error Bound}, we have
\[
\begin{aligned}
\mathbb E[V_k]-\mathbb E[V_{k+1}]
\ge\;&
\frac{1}{32\gamma_{w,k}}
\mathbb E\left[
\|w^k-w^+(\hat w^k,\lambda^k)\|^2
\right]                                +
\frac{1}{64\gamma_{\lambda,k}}
\mathbb E\left[
\|\lambda^+(\hat w^k)-\lambda^k\|^2
\right]                             +
\frac{p_w\beta}{8}
\mathbb E\left[
\|w^k-\hat w^k\|^2
\right]                                                        \\
&-
C_g
\left(
\frac{l_{g,2}^2C_\lambda^4}{4}\tau^2
+
\frac{\sigma_g^2}{B\tau^2}
\right)
-
\frac{C_f\sigma_f^2}{B}
-
\frac{9\gamma_{\lambda,k}\sigma_g^2}{8B}.
\end{aligned}
\]
By the definition of $a$,
\[
a
=
\min\left\{
\frac{1}{32\gamma_{w,k}},
\frac{\bar{\kappa}_y^2}{64\gamma_{\lambda,k}},
\frac{\bar{\kappa}_y^2p_w\beta}{8}
\right\},
\]
we obtain
\[
\begin{aligned}
\mathbb E[V_k]-\mathbb E[V_{k+1}]
\ge\;&
a
\max\left\{
\mathbb E\left[
\|w^k-w^+(\hat w^k,\lambda^k)\|^2
\right],
\bar{\kappa}_y^{-2}\mathbb E\left[
\|\lambda^+(\hat w^k)-\lambda^k\|^2
\right],
\bar{\kappa}_y^{-2}\mathbb E\left[
\|w^k-\hat w^k\|^2
\right]
\right\}                                                       \\
&-
C_g
\left(
\frac{l_{g,2}^2C_\lambda^4}{4}\tau^2
+
\frac{\sigma_g^2}{B\tau^2}
\right)
-
\frac{C_f\sigma_f^2}{B}
-
\frac{9\gamma_{\lambda,k}\sigma_g^2}{8B}.
\end{aligned}
\]
Summing from $k=0$ to $K-1$ gives
\[
\begin{aligned}
\mathbb E[V_0]-\underline{f}
\ge\;&
a
\sum_{k=0}^{K-1}
\max\left\{
\mathbb E\left[
\|w^k-w^+(\hat w^k,\lambda^k)\|^2
\right],
\bar{\kappa}_y^{-2}\mathbb E\left[
\|\lambda^+(\hat w^k)-\lambda^k\|^2
\right],
\bar{\kappa}_y^{-2}\mathbb E\left[
\|w^k-\hat w^k\|^2
\right]
\right\}                                                       \\
&-
K C_g
\left(
\frac{l_{g,2}^2C_\lambda^4}{4}\tau^2
+
\frac{\sigma_g^2}{B\tau^2}
\right)
-
K\frac{C_f\sigma_f^2}{B}
-
K\frac{9\gamma_{\lambda,k}\sigma_g^2}{8B}.
\end{aligned}
\]
Therefore, there exists $k\in\{0,\ldots,K-1\}$ such that
\[
\begin{aligned}
&\max\left\{
\mathbb E\left[
\|w^k-w^+(\hat w^k,\lambda^k)\|^2
\right],
\bar{\kappa}_y^{-2}\mathbb E\left[
\|\lambda^+(\hat w^k)-\lambda^k\|^2
\right],
\bar{\kappa}_y^{-2}\mathbb E\left[
\|w^k-\hat w^k\|^2
\right]
\right\}                                                       \\
&\le
\frac{\mathbb E[V_0]-\underline{f}}{aK}
+
\frac{C_g}{a}
\left(
\frac{l_{g,2}^2C_\lambda^4}{4}\tau^2
+
\frac{\sigma_g^2}{B\tau^2}
\right)
+
\frac{C_f\sigma_f^2}{aB}
+
\frac{9\gamma_{\lambda,k}\sigma_g^2}{8aB}.
\end{aligned}
\]
By choosing
\[
K
\ge
\left\lceil
\frac{5\left(\mathbb E[V_0]-\underline{f}\right)}
{a\epsilon_{\mathcal L}^2}
\right\rceil,
\]
we have
\[
\frac{\mathbb E[V_0]-\underline{f}}{aK}
\le \frac{1}{5}\epsilon_{\mathcal L}^2.
\]
By the choice of $\tau$,
\[
\tau^2
\le
\frac{4a}{5C_g l_{g,2}^2C_\lambda^4}
\epsilon_{\mathcal L}^2,
\]
we have
\[
\frac{C_g}{a}
\frac{l_{g,2}^2C_\lambda^4}{4}\tau^2
\le \frac{1}{5}\epsilon_{\mathcal L}^2.
\]
By the choice of $B$,
\[
B
\ge
\max\left\{
\frac{5C_g\sigma_g^2}{a\tau^2\epsilon_{\mathcal L}^2},
\frac{5C_f\sigma_f^2}{a\epsilon_{\mathcal L}^2},
\frac{45\gamma_{\lambda,k}\sigma_g^2}{8a\epsilon_{\mathcal L}^2}
\right\},
\]
we have
\[
\frac{C_g\sigma_g^2}{aB\tau^2}
\le
\frac{1}{5}\epsilon_{\mathcal L}^2,\quad \frac{C_f\sigma_f^2}{aB}
\le
\frac{1}{5}\epsilon_{\mathcal L}^2,\quad \frac{9\gamma_{\lambda,k}\sigma_g^2}{8aB}
\le
\frac{1}{5}\epsilon_{\mathcal L}^2.
\]
Combining these estimates yields
\[
\max\left\{
\mathbb E\left[
\|w^k-w^+(\hat w^k,\lambda^k)\|^2
\right],
\bar{\kappa}_y^{-2}\mathbb E\left[
\|\lambda^+(\hat w^k)-\lambda^k\|^2
\right],
\bar{\kappa}_y^{-2}\mathbb E\left[
\|w^k-\hat w^k\|^2
\right]
\right\}
\le
\epsilon_{\mathcal L}^2.
\]
This completes the proof.
\end{proof}

\subsection{Oracle Complexity of Stoc-SGHA in High Probability}

\paragraph{Proof of Theorem
\ref{Theorem: Stochastic Sample Complexity HP}.}
Define
\(
\epsilon_{\mathcal L}
:=
\frac{\epsilon\sqrt{\rho}}
{2^{12}(1+\bar L)\bar\kappa_y^3}\) and \(
L_0:=8\bar L\bar\kappa_y\).
Then
\[
p_\lambda
=
\frac18
\min\left\{
\mu_g,\frac{\epsilon_{\mathcal L}}{\bar\kappa_y}
\right\},
\qquad
p_\lambda C_\lambda
\le\frac14\epsilon_{\mathcal L},
\qquad
\tau
=
\frac{\epsilon_{\mathcal L}}
{2^{13}\bar\kappa_y}.
\]

The verification of the common stepsize and averaging conditions is
identical to that in the proof of
Theorem~\ref{Theorem: Deterministic Sample Complexity}.
In particular, with $L_0$ in place of $l_{\mathcal L,1}$,
\[
p_w=2L_0,\qquad
\gamma_1=2,\qquad
\gamma_2\le\frac{1}{8\bar\kappa_y},\qquad
l_{\Psi,1}\le\frac{\bar L}{4\bar\kappa_y},\qquad
\gamma_3\le\frac{3075}{4}.
\]
Hence, the parameters satisfy all the deterministic conditions required
in Lemma~\ref{Lemma: Stochastic Proximal Residual Bound}.

Define
\[
a
:=
\min\left\{
\frac{1}{32\gamma_w},
\frac{\bar\kappa_y^2}{64\gamma_\lambda},
\frac{\bar\kappa_y^2p_w\beta}{8}
\right\}
=
\frac{\bar L\bar\kappa_y}{2^{29}},
\]
and
\[
C_g
:=
\gamma_w
+
4(p_w+L_0)\gamma_w^2
+
p_w\beta\gamma_w^2,
\qquad
C_f
:=
4(p_w+L_0)\gamma_w^2
+
p_w\beta\gamma_w^2.
\]
Direct calculation gives
\[
C_g
=
\frac{22+\beta}{2^{12}\bar L\bar\kappa_y}
\le
\frac{23}{2^{12}\bar L\bar\kappa_y},
\qquad
C_f
=
\frac{6+\beta}{2^{12}\bar L\bar\kappa_y}
\le
\frac{7}{2^{12}\bar L\bar\kappa_y}.
\]

Using $l_{g,2}\le\bar L$ and
$C_\lambda=2\bar\kappa_y$, our choice of $\tau$ satisfies the condition in Lemma~\ref{Lemma: Stochastic Proximal Residual Bound}:
\[
\frac{C_gl_{g,2}^2C_\lambda^4}{4a}\tau^2
\le
\frac{23}{128}\epsilon_{\mathcal L}^2
<
\frac15\epsilon_{\mathcal L}^2.
\]

Moreover, the batch size in the theorem satisfies
\(
B
\ge
\frac{
2^{50}\bar\sigma^2
}{
\bar L^2\epsilon_{\mathcal L}^4
}\).
Using the bounds on $a$, $C_g$, and $C_f$, we obtain
\[
\frac{5C_g\sigma_g^2}
{a\tau^2\epsilon_{\mathcal L}^2}
\le
\frac{
115\cdot2^{43}\sigma_g^2
}{
\bar L^2\epsilon_{\mathcal L}^4
}
\le B,
\]
as well as
\[
\frac{5C_f\sigma_f^2}{a\epsilon_{\mathcal L}^2}
\le
\frac{
35\cdot2^{17}\sigma_f^2
}{
\bar L^2\bar\kappa_y^2\epsilon_{\mathcal L}^2
}
\le B,
\qquad
\frac{45\gamma_\lambda\sigma_g^2}
{8a\epsilon_{\mathcal L}^2}
\le
\frac{
45\cdot2^{18}\sigma_g^2
}{
\bar L^2\epsilon_{\mathcal L}^2
}
\le B,
\]
where we used $\bar\kappa_y\ge1$,
$\epsilon_{\mathcal L}\le1$, and
$\bar\sigma^2\ge\sigma_f^2,\sigma_g^2$.

By Lemma~\ref{Lemma: Stochastic Proximal Residual Bound}, the integer
$K$ specified in the theorem satisfies
\[
K\ge
\left\lceil
\frac{5\Delta_V}{a\epsilon_{\mathcal L}^2}
\right\rceil
=
\left\lceil
\frac{5\cdot 2^{53}(1+\bar L)^2\Delta_V}{\bar L}
\bar\kappa_y^5\epsilon^{-2}\rho^{-1}
\right\rceil.
\]
Consequently, there exists
$k_\star\in\{0,\ldots,K-1\}$ such that
\begin{align*}
\max\biggl\{
\mathbb E\left[
\|w^{k_\star}
-w^+(\hat w^{k_\star},\lambda^{k_\star})\|^2
\right],
\bar\kappa_y^{-2}
\mathbb E\left[
\|\lambda^+(\hat w^{k_\star})-\lambda^{k_\star}\|^2
\right],\,
\bar\kappa_y^{-2}
\mathbb E\left[
\|w^{k_\star}-\hat w^{k_\star}\|^2
\right]
\biggr\}
\le
\epsilon_{\mathcal L}^2.
\end{align*}

Set
\(
\delta
:=
\frac{\sqrt{3}\epsilon_{\mathcal L}}{\sqrt{\rho}}\).
By Markov's inequality and a union bound, with probability at least
$1-\rho$, we have
\[
\|w^{k_\star}
-w^+(\hat w^{k_\star},\lambda^{k_\star})\|
\le\delta,\qquad
\|\lambda^+(\hat w^{k_\star})-\lambda^{k_\star}\|
\le\bar\kappa_y\delta,
\qquad
\|w^{k_\star}-\hat w^{k_\star}\|
\le\bar\kappa_y\delta.
\]

Since
\[
\delta
=
\frac{\sqrt{3}\epsilon}
{2^{12}(1+\bar L)\bar\kappa_y^3}
\le
\frac{\epsilon}
{2^{11}(1+\bar L)\bar\kappa_y^3},
\]
the conditions of Lemma~\ref{Lemma: Deterministic Stationary Point Bridge} that are verified in the proof of
Theorem~\ref{Theorem: Deterministic Sample Complexity}
hold with $\delta$ in place of $\epsilon_{\mathcal L}$.
Also,
\[
p_\lambda C_\lambda
\le
\frac14\epsilon_{\mathcal L}
\le
\frac14\delta.
\]
Thus, Lemma~\ref{Lemma: Deterministic Stationary Point Bridge}
applies pathwise and gives
\[
\|\nabla_w
\mathcal L(w^{k_\star},\lambda^{k_\star})\|
\le
272\bar L\bar\kappa_y^2\delta.
\]
Furthermore,
\[
\frac{l_{g,1}}{(p_w-L_0)\gamma_w}
+
\bar\kappa_y\gamma_\lambda^{-1}
=
32l_{g,1}+2^8\bar L
\le
288\bar L,
\]
and hence
\[
\|\nabla_\lambda
\mathcal L(w^{k_\star},\lambda^{k_\star})\|
\le
\left(
288\bar L+\frac14
\right)\delta.
\]

As shown in the proof of
Theorem~\ref{Theorem: Deterministic Sample Complexity}, we have
\(
\frac{L_{\bar F,y}}{\mu_g}
\le
4\bar\kappa_y^3\).
Therefore, Theorem~\ref{Theorem: Approximation of Hypergradient}
implies, on the same event,
\begin{align*}
\|\nabla F(x^{k_\star})\|
&\le
(1+\bar\kappa_y)
272\bar L\bar\kappa_y^2\delta
+
4\bar\kappa_y^3
\left(
288\bar L+\frac14
\right)\delta\\
&\le
2^{11}(1+\bar L)\bar\kappa_y^3\delta =
\frac{\sqrt{3}}{2}\epsilon
<
\epsilon.
\end{align*}
Thus,
\[
\mathbb P\left(
\|\nabla F(x^{k_\star})\|\le\epsilon
\right)
\ge1-\rho.
\]

Finally, the total oracle complexity is
\[
K B
=
\mathcal O\left(
\frac{
(1+\Delta_V)\bar\sigma^2(1+\bar L)^6
}{
\bar L^3
}
\bar\kappa_y^{17}\epsilon^{-6}\rho^{-3}
\right)
=
\mathcal O\left(
\bar\kappa_y^{17}\epsilon^{-6}\rho^{-3}
\right).
\]
This completes the proof.

\subsection{Oracle Complexity of Stoc-SGHA in Expectation}
\begin{lemma}[Bounded Lagrangian and Bilevel Gradients on Bounded Iterates]
\label{Lemma: Bounded Lagrangian Gradient}
Under Assumption \ref{Assumption: Bilevel Optimization}, suppose that $\Lambda=\{\lambda:\|\lambda\|\le C_\lambda\}$ with $C_{\lambda} > \frac{l_{f,0}}{\mu_g} $. Suppose further that
there exists a constant $D_w>0$ such that
\[
\|w^k\|\le D_w,
\quad \text{almost surely}.
\]
Define
\[
G_f(D_w):=\sup_{\|w\|\le D_w}\|\nabla f(w)\|,\quad
G_g(D_w):=\sup_{\|w\|\le D_w}\|\nabla_y g(w)\|,
\]
and define
\[
G_J(D_w):=\sup_{\|w\|\le D_w}
\|\nabla_w(\nabla_y g(w))\|.
\]
Then these constants are finite. Moreover, with
\[
G_w:=G_f(D_w)+G_J(D_w)C_\lambda,
\qquad
G_\lambda:=G_g(D_w),
\]
we have
\[
\|\nabla_w\mathcal L(w^k,\lambda^k)\|\le G_w,
\quad
\|\nabla_\lambda\mathcal L(w^k,\lambda^k)\|\le G_\lambda,
\qquad
\forall k\ge 0
\]
almost surely.
In addition, the true bilevel hypergradient satisfies
\[
\|\nabla F(x^k)\|
\le
G_F
:=
G_f(D_w)
+
\frac{l_{f,1}}{\mu_g}G_g(D_w)
+
\frac{l_{g,1}}{\mu_g}l_{f,0}
\]
almost surely.
\end{lemma}
\begin{proof}
Recall that
\[
\mathcal L(w,\lambda)
=
f(w)+\lambda^\top\nabla_y g(w).
\]
Thus,
\[
\nabla_w\mathcal L(w,\lambda)
=
\nabla f(w)+J(w)^\top\lambda,
\quad
J(w):=\nabla_w(\nabla_y g(w)),\quad \nabla_\lambda\mathcal L(w,\lambda)=\nabla_y g(w).
\]
Recall the definitions of $G_f(D_w)$, $G_g(D_w)$, and $G_J(D_w)$
from the statement.
Since $f$ and $g$ are smooth, the mappings
\[
w\mapsto \nabla f(w),
\quad
w\mapsto \nabla_y g(w),
\quad
w\mapsto J(w)
\]
are continuous. Since the set
\[
\{w:\|w\|\le D_w\}
\]
is compact, the constants $G_f(D_w)$, $G_g(D_w)$, and $G_J(D_w)$ are finite.

Therefore, for every $k\ge 0$,
\[
\|\nabla_w\mathcal L(w^k,\lambda^k)\|
\le
G_f(D_w)+G_J(D_w)C_\lambda,
\quad
\|\nabla_\lambda\mathcal L(w^k,\lambda^k)\|
\le G_g(D_w).
\]
Taking
\[
G_w:=G_f(D_w)+G_J(D_w)C_\lambda,
\quad
G_\lambda:=G_g(D_w),
\]
we obtain
\[
\|\nabla_w\mathcal L(w^k,\lambda^k)\|\le G_w,
\quad
\|\nabla_\lambda\mathcal L(w^k,\lambda^k)\|\le G_\lambda.
\]
Let $y_k^\star:=y^\star(x^k)$. By lower-level strong convexity,
\[
\|y^k-y_k^\star\|
\le
\frac{1}{\mu_g}\|\nabla_y g(x^k,y^k)\|
\le
\frac{G_g(D_w)}{\mu_g}.
\]
Using the implicit hypergradient formula, the smoothness of $f$ and $g$,
and $\|\nabla_y f(x,y)\|\le l_{f,0}$, we obtain
\begin{align*}
\|\nabla F(x^k)\|
&\le
\|\nabla_x f(x^k,y_k^\star)\|
+
\frac{l_{g,1}}{\mu_g}l_{f,0}\\
&\le
\|\nabla_x f(x^k,y^k)\|
+
l_{f,1}\|y^k-y_k^\star\|
+
\frac{l_{g,1}}{\mu_g}l_{f,0}\\
&\le
G_f(D_w)
+
\frac{l_{f,1}}{\mu_g}G_g(D_w)
+
\frac{l_{g,1}}{\mu_g}l_{f,0}
=G_F.
\end{align*}
This completes the proof.
\end{proof}

\begin{lemma}[Expected stationary point bridge]
\label{Lemma: Expected Stationary Point Bridge}
Under Assumptions~\ref{Assumption: Bilevel Optimization} and
\ref{Assumption: Stochastic First-order Oracle}, suppose that
\(\Lambda=\{\lambda:\|\lambda\|\le C_\lambda\}\) with
\(C_\lambda>l_{f,0}/\mu_g\). Set \(p_w=2l_{\mathcal L,1}\).
Assume further that there exists a constant \(D_w>0\) such that
\[
    \|w^k\|\le D_w
\]
almost surely. Let $G_w$, $G_\lambda$, and $G_F$ be the constants in
Lemma~\ref{Lemma: Bounded Lagrangian Gradient}. Define
\[
A_w:=\gamma_{w,k}^{-1}+\bar{\kappa}_yp_w,\quad
A_\lambda
:=
\frac{l_{g,1}}{(p_w-l_{\mathcal L,1})\gamma_{w,k}}
+
\bar{\kappa}_y\gamma_{\lambda,k}^{-1}.
\]
Also define
\[
r_0
:=
\min\left\{
\frac{\mu_g}{4A_w}
\left(C_\lambda-\frac{l_{f,0}}{\mu_g}\right),
\frac{1}{4\bar{\kappa}_y}
\left(C_\lambda-\frac{l_{f,0}}{\mu_g}\right)
\right\}.
\]
If
\[
\mathbb E\left[\|w^k-w^+(\hat w^k,\lambda^k)\|^2\right]
\le\epsilon_{\mathcal L}^2,\quad
\mathbb E\left[\|\lambda^+(\hat w^k)-\lambda^k\|^2\right]
\le\bar\kappa_y^2\epsilon_{\mathcal L}^2,
\quad
\mathbb E\left[\|w^k-\hat w^k\|^2\right]
\le\bar\kappa_y^2\epsilon_{\mathcal L}^2,
\]
then, we have
\[
\mathbb E\left[
\|\nabla_w\mathcal L(w^k,\lambda^k)\|
\right]
\le
A_w\epsilon_{\mathcal L}+
\frac{3G_w}{r_0^2}\epsilon_{\mathcal L}^2,\quad
\mathbb E\left[
\|\nabla_\lambda\mathcal L(w^k,\lambda^k)\|
\right]
\le
A_\lambda\epsilon_{\mathcal L}+p_\lambda C_\lambda+
\frac{3G_\lambda}{r_0^2}\epsilon_{\mathcal L}^2.
\]
Furthermore,
\begin{align*}
\mathbb E\left[\|\nabla F(x^k)\|\right]
\le{}&
(1+\kappa_y)A_w\epsilon_{\mathcal L}\\
&+
\frac{L_{\bar F,y}}{\mu_g}
\left(A_\lambda\epsilon_{\mathcal L}+p_\lambda C_\lambda\right)
+
\frac{3G_F}{r_0^2}\epsilon_{\mathcal L}^2.
\end{align*}
\end{lemma}
\begin{proof}
Define the good event
\[
\mathcal G_k
:=
\left\{
\begin{aligned}
&\|w^k-w^+(\hat w^k,\lambda^k)\|\le r_0,\\
&\|\lambda^+(\hat w^k)-\lambda^k\|\le\bar\kappa_y r_0,\\
&\|w^k-\hat w^k\|\le\bar\kappa_y r_0
\end{aligned}
\right\}.
\]
By the definition of \(r_0\), the interiority conditions in
Lemma~\ref{Lemma: Deterministic Stationary Point Bridge} hold on \(\mathcal G_k\). Therefore, the proof of
Lemma~\ref{Lemma: Deterministic Stationary Point Bridge} gives
\[
\|\nabla_w\mathcal L(w^k,\lambda^k)\|
\le
\gamma_{w,k}^{-1}\|w^k-w^+(\hat w^k,\lambda^k)\|
+p_w\|w^k-\hat w^k\|,
\]
and
\[
\|\nabla_\lambda\mathcal L(w^k,\lambda^k)\|
\le
\frac{l_{g,1}}{(p_w-l_{\mathcal L,1})\gamma_{w,k}}
\|w^k-w^+(\hat w^k,\lambda^k)\|+\gamma_{\lambda,k}^{-1}
\|\lambda^+(\hat w^k)-\lambda^k\|
+p_\lambda C_\lambda.
\]
It follows that
\[
\mathbb E\left[
\|\nabla_w\mathcal L(w^k,\lambda^k)\|
\mathbf 1_{\mathcal G_k}
\right]
\le
A_w\epsilon_{\mathcal L},\quad
\mathbb E\left[
\|\nabla_\lambda\mathcal L(w^k,\lambda^k)\|
\mathbf 1_{\mathcal G_k}
\right]
\le
A_\lambda\epsilon_{\mathcal L}+p_\lambda C_\lambda.
\]
On the complement \(\mathcal G_k^c\), Lemma~\ref{Lemma: Bounded Lagrangian Gradient} bounds the primal and multiplier gradients by \(G_w\) and \(G_\lambda\), respectively. Moreover, Markov's inequality gives
\[
\begin{aligned}
\mathbb P(\mathcal G_k^c)
&\le
\frac{1}{r_0^2}
\mathbb E\left[\|w^k-w^+(\hat w^k,\lambda^k)\|^2\right]+\frac{1}{\bar\kappa_y^2r_0^2}
\mathbb E\left[\|\lambda^+(\hat w^k)-\lambda^k\|^2\right]+\frac{1}{\bar\kappa_y^2r_0^2}
\mathbb E\left[\|w^k-\hat w^k\|^2\right]\\
&\le\frac{3\epsilon_{\mathcal L}^2}{r_0^2}.
\end{aligned}
\]
Consequently, the primal and multiplier components contribute at most
\(3G_w\epsilon_{\mathcal L}^2/r_0^2\) and
\(3G_\lambda\epsilon_{\mathcal L}^2/r_0^2\), respectively, on
\(\mathcal G_k^c\).
Combining the estimates on \(\mathcal G_k\) and \(\mathcal G_k^c\) separately for the primal and multiplier components gives
\[
\mathbb E\left[
\|\nabla_w\mathcal L(w^k,\lambda^k)\|
\right]
\le
A_w\epsilon_{\mathcal L}
+
\frac{3G_w}{r_0^2}\epsilon_{\mathcal L}^2.
\]
Likewise,
\[
\mathbb E\left[
\|\nabla_\lambda\mathcal L(w^k,\lambda^k)\|
\right]
\le
A_\lambda\epsilon_{\mathcal L}+p_\lambda C_\lambda+
\frac{3G_\lambda}{r_0^2}\epsilon_{\mathcal L}^2.
\]
On the good event $\mathcal G_k$, Theorem~\ref{Theorem: Approximation of Hypergradient}
and the two pathwise bounds above imply
\begin{align*}
\mathbb E\left[
\|\nabla F(x^k)\|\mathbf 1_{\mathcal G_k}
\right]\le
(1+\kappa_y)A_w\epsilon_{\mathcal L}
+
\frac{L_{\bar F,y}}{\mu_g}
\left(A_\lambda\epsilon_{\mathcal L}+p_\lambda C_\lambda\right).
\end{align*}
On $\mathcal G_k^c$, Lemma~\ref{Lemma: Bounded Lagrangian Gradient}
gives $\|\nabla F(x^k)\|\le G_F$. Combining this with
$\mathbb P(\mathcal G_k^c)\le3\epsilon_{\mathcal L}^2/r_0^2$
proves the stated hypergradient bound.
This completes the proof.
\end{proof}

\paragraph{Proof of Theorem
\ref{Theorem: Stochastic Sample Complexity Expectation}.}
By Lemma~\ref{Lemma: Bounded Lagrangian Gradient} and
$C_\lambda=2\bar\kappa_y$, choose the problem-dependent constant
$\mathcal G$ in the theorem large enough that
$\mathcal G\ge\max\{1,G_w/\bar\kappa_y,G_\lambda\}$. Define \(
\epsilon_{\mathcal L}
:=
\frac{\epsilon}
{2^{13}(1+\bar L)\mathcal G\bar\kappa_y^3} \) and \(
L_0:=8\bar L\bar\kappa_y\).
Then the parameters in the theorem satisfy
\[
p_\lambda
=
\frac18
\min\left\{
\mu_g,\frac{\epsilon_{\mathcal L}}{\bar\kappa_y}
\right\},
\qquad
p_\lambda C_\lambda
\le\frac14\epsilon_{\mathcal L},
\qquad
\tau
=
\frac{\epsilon_{\mathcal L}}
{2^{13}\bar\kappa_y}.
\]
Define \[ a := \min\left\{ \frac{1}{32\gamma_w}, \frac{\bar\kappa_y^2}{64\gamma_\lambda}, \frac{\bar\kappa_y^2p_w\beta}{8} \right\} = \frac{\bar L\bar\kappa_y}{2^{29}}.\]

The verification of the required conditions of Lemma~\ref{Lemma: Stochastic Proximal Residual Bound} is identical
to that in the proof of
Theorem~\ref{Theorem: Stochastic Sample Complexity HP}. In particular, with $L_0$ in place of $l_{\mathcal L,1}$,
\[ p_w=2L_0, \qquad \gamma_1=2,\qquad \gamma_2\le\frac{1}{8\bar\kappa_y},\qquad l_{\Psi,1}\le\frac{\bar L}{4\bar\kappa_y},\qquad \gamma_3\le\frac{3075}{4}. \]

By Lemma~\ref{Lemma: Stochastic Proximal Residual Bound}, the integer
$K$ specified in the theorem satisfies
\[
K\ge
\left\lceil
\frac{5\Delta_V}{a\epsilon_{\mathcal L}^2}
\right\rceil
=
\left\lceil
\frac{5\cdot2^{55}(1+\bar L)^2\mathcal G^2\Delta_V}{\bar L}
\bar\kappa_y^5\epsilon^{-2}
\right\rceil.
\]
Consequently, there exists an index
$k_\star\in\{0,\ldots,K-1\}$ such that
\begin{align}
\max\biggl\{&
\mathbb E\left[
\|w^{k_\star}
-w^+(\hat w^{k_\star},\lambda^{k_\star})\|^2
\right],
\bar\kappa_y^{-2}
\mathbb E\left[
\|\lambda^+(\hat w^{k_\star})-\lambda^{k_\star}\|^2
\right],
\,
\bar\kappa_y^{-2}
\mathbb E\left[
\|w^{k_\star}-\hat w^{k_\star}\|^2
\right]
\biggr\}
\le
\epsilon_{\mathcal L}^2.
\label{eq:expectation-residual-bound}
\end{align}

We now apply
Lemma~\ref{Lemma: Expected Stationary Point Bridge}.
Its constants satisfy
\[
A_w
=
\gamma_w^{-1}+\bar\kappa_yp_w
\le
272\bar L\bar\kappa_y^2,\qquad
A_\lambda
=
\frac{l_{g,1}}{(p_w-L_0)\gamma_w}
+
\bar\kappa_y\gamma_\lambda^{-1}
\le
288\bar L.
\]
Furthermore, since
$
C_\lambda-\frac{l_{f,0}}{\mu_g}
\ge\bar\kappa_y$ and $
\mu_g=\frac{\bar L}{\bar\kappa_y}$,
the radius in
Lemma~\ref{Lemma: Expected Stationary Point Bridge}
satisfies
\begin{align*}
r_0
=
\min\left\{
\frac{\mu_g}{4A_w}
\left(
C_\lambda-\frac{l_{f,0}}{\mu_g}
\right),
\frac{1}{4\bar\kappa_y}
\left(
C_\lambda-\frac{l_{f,0}}{\mu_g}
\right)
\right\} \ge
\min\left\{
\frac{1}{1088\bar\kappa_y^2},\frac14
\right\}
=
\frac{1}{1088\bar\kappa_y^2}.
\end{align*}
Consequently,
\[
\frac{3}{r_0^2}
\le
3(1088)^2\bar\kappa_y^4
<
2^{22}\bar\kappa_y^4.
\]

The bounded-hypergradient constant in
Lemma~\ref{Lemma: Bounded Lagrangian Gradient} satisfies
\begin{align*}
G_F
&=
G_f(D_w)
+
\frac{l_{f,1}}{\mu_g}G_g(D_w)
+
\frac{l_{g,1}}{\mu_g}l_{f,0}\\
&\le
2(1+\bar L)\mathcal G\bar\kappa_y.
\end{align*}
Indeed, the preceding choice of $\mathcal G$ gives
$G_f(D_w)\le\mathcal G\bar\kappa_y$ and
$G_g(D_w)\le\mathcal G$, while
$l_{f,1}/\mu_g\le\bar\kappa_y$,
$l_{g,1}/\mu_g\le\bar\kappa_y$, and
$l_{f,0}\le\bar L$.

As shown in the proof of
Theorem~\ref{Theorem: Deterministic Sample Complexity},
$L_{\bar F,y}/\mu_g\le4\bar\kappa_y^3$.
Using $p_\lambda C_\lambda\le\epsilon_{\mathcal L}/4$ and applying
the hypergradient conclusion of
Lemma~\ref{Lemma: Expected Stationary Point Bridge}, we obtain
\begin{align*}
\mathbb E\left[\|\nabla F(x^{k_\star})\|\right]
&\le
2^{11}(1+\bar L)\bar\kappa_y^3\epsilon_{\mathcal L}
+
2^{23}(1+\bar L)\mathcal G
\bar\kappa_y^5\epsilon_{\mathcal L}^2\\
&=
\frac{\epsilon}{4\mathcal G}
+
\frac{\epsilon^2}
{8(1+\bar L)\mathcal G\bar\kappa_y}\\
&\le
\frac38\epsilon
<
\epsilon,
\end{align*}
where the last inequality uses
$0<\epsilon\le1$, $\mathcal G\ge1$, and
$\bar\kappa_y\ge1$.
Finally, the total oracle complexity is
\[
KB
=
\mathcal O\left(
\frac{
(1+\Delta_V)\bar\sigma^2(1+\bar L)^6\mathcal G^6
}{
\bar L^3
}
\bar\kappa_y^{17}\epsilon^{-6}
\right)
=
\mathcal O\left(
\bar\kappa_y^{17}\epsilon^{-6}
\right),
\]
where the last expression treats the remaining problem-dependent
constants as fixed. This completes the proof.

\section{Improved Complexities with Lower-Level Stochastic Smoothness}
\subsection{Properties of Stochastic Finite Difference Approximation with Stochastic Smoothness of \texorpdfstring{\(g\)}{g}}
\begin{lemma}\label{Lemma: Stochastic Finite Difference Estimator with g Smoothness}
Under Assumptions \ref{Assumption: Bilevel Optimization}, \ref{Assumption: Stochastic First-order Oracle} and \ref{Assumption: Stochastic Smoothness of g}, suppose that
\[
\lambda\in\Lambda=\{\lambda:\|\lambda\|\le C_\lambda\},
\]
we have
\begin{equation}\label{Lemma: Stoc FD 1 with g Smoothness}
\mathbb{E}\left[
\left\|
\nabla_{xy}^2 g(x,y)\lambda
-
\tilde{\nabla}_{xy}^2 g_{\lambda}(x,y;\mathcal{B}_g)
\right\|^2
\right]
\le
\frac{l_{g,2}^2C_{\lambda}^4}{4}\tau^2
+
\frac{\tilde{l}_{g,1} C_{\lambda}^2}{B},
\end{equation}
and
\begin{equation}\label{Lemma: Stoc FD 2 with g Smoothness}
\mathbb{E}\left[
\left\|
\nabla_{yy}^2 g(x,y)\lambda
-
\tilde{\nabla}_{yy}^2 g_{\lambda}(x,y;\mathcal{B}_g)
\right\|^2
\right]
\le
\frac{l_{g,2}^2C_{\lambda}^4}{4}\tau^2
+\frac{\tilde{l}_{g,1} C_{\lambda}^2}{B}.
\end{equation}
\end{lemma}
\begin{proof}
We prove the first bound; the second is identical. The same sample
$\zeta_i$ is used at the two perturbed points. Define
\[
D_i
:=
\frac{
\nabla_x g(x,y+\tau\lambda;\zeta_i)
-
\nabla_x g(x,y-\tau\lambda;\zeta_i)
}{2\tau}.
\]
The samples are independent across $i$, and $\mathbb E[D_i]$ is the
deterministic central finite difference. Hence,
\begin{align*}
\mathbb E\left\|
\frac1B\sum_{i=1}^B\left(D_i-\mathbb E[D_i]\right)
\right\|^2
&=
\frac1B
\mathbb E\left\|D_1-\mathbb E[D_1]\right\|^2\\
&\le
\frac1B\mathbb E\|D_1\|^2\\
&\le
\frac{\widetilde l_{g,1}}{4B\tau^2}
\left\|(x,y+\tau\lambda)-(x,y-\tau\lambda)\right\|^2\\
&=
\frac{\widetilde l_{g,1}\|\lambda\|^2}{B}
\le
\frac{\widetilde l_{g,1}C_\lambda^2}{B}.
\end{align*}
The deterministic central-difference bias is bounded by
$l_{g,2}C_\lambda^2\tau/2$ in
Lemma~\ref{Lemma: Deterministic Finite Difference Estimator}. Since the
sampling error is centered, the squared bias and variance add. This yields
\eqref{Lemma: Stoc FD 1 with g Smoothness}; applying the same calculation to
the $y$-component yields \eqref{Lemma: Stoc FD 2 with g Smoothness}.
\end{proof}

\subsection{Preliminary Lemmas}
\begin{lemma}\label{Lemma: Stochastic Iterates Gap with g Smoothness}
Under Assumptions \ref{Assumption: Bilevel Optimization}, \ref{Assumption: Stochastic First-order Oracle} and \ref{Assumption: Stochastic Smoothness of g}, and the update rule of Algorithm \ref{Stoc-SGHA}, we have
\begin{align*}
\mathbb{E}[\|x^{k+1}-x^k\|^2]
\le
2\gamma_{x,k}^2
\mathbb{E}\left[
\left\|
\nabla_x K(w^k,\hat w^k,\lambda^k)
\right\|^2
\right]
+
\frac{2\gamma_{x,k}^2\sigma_f^2}{B} +
2\gamma_{x,k}^2
\left(
\frac{l_{g,2}^2C_{\lambda}^4}{4}\tau^2
+
\frac{\tilde{l}_{g,1} C_{\lambda}^2}{B}
\right),
\end{align*}
\begin{align*}
\mathbb{E}[\|y^{k+1}-y^k\|^2]
\le
2\gamma_{y,k}^2
\mathbb{E}\left[
\left\|
\nabla_y K(w^k,\hat w^k,\lambda^k)
\right\|^2
\right]
+
\frac{2\gamma_{y,k}^2\sigma_f^2}{B} +
2\gamma_{y,k}^2
\left(
\frac{l_{g,2}^2C_{\lambda}^4}{4}\tau^2
+
\frac{\tilde{l}_{g,1} C_{\lambda}^2}{B}
\right),
\end{align*}
\[
\mathbb{E}[\|\lambda^{k+1}-\lambda^k\|^2]
\le
\gamma_{\lambda,k}^2
\mathbb{E}\left[
\left\|
\nabla_y g(x^k,y^k)-p_\lambda\lambda^k
\right\|^2
\right]
+
\frac{\gamma_{\lambda,k}^2\sigma_g^2}{B}.
\]
In particular, since $\|\lambda^k\|\le C_\lambda$, we also have
\[
\mathbb{E}[\|\lambda^{k+1}-\lambda^k\|^2]
\le
2\gamma_{\lambda,k}^2
\mathbb{E}\left[
\left\|
\nabla_y g(x^k,y^k)
\right\|^2
\right]
+
2\gamma_{\lambda,k}^2p_\lambda^2C_\lambda^2
+
\frac{\gamma_{\lambda,k}^2\sigma_g^2}{B}.
\]
\end{lemma}
\begin{proof}
The proof is identical to that of Lemma \ref{Lemma: Stochastic Iterates Gap}, with $\frac{\sigma_g^2}{B\tau^2}$ replaced by $\frac{\tilde{l}_{g,1} C_{\lambda}^2}{B}$ throughout.
\end{proof}

\subsection{Intermediate Lemmas for Lyapunov Function}
\begin{lemma}\label{Lemma: Stochastic Primal Descent w lambda with g Smoothness}
Under Assumptions \ref{Assumption: Bilevel Optimization}, \ref{Assumption: Stochastic First-order Oracle} and \ref{Assumption: Stochastic Smoothness of g}, suppose $\gamma_{w,k}\le \frac{1}{8(p_w+l_{\mathcal L,1})}$. Then, we have
\begin{align*}
&\mathbb{E}\left[
K(w^k,\hat w^k,\lambda^k)
\right]
-
\mathbb{E}\left[
K(w^{k+1},\hat w^k,\lambda^{k+1})
\right]                                                        \\
&\ge
\frac{\gamma_{w,k}}{4}
\mathbb{E}\left[
\left\|
\nabla_wK(w^k,\hat w^k,\lambda^k)
\right\|^2
\right]                                                     \\
&\quad
-
\left(
\gamma_{w,k}
+
4(p_w+l_{\mathcal L,1})\gamma_{w,k}^2
\right)
\left(
\frac{l_{g,2}^2C_\lambda^4}{4}\tau^2
+
\frac{\tilde{l}_{g,1} C_{\lambda}^2}{B}
\right)      -
\frac{4(p_w+l_{\mathcal L,1})\gamma_{w,k}^2}{B}
\sigma_f^2                                                   \\
&\quad
+
\mathbb{E}\left[
\left\langle
\nabla_y g(x^k,y^k)-p_\lambda\lambda^k,
\lambda^k-\lambda^{k+1}
\right\rangle
\right]
-
\left(\frac{p_\lambda}{2} + \frac{l_{g,1}^2}{2(p_w+l_{\mathcal{L},1})}\right)
\mathbb{E}\left[
\|\lambda^{k+1}-\lambda^k\|^2
\right].
\end{align*}
\end{lemma}
\begin{proof}
The proof is identical to that of Lemma \ref{Lemma: Stochastic Primal Descent w lambda}, with $\frac{\sigma_g^2}{B\tau^2}$ replaced by $\frac{\tilde{l}_{g,1} C_{\lambda}^2}{B}$ throughout.
\end{proof}

\begin{lemma}\label{Lemma: Stochastic Primal Descent - Dual Ascent with g Smoothness}
Suppose the conditions in Lemma \ref{Lemma: Stochastic Primal Descent w lambda with g Smoothness} hold. Set $p_w=2l_{\mathcal L,1}$. Suppose further that
\[
\gamma_{\lambda,k}
\le
\min\left\{
\frac{1}{2\left(2l_{\Psi,1}+p_\lambda + \frac{l_{g,1}^2}{p_w+l_{\mathcal{L},1}}\right)},
\frac{l_{\mathcal L,1}^2}{16l_{g,1}^2}\gamma_{w,k}
\right\}.
\]
Then, we have
\begin{align*}
&\mathbb{E}\left[
K(w^k,\hat w^k,\lambda^k)
\right]
-
\mathbb{E}\left[
K(w^{k+1},\hat w^k,\lambda^{k+1})
\right]
+
2\mathbb{E}\left[
\Psi(\hat w^k,\lambda^{k+1})-\Psi(\hat w^k,\lambda^k)
\right]                                                        \\
&\ge
\frac{\gamma_{w,k}}{8}
\mathbb{E}\left[
\left\|
\nabla_wK(w^k,\hat w^k,\lambda^k)
\right\|^2
\right]
+
\frac{1}{4\gamma_{\lambda,k}}
\mathbb{E}\left[
\left\|
\lambda^{k+1}-\lambda^k
\right\|^2
\right]                                                        \\
&\quad
-
\frac{4(p_w+l_{\mathcal L,1})\gamma_{w,k}^2}{B}
\sigma_f^2
-
\left(
\gamma_{w,k}
+
4(p_w+l_{\mathcal L,1})\gamma_{w,k}^2
\right)
\left(
\frac{l_{g,2}^2C_\lambda^4}{4}\tau^2
+
\frac{\tilde{l}_{g,1} C_{\lambda}^2}{B}
\right)
-
\frac{\gamma_{\lambda,k}\sigma_g^2}{B}.
\end{align*}
\end{lemma}
\begin{proof}
The proof is identical to that of Lemma \ref{Lemma: Stochastic Primal Descent - Dual Ascent}, with $\frac{\sigma_g^2}{B\tau^2}$ replaced by $\frac{\tilde{l}_{g,1} C_{\lambda}^2}{B}$ throughout.
\end{proof}

\begin{proposition}\label{Proposition: Stochastic Lyapunov Function Descent with g Smoothness}
Define the Lyapunov function
\[
V_k
=
K(w^k,\hat w^k,\lambda^k)
-
2\Psi(\hat w^k,\lambda^k)
+
2P(\hat w^k).
\]
Suppose
\[
\gamma_{w,k}
\le
\min\left\{
\frac{1}{8p_w\beta},
\frac{1}{8(p_w+l_{\mathcal L,1})}
\right\},\quad \gamma_{\lambda,k}
\le
\min\left\{\frac{1}{2\left(2l_{\Psi,1}+p_\lambda+ \frac{l_{g,1}^2}{p_w+l_{\mathcal{L},1}}\right)},
\frac{l_{\mathcal L,1}^2}{16l_{g,1}^2}\gamma_{w,k},
\frac{1}{1536p_w\beta\gamma_2^2}
\right\},
\]
and
\[
\beta\le \frac{\sqrt{5}-1}{48\gamma_1}.
\]
Then, we have
\begin{align*}
&\mathbb{E}[V_k]-\mathbb{E}[V_{k+1}]                                      \\
&\ge
\frac{1}{32\gamma_{w,k}}
\mathbb{E}\left[
\left\|
w^k-w^+(\hat w^k,\lambda^k)
\right\|^2
\right]
+
\frac{1}{32\gamma_{\lambda,k}}
\mathbb{E}\left[
\left\|
\lambda^+(\hat w^k)-\lambda^k
\right\|^2
\right]                                                                  \\
&\quad
+
\frac{p_w\beta}{8}
\mathbb{E}\left[
\left\|
w^k-\hat w^k
\right\|^2
\right]
-
48p_w\beta
\mathbb{E}\left[
\left\|
w^\star(\hat w^k)
-
w^\star(\hat w^k,\lambda^+(\hat w^k))
\right\|^2
\right]                                                                  \\
&\quad
-
\left(
\gamma_{w,k}
+
4(p_w+l_{\mathcal L,1})\gamma_{w,k}^2
+
p_w\beta\gamma_{w,k}^2
\right)
\left(
\frac{l_{g,2}^2C_\lambda^4}{4}\tau^2
+
\frac{\tilde{l}_{g,1} C_{\lambda}^2}{B}
\right)                                                                 \\
&\quad
-
\frac{
4(p_w+l_{\mathcal L,1})\gamma_{w,k}^2
\sigma_f^2
}{B} - \frac{
p_w\beta\gamma_{w,k}^2\sigma_f^2
}{B}
-
\frac{9\gamma_{\lambda,k}\sigma_g^2}{8B}.
\end{align*}
\end{proposition}
\begin{proof}
The proof is identical to that of Proposition \ref{Proposition: Stochastic Lyapunov Function Descent}, with $\frac{\sigma_g^2}{B\tau^2}$ replaced by $\frac{\tilde{l}_{g,1} C_{\lambda}^2}{B}$ throughout.
\end{proof}

\subsection{Oracle Complexity under Stochastic Setting with Stochastic Smoothness of \texorpdfstring{\(g\)}{g}}
\begin{proposition}\label{Proposition: Stochastic Lyapunov Function Descent After Error Bound with g Smoothness}
Suppose the conditions in Proposition \ref{Proposition: Stochastic Lyapunov Function Descent with g Smoothness} hold. Suppose further that
\[
\beta
\le
\frac{1}{3072p_w\gamma_{\lambda,k}\gamma_3^2},
\]
where
\[
\gamma_3
:=
\frac{
l_{\Psi,1}+\gamma_{\lambda,k}^{-1}
}{
\sqrt{
2(p_w-l_{\mathcal L,1})p_\lambda
+
\frac{(p_w-l_{\mathcal L,1})^2\mu_g^2}
{(p_w+l_{\mathcal L,1})^2}
}
}.
\]
Then, we have
\begin{align*}
&\mathbb{E}[V_k]-\mathbb{E}[V_{k+1}]                                      \\
&\ge
\frac{1}{32\gamma_{w,k}}
\mathbb{E}\left[
\left\|
w^k-w^+(\hat w^k,\lambda^k)
\right\|^2
\right]
+
\frac{1}{64\gamma_{\lambda,k}}
\mathbb{E}\left[
\left\|
\lambda^+(\hat w^k)-\lambda^k
\right\|^2
\right]
+
\frac{p_w\beta}{8}
\mathbb{E}\left[
\left\|
w^k-\hat w^k
\right\|^2
\right]                                                                  \\
&\quad
-
\left(
\gamma_{w,k}
+
4(p_w+l_{\mathcal L,1})\gamma_{w,k}^2
+
p_w\beta\gamma_{w,k}^2
\right)
\left(
\frac{l_{g,2}^2C_\lambda^4}{4}\tau^2
+
\frac{\tilde{l}_{g,1} C_{\lambda}^2}{B}
\right)                                                                 \nonumber\\
&\quad
-
\frac{
4(p_w+l_{\mathcal L,1})\gamma_{w,k}^2
\sigma_f^2
}{B} - \frac{
p_w\beta\gamma_{w,k}^2\sigma_f^2
}{B}
-
\frac{9\gamma_{\lambda,k}\sigma_g^2}{8B}.
\end{align*}
\end{proposition}
\begin{proof}
The proof is identical to that of Proposition \ref{Proposition: Stochastic Lyapunov Function Descent After Error Bound}, with $\frac{\sigma_g^2}{B\tau^2}$ replaced by $\frac{\tilde{l}_{g,1} C_{\lambda}^2}{B}$ throughout.
\end{proof}

\begin{lemma}
\label{Lemma: Stochastic Proximal Residual Bound with g Smoothness}
Under Assumptions \ref{Assumption: Bilevel Optimization}, \ref{Assumption: Stochastic First-order Oracle} and \ref{Assumption: Stochastic Smoothness of g}, suppose that $\Lambda=\{\lambda:\|\lambda\|\le C_\lambda\}$ with $C_{\lambda} > \frac{l_{f,0}}{\mu_g} $. Set $p_w=2l_{\mathcal L,1}$.
Further choose constant stepsizes
\[
\gamma_{w,k} \leq \min\left\{
\frac{1}{8p_w\beta},
\frac{1}{8(p_w+l_{\mathcal L,1})}
\right\},
\quad
\gamma_{\lambda,k} \le
\min\left\{
\frac{1}{2\left(2l_{\Psi,1}+p_\lambda + \frac{l_{g,1}^2}{p_w+l_{\mathcal{L},1}}\right)},
\frac{l_{\mathcal L,1}^2}{16l_{g,1}^2}\gamma_{w,k},
\frac{1}{1536p_w\beta\gamma_2^2}
\right\}.
\]
Moreover, choose the averaging parameter
\[
\beta
\le
\min\left\{
\frac{\sqrt{5}-1}{48\gamma_1},
\frac{1}{3072p_w\gamma_{\lambda,k}\gamma_3^2}
\right\}.
\]
Define
\[
a
:=
\min\left\{
\frac{1}{32\gamma_{w,k}},
\frac{\bar{\kappa}_y^2}{64\gamma_{\lambda,k}},
\frac{\bar{\kappa}_y^2p_w\beta}{8}
\right\}.
\]
Also define
\[
C_g
:=
\gamma_{w,k}
+
4(p_w+l_{\mathcal L,1})\gamma_{w,k}^2
+
p_w\beta\gamma_{w,k}^2,
\]
and
\[
C_f
:=
4(p_w+l_{\mathcal L,1})\gamma_{w,k}^2 + p_w\beta\gamma_{w,k}^2.
\]
Without loss of generality, take the declared Hessian--Lipschitz upper
bound $l_{g,2}>0$ (it may always be enlarged). Choose
\[
\tau^2
=
\frac{4a}{5C_g l_{g,2}^2C_\lambda^4}
\epsilon_{\mathcal L}^2
\]
and the positive integer batch size
\[
B
=
\max\left\{
1,
\left\lceil
\max\left\{
\frac{5\tilde l_{g,1}C_gC_\lambda^2}{a\epsilon_{\mathcal L}^2},
\frac{5C_f\sigma_f^2}{a\epsilon_{\mathcal L}^2},
\frac{45\gamma_{\lambda,k}\sigma_g^2}{8a\epsilon_{\mathcal L}^2}
\right\}
\right\rceil
\right\},
\]
and run Algorithm~\ref{Stoc-SGHA} for
\[
K
=
\max\left\{
1,
\left\lceil
\frac{5\left(\mathbb E[V_0]-\underline f\right)}
{a\epsilon_{\mathcal L}^2}
\right\rceil
\right\}
\]
iterations. More generally, the same conclusion holds if $\tau^2$ is
chosen no larger than its displayed value and $B$ and $K$ are chosen as
positive integers no smaller than their displayed values. Then there
exists an index $k\in\{0,\ldots,K-1\}$ such that
\[
\max\left\{
\mathbb E\left[
\|w^k-w^+(\hat w^k,\lambda^k)\|^2
\right],
\bar{\kappa}_y^{-2}\mathbb E\left[
\|\lambda^+(\hat w^k)-\lambda^k\|^2
\right],
\bar{\kappa}_y^{-2}\mathbb E\left[
\|w^k-\hat w^k\|^2
\right]
\right\}
\le
\epsilon_{\mathcal L}^2.
\]
\end{lemma}
\begin{proof}
By Proposition \ref{Proposition: Stochastic Lyapunov Function Descent After Error Bound with g Smoothness}, we have
\[
\begin{aligned}
\mathbb E[V_k]-\mathbb E[V_{k+1}]
\ge\;&
\frac{1}{32\gamma_{w,k}}
\mathbb E\left[
\|w^k-w^+(\hat w^k,\lambda^k)\|^2
\right]                                +
\frac{1}{64\gamma_{\lambda,k}}
\mathbb E\left[
\|\lambda^+(\hat w^k)-\lambda^k\|^2
\right]                             +
\frac{p_w\beta}{8}
\mathbb E\left[
\|w^k-\hat w^k\|^2
\right]                                                        \\
&-
C_g
\left(
\frac{l_{g,2}^2C_\lambda^4}{4}\tau^2
+
\frac{\tilde{l}_{g,1} C_{\lambda}^2}{B}
\right)
-
\frac{C_f\sigma_f^2}{B}
-
\frac{9\gamma_{\lambda,k}\sigma_g^2}{8B}.
\end{aligned}
\]
By the definition of $a$,
\[
a
=
\min\left\{
\frac{1}{32\gamma_{w,k}},
\frac{\bar{\kappa}_y^2}{64\gamma_{\lambda,k}},
\frac{\bar{\kappa}_y^2p_w\beta}{8}
\right\},
\]
we obtain
\[
\begin{aligned}
\mathbb E[V_k]-\mathbb E[V_{k+1}]
\ge\;&
a
\max\left\{
\mathbb E\left[
\|w^k-w^+(\hat w^k,\lambda^k)\|^2
\right],
\bar{\kappa}_y^{-2}\mathbb E\left[
\|\lambda^+(\hat w^k)-\lambda^k\|^2
\right],
\bar{\kappa}_y^{-2}\mathbb E\left[
\|w^k-\hat w^k\|^2
\right]
\right\}                                                       \\
&-
C_g
\left(
\frac{l_{g,2}^2C_\lambda^4}{4}\tau^2
+
\frac{\tilde{l}_{g,1} C_{\lambda}^2}{B}
\right)
-
\frac{C_f\sigma_f^2}{B}
-
\frac{9\gamma_{\lambda,k}\sigma_g^2}{8B}.
\end{aligned}
\]
Summing from $k=0$ to $K-1$ gives
\[
\begin{aligned}
\mathbb E[V_0]-\underline{f}
\ge\;&
a
\sum_{k=0}^{K-1}
\max\left\{
\mathbb E\left[
\|w^k-w^+(\hat w^k,\lambda^k)\|^2
\right],
\bar{\kappa}_y^{-2}\mathbb E\left[
\|\lambda^+(\hat w^k)-\lambda^k\|^2
\right],
\bar{\kappa}_y^{-2}\mathbb E\left[
\|w^k-\hat w^k\|^2
\right]
\right\}                                                       \\
&-
K C_g
\left(
\frac{l_{g,2}^2C_\lambda^4}{4}\tau^2
+
\frac{\tilde{l}_{g,1} C_{\lambda}^2}{B}
\right)
-
K\frac{C_f\sigma_f^2}{B}
-
K\frac{9\gamma_{\lambda,k}\sigma_g^2}{8B}.
\end{aligned}
\]
Therefore, there exists $k\in\{0,\ldots,K-1\}$ such that
\[
\begin{aligned}
&\max\left\{
\mathbb E\left[
\|w^k-w^+(\hat w^k,\lambda^k)\|^2
\right],
\bar{\kappa}_y^{-2}\mathbb E\left[
\|\lambda^+(\hat w^k)-\lambda^k\|^2
\right],
\bar{\kappa}_y^{-2}\mathbb E\left[
\|w^k-\hat w^k\|^2
\right]
\right\}                                                       \\
&\le
\frac{\mathbb E[V_0]-\underline{f}}{aK}
+
\frac{C_g}{a}
\left(
\frac{l_{g,2}^2C_\lambda^4}{4}\tau^2
+
\frac{\tilde{l}_{g,1} C_{\lambda}^2}{B}
\right)
+
\frac{C_f\sigma_f^2}{aB}
+
\frac{9\gamma_{\lambda,k}\sigma_g^2}{8aB}.
\end{aligned}
\]
By choosing
\[
K
\ge
\left\lceil
\frac{5\left(\mathbb E[V_0]-\underline{f}\right)}
{a\epsilon_{\mathcal L}^2}
\right\rceil,
\]
we have
\[
\frac{\mathbb E[V_0]-\underline{f}}{aK}
\le
\frac{1}{5}\epsilon_{\mathcal L}^2.
\]
By the choice of $\tau$,
\[
\tau^2
\le
\frac{4a}{5C_g l_{g,2}^2C_\lambda^4}
\epsilon_{\mathcal L}^2,
\]
we have
\[
\frac{C_g}{a}
\frac{l_{g,2}^2C_\lambda^4}{4}\tau^2
\le
\frac{1}{5}\epsilon_{\mathcal L}^2.
\]
By the choice of $B$,
\[
B
\ge
\left\lceil
\max\left\{
\frac{5\tilde{l}_{g,1}C_gC_{\lambda}^2}{a\epsilon_{\mathcal L}^2},
\frac{5C_f\sigma_f^2}{a\epsilon_{\mathcal L}^2},
\frac{45\gamma_{\lambda,k}\sigma_g^2}{8a\epsilon_{\mathcal L}^2}
\right\}
\right\rceil,
\]
we have
\[
\frac{\tilde{l}_{g,1}C_gC_{\lambda}^2}{aB}
\le
\frac{1}{5}\epsilon_{\mathcal L}^2,\quad \frac{C_f\sigma_f^2}{aB}
\le
\frac{1}{5}\epsilon_{\mathcal L}^2,\quad \frac{9\gamma_{\lambda,k}\sigma_g^2}{8aB}
\le
\frac{1}{5}\epsilon_{\mathcal L}^2.
\]
Combining these estimates yields
\[
\max\left\{
\mathbb E\left[
\|w^k-w^+(\hat w^k,\lambda^k)\|^2
\right],
\bar{\kappa}_y^{-2}\mathbb E\left[
\|\lambda^+(\hat w^k)-\lambda^k\|^2
\right],
\bar{\kappa}_y^{-2}\mathbb E\left[
\|w^k-\hat w^k\|^2
\right]
\right\}
\le
\epsilon_{\mathcal L}^2.
\]
This completes the proof.
\end{proof}

\subsection{Oracle Complexity of Stoc-SGHA in High Probability with Stochastic Smoothness of \texorpdfstring{\(g\)}{g}}

\paragraph{Proof of Theorem
\ref{Theorem: Stochastic Sample Complexity HP with g Smoothness}.}
Define
$
\epsilon_{\mathcal L}
:=
\frac{\epsilon\sqrt{\rho}}
{2^{12}(1+\bar L)\bar\kappa_y^3}$ and $
L_0:=8\bar L\bar\kappa_y$.
Then
\[
p_\lambda
=
\frac18
\min\left\{
\mu_g,\frac{\epsilon_{\mathcal L}}{\bar\kappa_y}
\right\},
\qquad
p_\lambda C_\lambda
\le\frac14\epsilon_{\mathcal L},
\qquad
\tau
=
\frac{\epsilon_{\mathcal L}}
{2^{13}\bar\kappa_y}.
\]

The verification of the stepsize and averaging conditions is identical
to that in the proof of
Theorem~\ref{Theorem: Stochastic Sample Complexity HP}.
Indeed, with $L_0$ in place of $l_{\mathcal L,1}$,
\[
p_w=2L_0,\qquad
\gamma_1=2,\qquad
\gamma_2\le\frac{1}{8\bar\kappa_y},\qquad
l_{\Psi,1}\le\frac{\bar L}{4\bar\kappa_y},\qquad
\gamma_3\le\frac{3075}{4}.
\]
Thus, the parameters satisfy the common conditions required in
Lemma~\ref{Lemma: Stochastic Proximal Residual Bound with g Smoothness}.

Define
\[
a
:=
\min\left\{
\frac{1}{32\gamma_w},
\frac{\bar\kappa_y^2}{64\gamma_\lambda},
\frac{\bar\kappa_y^2p_w\beta}{8}
\right\}
=
\frac{\bar L\bar\kappa_y}{2^{29}},
\]
and
\[
C_g
:=
\gamma_w
+
4(p_w+L_0)\gamma_w^2
+
p_w\beta\gamma_w^2,
\qquad
C_f
:=
4(p_w+L_0)\gamma_w^2
+
p_w\beta\gamma_w^2.
\]
The explicit parameter choices give
\[
C_g
=
\frac{22+\beta}{2^{12}\bar L\bar\kappa_y}
\le
\frac{23}{2^{12}\bar L\bar\kappa_y},
\qquad
C_f
=
\frac{6+\beta}{2^{12}\bar L\bar\kappa_y}
\le
\frac{7}{2^{12}\bar L\bar\kappa_y}.
\]

Using
$l_{g,2}\le\bar L$ and $C_\lambda=2\bar\kappa_y$,
\[
\frac{C_gl_{g,2}^2C_\lambda^4}{4a}\tau^2
\le
\frac{23}{128}\epsilon_{\mathcal L}^2
<
\frac15\epsilon_{\mathcal L}^2.
\]
Moreover, the batch size in the theorem satisfies
$
B
\ge
\frac{
2^{26}\bar M_g
}{
\bar L^2\epsilon_{\mathcal L}^2
}$.
Consequently,
\begin{align*}
\frac{5\widetilde l_{g,1}C_gC_\lambda^2}
{a\epsilon_{\mathcal L}^2}
&\le
\frac{
115\cdot2^{19}\widetilde l_{g,1}
}{
\bar L^2\epsilon_{\mathcal L}^2
}
\le B,\\
\frac{5C_f\sigma_f^2}
{a\epsilon_{\mathcal L}^2}
&\le
\frac{
35\cdot2^{17}\sigma_f^2
}{
\bar L^2\bar\kappa_y^2\epsilon_{\mathcal L}^2
}
\le B,\\
\frac{45\gamma_\lambda\sigma_g^2}
{8a\epsilon_{\mathcal L}^2}
&=
\frac{
45\cdot2^{18}\sigma_g^2
}{
\bar L^2\epsilon_{\mathcal L}^2
}
\le B,
\end{align*}
where we used $\bar\kappa_y\ge1$ and the definition of $\bar M_g$.

By Lemma~\ref{Lemma: Stochastic Proximal Residual Bound with g Smoothness},
the integer $K$ specified in the theorem satisfies
\[
K\ge
\left\lceil
\frac{5\Delta_V}{a\epsilon_{\mathcal L}^2}
\right\rceil
=
\left\lceil
\frac{5\cdot2^{53}(1+\bar L)^2\Delta_V}{\bar L}
\bar\kappa_y^5\epsilon^{-2}\rho^{-1}
\right\rceil.
\]
Consequently, there exists an index
$k_\star\in\{0,\ldots,K-1\}$ such that
\begin{align}
\max\biggl\{
\mathbb E\left[
\|w^{k_\star}
-w^+(\hat w^{k_\star},\lambda^{k_\star})\|^2
\right],
\bar\kappa_y^{-2}
\mathbb E\left[
\|\lambda^+(\hat w^{k_\star})-\lambda^{k_\star}\|^2
\right],
\,
\bar\kappa_y^{-2}
\mathbb E\left[
\|w^{k_\star}-\hat w^{k_\star}\|^2
\right]
\biggr\}
\le
\epsilon_{\mathcal L}^2.
\label{eq:g-smooth-hp-residual}
\end{align}

Set
$
\delta
:=
\frac{\sqrt3\,\epsilon_{\mathcal L}}{\sqrt\rho}$.
By Markov's inequality and a union bound, with probability at least
$1-\rho$,
\[
\|w^{k_\star}
-w^+(\hat w^{k_\star},\lambda^{k_\star})\|
\le\delta,\qquad
\|\lambda^+(\hat w^{k_\star})-\lambda^{k_\star}\|
\le\bar\kappa_y\delta,
\qquad
\|w^{k_\star}-\hat w^{k_\star}\|
\le\bar\kappa_y\delta.
\]

We next verify the conditions of
Lemma~\ref{Lemma: Deterministic Stationary Point Bridge}.
Since
\[
\delta
=
\frac{\sqrt3\,\epsilon}
{2^{12}(1+\bar L)\bar\kappa_y^3}
\le
\frac{\epsilon}
{2^{11}(1+\bar L)\bar\kappa_y^3},
\]
the calculations in the deterministic proof give
\[
\left(
\gamma_w^{-1}+\bar\kappa_yp_w
\right)\delta
\le
\frac{\mu_g}{4}
\left(
C_\lambda-\frac{l_{f,0}}{\mu_g}
\right),
\qquad
\bar\kappa_y\delta
\le
\frac14
\left(
C_\lambda-\frac{l_{f,0}}{\mu_g}
\right).
\]
Furthermore,
\[
p_\lambda C_\lambda
\le
\frac14\epsilon_{\mathcal L}
\le
\frac14\delta.
\]
Thus, Lemma~\ref{Lemma: Deterministic Stationary Point Bridge}
applies pathwise. Using
\[
\gamma_w^{-1}+\bar\kappa_yp_w
\le
272\bar L\bar\kappa_y^2
\]
and
\[
\frac{l_{g,1}}{(p_w-L_0)\gamma_w}
+
\bar\kappa_y\gamma_\lambda^{-1}
\le
288\bar L,
\]
we obtain, on the same event,
\[
\|\nabla_w
\mathcal L(w^{k_\star},\lambda^{k_\star})\|
\le
272\bar L\bar\kappa_y^2\delta,\qquad
\|\nabla_\lambda
\mathcal L(w^{k_\star},\lambda^{k_\star})\|
\le
\left(
288\bar L+\frac14
\right)\delta.
\]

As shown in the proof of
Theorem~\ref{Theorem: Deterministic Sample Complexity}, we have
$
\frac{L_{\bar F,y}}{\mu_g}
\le
4\bar\kappa_y^3$.
Therefore, Theorem~\ref{Theorem: Approximation of Hypergradient}
implies
\begin{align*}
\|\nabla F(x^{k_\star})\|
&\le
(1+\bar\kappa_y)
272\bar L\bar\kappa_y^2\delta
+
4\bar\kappa_y^3
\left(
288\bar L+\frac14
\right)\delta\\
&\le
2^{11}(1+\bar L)\bar\kappa_y^3\delta =
\frac{\sqrt3}{2}\epsilon
<
\epsilon.
\end{align*}
Hence,
$
\mathbb P\left(
\|\nabla F(x^{k_\star})\|\le\epsilon
\right)
\ge1-\rho$.
Finally, the total oracle complexity is
\[
KB
=
\mathcal O\left(
\frac{
(1+\Delta_V)\bar M_g(1+\bar L)^4
}{
\bar L^3
}
\bar\kappa_y^{11}\epsilon^{-4}\rho^{-2}
\right)
=
\mathcal O\left(
\bar\kappa_y^{11}\epsilon^{-4}\rho^{-2}
\right).
\]
This completes the proof.

\subsection{Oracle Complexity of Stoc-SGHA in Expectation with Stochastic Smoothness of \texorpdfstring{\(g\)}{g}}

\paragraph{Proof of Theorem
\ref{Theorem: Stochastic Sample Complexity Expectation with g Smoothness}.}
By Lemma~\ref{Lemma: Bounded Lagrangian Gradient} and
$C_\lambda=2\bar\kappa_y$, choose the problem-dependent constant
$\mathcal G$ in the theorem large enough that
$\mathcal G\ge\max\{1,G_w/\bar\kappa_y,G_\lambda\}$.
Define
$
\epsilon_{\mathcal L}
:=
\frac{\epsilon}
{2^{13}(1+\bar L)\mathcal G\bar\kappa_y^3}$ and $
L_0:=8\bar L\bar\kappa_y$.
Then
\[
p_\lambda
=
\frac18
\min\left\{
\mu_g,\frac{\epsilon_{\mathcal L}}{\bar\kappa_y}
\right\},
\qquad
p_\lambda C_\lambda
\le\frac14\epsilon_{\mathcal L},
\qquad
\tau
=
\frac{\epsilon_{\mathcal L}}
{2^{13}\bar\kappa_y}.
\]
Define
\[
a
:=
\min\left\{
\frac{1}{32\gamma_w},
\frac{\bar\kappa_y^2}{64\gamma_\lambda},
\frac{\bar\kappa_y^2p_w\beta}{8}
\right\}
=
\frac{\bar L\bar\kappa_y}{2^{29}}.
\]

The verification of the required conditions of Lemma~\ref{Lemma: Stochastic Proximal Residual Bound with g Smoothness} is identical
to that in the proof of
Theorem~\ref{Theorem: Stochastic Sample Complexity HP with g Smoothness}.
Indeed, with $L_0$ in place of $l_{\mathcal L,1}$,
\[
p_w=2L_0,\qquad
\gamma_1=2,\qquad
\gamma_2\le\frac{1}{8\bar\kappa_y},\qquad
l_{\Psi,1}\le\frac{\bar L}{4\bar\kappa_y},\qquad
\gamma_3\le\frac{3075}{4}.
\]

By Lemma~\ref{Lemma: Stochastic Proximal Residual Bound with g Smoothness},
the integer $K$ specified in the theorem satisfies
\[
K
\ge
\left\lceil
\frac{5\Delta_V}
{a\epsilon_{\mathcal L}^2}
\right\rceil
=
\left\lceil
\frac{
5\cdot2^{55}(1+\bar L)^2\mathcal G^2\Delta_V
}{
\bar L
}
\bar\kappa_y^5\epsilon^{-2}
\right\rceil.
\]
Consequently, there exists
$k_\star\in\{0,\ldots,K-1\}$ such that
\begin{align}
\max\biggl\{
\mathbb E\left[
\|w^{k_\star}
-w^+(\hat w^{k_\star},\lambda^{k_\star})\|^2
\right],
\bar\kappa_y^{-2}
\mathbb E\left[
\|\lambda^+(\hat w^{k_\star})-\lambda^{k_\star}\|^2
\right],
\bar\kappa_y^{-2}
\mathbb E\left[
\|w^{k_\star}-\hat w^{k_\star}\|^2
\right]
\biggr\}
\le
\epsilon_{\mathcal L}^2.
\label{eq:g-smooth-expectation-residual}
\end{align}

We next apply
Lemma~\ref{Lemma: Expected Stationary Point Bridge}.
Its constants satisfy
\[
A_w
=
\gamma_w^{-1}+\bar\kappa_yp_w
\le
272\bar L\bar\kappa_y^2,
\qquad
A_\lambda
=
\frac{l_{g,1}}{(p_w-L_0)\gamma_w}
+
\bar\kappa_y\gamma_\lambda^{-1}
\le
288\bar L.
\]
Since
$
C_\lambda-\frac{l_{f,0}}{\mu_g}
\ge\bar\kappa_y$ and $
\mu_g=\frac{\bar L}{\bar\kappa_y}$,
the radius in
Lemma~\ref{Lemma: Expected Stationary Point Bridge} satisfies
\[
r_0
\ge
\min\left\{
\frac{1}{1088\bar\kappa_y^2},\frac14
\right\}
=
\frac{1}{1088\bar\kappa_y^2},
\qquad
\frac{3}{r_0^2}
<
2^{22}\bar\kappa_y^4.
\]

The bounded-hypergradient constant in
Lemma~\ref{Lemma: Bounded Lagrangian Gradient} satisfies
\[
G_F
\le
2(1+\bar L)\mathcal G\bar\kappa_y.
\]
As shown in the deterministic proof,
$L_{\bar F,y}/\mu_g\le4\bar\kappa_y^3$.
Using $p_\lambda C_\lambda\le\epsilon_{\mathcal L}/4$ and applying
the hypergradient conclusion of
Lemma~\ref{Lemma: Expected Stationary Point Bridge}, we obtain
\begin{align*}
\mathbb E\left[\|\nabla F(x^{k_\star})\|\right]
&\le
2^{11}(1+\bar L)\bar\kappa_y^3\epsilon_{\mathcal L}
+
2^{23}(1+\bar L)\mathcal G
\bar\kappa_y^5\epsilon_{\mathcal L}^2\\
&=
\frac{\epsilon}{4\mathcal G}
+
\frac{\epsilon^2}
{8(1+\bar L)\mathcal G\bar\kappa_y}\\
&\le
\frac38\epsilon
<
\epsilon,
\end{align*}
where the last inequality uses
$0<\epsilon\le1$, $\mathcal G\ge1$, and
$\bar\kappa_y\ge1$.

Finally, the total stochastic first-order oracle complexity is
\[
KB
=
\mathcal O\left(
\frac{
(1+\Delta_V)\bar M_g(1+\bar L)^4\mathcal G^4
}{
\bar L^3
}
\bar\kappa_y^{11}\epsilon^{-4}
\right)
=
\mathcal O\left(
\bar\kappa_y^{11}\epsilon^{-4}
\right).
\]
This completes the proof.

\end{document}